\documentclass[12pt]{article}
\usepackage{amsmath,mathdots}
\usepackage{amssymb}
\usepackage{latexsym}
\usepackage{amsthm}
\usepackage{amsmath,amssymb,amsfonts,amsthm}
\usepackage{mdframed}
\usepackage{graphicx}
\usepackage{stmaryrd}
\usepackage{ulem}
\usepackage{gensymb}
\usepackage{tikz}   
\usepackage{stmaryrd}
\usepackage{multirow}

\usepackage{graphicx}
\usepackage{tikz}
\usetikzlibrary{matrix,arrows}
\usetikzlibrary{positioning}
\usetikzlibrary{fit}
\usetikzlibrary{patterns}

\usepackage[colorlinks,
linkcolor=blue,
anchorcolor=blue,
citecolor=blue
]{hyperref}

\newtheorem{thm}{Theorem}[section]

\newtheorem{coro}[thm]{Corollary}

\newtheorem{remark}[thm]{Remark}

\usepackage{graphicx}
\usepackage{tikz}
\usetikzlibrary{matrix,arrows}
\usetikzlibrary{positioning}
\usetikzlibrary{fit}
\usetikzlibrary{patterns}

\newcommand{\dbrac}[1]{{\llbracket #1 \rrbracket}} 	
\newcommand{\boks}[2]{({#1, #2})}   

\newcommand{\pattern}[4]{										
	\raisebox{0.6ex}{
		\begin{tikzpicture}[scale=0.35, baseline=(current bounding box.center), #1]
		\foreach \x/\y in {#4}		\fill[gray!20] (\x,\y) rectangle +(1,1);
		\draw (0.01,0.01) grid (#2+0.99,#2+0.99);
		\foreach \x/\y in {#3}		\filldraw (\x,\y) circle (6pt);
		\end{tikzpicture}}
}

\definecolor{red}{rgb}{1,0,0}
{}

\begin{document}

\begin{center}
{\large \bf  On (joint) equidistributions of mesh patterns 123 and 132 with symmetric shadings}
\end{center}

\begin{center}
Sergey Kitaev$^{a}$ and  Shuzhen Lv$^{b}$
\\[6pt]

$^{a}$Department of Mathematics and Statistics, University of Strathclyde, \\ 26 Richmond Street, Glasgow G1 1XH, UK\\[6pt]

$^{b}$College of Mathematical Science \& Institute of Mathematics and Interdisciplinary Sciences, Tianjin Normal University, \\ Tianjin  300387, P. R. China\\[6pt]

Email:  $^{a}${\tt sergey.kitaev@cis.strath.ac.uk},
           $^{b}${\tt lvshuzhenlsz@yeah.net}
\end{center}

\noindent\textbf{Abstract.}
A notable problem within permutation patterns that has attracted considerable attention in literature since 1973 is the search for a bijective proof demonstrating that 123-avoiding and 132-avoiding permutations are equinumerous, both counted by the Catalan numbers. Despite this equivalence, the distributions of occurrences of the patterns 123 and 132 are distinct. When considering 123 and 132 as mesh patterns and selectively shading boxes, similar scenarios arise, even when avoidance is defined by the Bell numbers or other sequences, rather than the Catalan numbers.

However, computer experiments suggest that mesh patterns 123 and 132 may indeed be equidistributed. Furthermore, by considering symmetric shadings relative to the anti-diagonal, a maximum of 93 such equidistributed pairs can potentially exist. This paper establishes 75 such equidistributions, leaving the justification of the remaining cases as open problems. As a by-product, we also prove 36 relevant non-symmetric equidistributions. All our proofs are bijective and involve swapping occurrences of the patterns in question, thereby demonstrating their joint equidistribution. Our findings are a continuation of the systematic study of distributions of short-length mesh patterns initiated by Kitaev and Zhang in 2019. \\




\section{Introduction}\label{intro}
A permutation of length $n$, or $n$-permutation, is a rearrangement of the set $[n]:=\{1, 2, \ldots, n\}$.  Denote by  $S_n$  the set of permutations of $[n]$. For $\pi \in S_n$, let $\pi^r=\pi_n\pi_{n-1}\cdots\pi_1$ and $\pi^c= (n + 1 - \pi_1)(n + 1 - \pi_2)\cdots (n + 1-\pi_n)$
denote the \textit{reverse} and \textit{complement} of $\pi$, respectively. Then $\pi^{rc}=(n + 1-\pi_n)(n + 1-\pi_{n-1})\cdots (n + 1-\pi_1)$. A permutation $\pi_1\pi_2\cdots\pi_n\in S_n$ avoids a \textit{pattern} $p=p_1p_2\cdots p_k\in S_k$  if there is no subsequence $\pi_{i_1}\pi_{i_2}\cdots\pi_{i_k}$ such that $\pi_{i_j}<\pi_{i_m}$ if and only if $p_j<p_m$. For example, the permutation $32145$ avoids the pattern $132$. $S_n(p)$ denotes the number of permutations of length $n$ avoiding a pattern $p$. Patterns $p_1$ and $p_2$ are \textit{Wilf-equivalent}, denoted $p_1\sim p_2$, if $|S_n(p_1)|=|S_n(p_2)|$ for all $n\geq0$. Also, $p_1$ and $p_2$ are \textit{equidistributed}, denoted  $p_1\sim_d p_2$, if the number of permutations of length $n$ with $k$ occurrences of $p_1$ is equal to that with $k$ occurrences of $p_1$ for any $k,n\geq 0$.
 
Patterns in permutations have attracted much attention in the literature (see~\cite{Kit} and references therein), and this area of research continues to grow rapidly.  A significant problem in the field that has garnered considerable attention since 1973 is finding a bijective proof of the equivalence $123 \sim 132$, specifically that $|S_n(123)|=|S_n(132)|=\frac{1}{n+1}{2n\choose n}$, the $n$-th \textit{Catalan number}, for $n\geq 1$ (for example, see \cite{EP04,Kr01,MDD,R02,R88,SS85,W95}).  Claesson and Kitaev \cite{ClaKit2008} discovered that some of the published bijections can be easily derived from others via ``trivial'' bijections. They provide a comprehensive survey and systematic analysis of these bijections. Subsequent work in this area includes \cite{BloomSar, Sar2011}. Note that the patterns 123 and 132 are not equidistributed; for instance, there are 6 (resp., 5) 4-permutations with one occurrence of the pattern 123 (resp., 132).

The notion of a \textit{mesh pattern}, generalizing several classes of patterns, was introduced by Br\"and\'en and Claesson \cite{BrCl} to provide explicit expansions for certain permutation statistics as, possibly infinite, linear combinations of (classical) permutation patterns. 
A pair $(\tau,R)$, where $\tau$ is a permutation of length $k$ and $R$ is a subset of $\dbrac{0,k} \times \dbrac{0,k}$, where
$\dbrac{0,k}$ denotes the interval of the integers from $0$ to $k$, is a
mesh pattern of length $k$.
Let $\boks{i}{j}$ denote the box whose corners have coordinates $(i,j), (i,j+1),
(i+1,j+1)$, and $(i+1,j)$. Let the horizontal lines represent the values,  and the vertical lines denote the positions in the pattern. Mesh patterns can be drawn by shading the boxes in $R$. The picture 
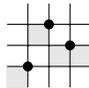
\[
\pattern{scale=0.8}{3}{1/1,2/3,3/2}{0/0,1/2, 2/1,3/1}
\]
represents the mesh pattern with $\tau=132$ and $R = \{\boks{0}{0},\boks{1}{2},\boks{2}{1},\boks{3}{1}\}$. Note that occurrences of patterns  $\pattern{scale=0.5}{2}{1/1,2/2}{1/0, 1/1,1/2}$ and $\pattern{scale=0.5}{2}{1/2,2/1}{1/0, 1/1,1/2}$ are known as \textit{ascents} and \textit{descents}. Many papers are dedicated to the study of mesh patterns and their generalizations; for example, see \cite{AKV, BG, Borie, HanZeng, KL, T2, TR}. However, the first systematic study of avoidance of mesh patterns was not conducted until \cite{Hilmarsson2015Wilf}, and the study of their distribution not until \cite{KZ, KZZ}.

In relation to our paper, Kitaev and Liese~\cite{KL} have demonstrated Wilf-equivalence of the patterns $\pattern{scale=0.5}{3}{1/1,2/2,3/3}{0/1, 0/2,1/0,2/0}$ and  $\pattern{scale=0.5}{3}{1/1,2/3,3/2}{0/1, 0/2,1/0,2/0}$. These patterns are not equidistributed; for example, there are 4 (resp., 3) 4-permutations containing exactly one occurrence of the former (resp., latter) pattern. However, as proved in Theorem~\ref{thm-L-shape}, the patterns $\pattern{scale=0.5}{3}{1/1,2/2,3/3}{0/0,0/1, 0/2,1/0,2/0}$ and  $\pattern{scale=0.5}{3}{1/1,2/3,3/2}{0/0,0/1, 0/2,1/0,2/0}$ are not only Wilf-equivalent but also equidistributed.  On another note, Claesson~\cite{Claesson} proved that $\pattern{scale=0.5}{3}{1/1,2/2,3/3}{2/0, 2/1,2/2,2/3}$ and  $\pattern{scale=0.5}{3}{1/1,2/3,3/2}{2/0, 2/1,2/2,2/3}$ are Wilf-equivalent and are both counted by the \textit{Bell numbers} $1, 1, 2, 5, 15, 52, 203, 877, 4140,\ldots$. These patterns are not equidistributed; for example, there are 7 (resp., 6) 4-permutations containing exactly one occurrence of the former (resp., latter) pattern. For another relevant fact, as discussed  by Avgustinovich at el.~\cite{AKV}, the patterns $\pattern{scale=0.5}{3}{1/1,2/2,3/3}{1/1, 1/2,2/1,2/2}$ and  $\pattern{scale=0.5}{3}{1/1,2/3,3/2}{1/1, 1/2,2/1,2/2}$ are not Wilf-equivalent. However, as proved in Theorem~\ref{2x2-thm}, the patterns $\pattern{scale=0.5}{3}{1/1,2/2,3/3}{2/2, 2/3,3/2,3/3}$ and  $\pattern{scale=0.5}{3}{1/1,2/3,3/2}{2/2, 2/3,3/2,3/3}$ are equidistributed. Interestingly, in some cases, four patterns can be equidistributed. Three (resp., two and two) examples of such situations are presented in Table~\ref{tab-1} (resp., Tables~\ref{tab-3} and~\ref{tab-6}).
  
Note that the patterns 123 and 132 remain invariant under the usual group-theoretic inverse, reflecting the elements of a pattern across the anti-diagonal. This motivated us to restrict our attention in systematic studies to shadings symmetric with respect to the anti-diagonal. Computer experiments, for $n\leq 9$, suggest that the mesh patterns 123 and 132 with symmetric shadings may be potentially equidistributed in a maximum of 93 cases. Interestingly, the patterns  $\pattern{scale=0.5}{3}{1/1,2/2,3/3}{0/0, 0/2,0/3,2/0,3/0}$  and $\pattern{scale=0.5}{3}{1/1,2/3,3/2}{0/0, 0/2,0/3,2/0,3/0}$ (resp., $\pattern{scale=0.5}{3}{1/1,2/2,3/3}{0/0, 0/2,0/3,2/0,2/2,3/0}$  and $\pattern{scale=0.5}{3}{1/1,2/3,3/2}{0/0, 0/2,0/3,2/0,2/2,3/0}$) are equidistributed for $n\leq 7$, but not for $n=8$. 

In this paper we will establish 75 out of the potential 93 equidistributions (presented in Tables~\ref{tab-1}--\ref{tab-6}), leaving the justification of the remaining cases as open problems (see Table~\ref{unsolved-table}). As a by-product, we also note 36 relevant non-symmetric equidistributions (see Table~\ref{non-sym-table}). All our proofs are bijective. In particular, the results in Section~\ref{L-mesh-sec} (resp., Section~\ref{cor-box-sec}) rely on the bijection in the proof of Theorem~\ref{thm-L-shape} (resp., Theorem~\ref{9-box-rs}). Our findings are a continuation of the systematic exploration of distributions of short-length mesh patterns initiated by Kitaev and Zhang in 2019.

In fact, our bijective proofs of $p_1\sim_d p_2$ interchange occurrences of $p_1$ and $p_2$ in a permutation, and thus prove a stronger claims on \textit{joint equidistribution} $(p_1,p_2)\sim_d(p_2,p_1)$. This means that the number of permutations of length $n$ with $k$ (resp., $\ell$) occurrences of $p_1$ (resp., $p_2$) is equal to that with $\ell$ (resp., $k$) occurrences of $p_1$ (resp., $p_2$).  Therefore, every theorem in this paper stating that $p_1\sim_d p_2$, can be replaced by the more general result $(p_1,p_2)\sim_d(p_2,p_1)$. However, for brevity, we do not state the joint equidistribution results. Finally, note that in two cases (see Section~\ref{reduction-sec}) we could have used known equidistribution results for mesh patterns of length 2 to prove the equidistribution. However, we needed to derive alternative proofs to establish joint equidistribution. 

For a permutation $\pi=\pi_1\pi_2\cdots \pi_n$, $\pi_i$ is a  \textit{left-to-right minimum} (resp., \textit{right-to-left maximum}) if $\pi_i<\pi_j$ (resp., $\pi_i>\pi_j$) for all $j< i$ (resp., $j>i$). In particular, $\pi_1$ is  a left-to-right minimum. For example, in the permutation $426153$ the elements 4, 2, and 1 are left-to-right minima. Throughout the paper, $p_1$ and $p_2$ refer to the mesh patterns 123 and 132, respectively, with identical shading in question. 

Our paper is organized as following. In Section~\ref{trivial-sec} (resp., \ref{L-mesh-sec}, \ref{cor-box-sec}) we explain 22 (resp., 23, 30) equidistributions. The cases unsolved by us are presented, along with concluding remarks, in Section~\ref{concluding-sec}. Note that, while we provide thorough proofs and examples for our key results, as well as counterexamples for cases where our methods do not apply, many of the other justifications are relatively sketchy. This is because  it is not feasible to provide detailed proofs for all 111 equidistributions discussed.

\section{The first few cases}\label{trivial-sec}

In this section, we justify more straightforward equidistributions. From now on, for brevity, we usually omit ``and'' in ``$p_1$ and $p_2$'' when referring to a pair of mesh patterns $p_1=123$ and $p_2=132$ with the same shading.

\subsection{Cases justifiable by simple direct arguments}\label{direct-arg-sec}

Occurrences of the patterns in this subsection can be easily controlled in permutations. Therefore, we provide just brief justifications for the equidistributions  of the patterns listed in Table~\ref{tab-1} as bullet points. Additionally, a map that establishes the equidistribution of the patterns in question is denoted as $f$.
  \\[-3mm]

\begin{table}[t]
 	{
 		\renewcommand{\arraystretch}{1.5}
 \begin{center} 
 		\begin{tabular}{|c|c||c|c||c|c||c|c|}
 			\hline
 	\footnotesize{nr.}	& {\footnotesize patterns}  & \footnotesize{nr.}	& {\footnotesize patterns}  & \footnotesize{nr.}	& {\footnotesize patterns}  & \footnotesize{nr.}	& {\footnotesize patterns}  \\
 			\hline
 			\hline	
\footnotesize{1} & $\pattern{scale = 0.5}{3}{1/1,2/2,3/3}{0/0,0/1,0/2,0/3,1/0,1/1,1/2,1/3,2/0,2/1,2/2,2/3,3/0,3/1,3/2,3/3}\pattern{scale = 0.5}{3}{1/1,2/3,3/2}{0/0,0/1,0/2,0/3,1/0,1/1,1/2,1/3,2/0,2/1,2/2,2/3,3/0,3/1,3/2,3/3}$  & 
\footnotesize{2} & $\pattern{scale = 0.5}{3}{1/1,2/2,3/3}{0/0,0/1,0/2,0/3,1/0,1/1,1/2,1/3,2/0,2/1,2/2,2/3,3/0,3/1,3/2}\pattern{scale = 0.5}{3}{1/1,2/3,3/2}{0/0,0/1,0/2,0/3,1/0,1/1,1/2,1/3,2/0,2/1,2/2,2/3,3/0,3/1,3/2}$ & 
\footnotesize{3} & $\pattern{scale = 0.5}{3}{1/1,2/2,3/3}{0/1,0/2,0/3,1/0,1/1,1/2,1/3,2/0,2/1,2/2,2/3,3/0,3/1,3/2,3/3}\pattern{scale = 0.5}{3}{1/1,2/3,3/2}{0/1,0/2,0/3,1/0,1/1,1/2,1/3,2/0,2/1,2/2,2/3,3/0,3/1,3/2,3/3}$ &
\footnotesize{4} & $\pattern{scale = 0.5}{3}{1/1,2/2,3/3}{0/0,0/1,0/2,0/3,1/0,1/1,1/2,1/3,2/0,2/1,2/3,3/0,3/1,3/2,3/3}\pattern{scale = 0.5}{3}{1/1,2/3,3/2} {0/0,0/1,0/2,0/3,1/0,1/1,1/2,1/3,2/0,2/1,2/3,3/0,3/1,3/2,3/3}$  \\[5pt]
\hline
\footnotesize{5} & $\pattern{scale = 0.5}{3}{1/1,2/2,3/3}{0/0,0/1,0/2,0/3,1/0,1/2,1/3,2/0,2/1,2/2,2/3,3/0,3/1,3/2,3/3}\pattern{scale = 0.5}{3}{1/1,2/3,3/2}{0/0,0/1,0/2,0/3,1/0,1/2,1/3,2/0,2/1,2/2,2/3,3/0,3/1,3/2,3/3}$ & 
\footnotesize{6} & $\pattern{scale = 0.5}{3}{1/1,2/2,3/3}{0/0,0/1,0/2,1/0,1/1,1/2,1/3,2/0,2/1,2/2,2/3,3/1,3/2,3/3}\pattern{scale = 0.5}{3}{1/1,2/3,3/2}{0/0,0/1,0/2,1/0,1/1,1/2,1/3,2/0,2/1,2/2,2/3,3/1,3/2,3/3}$   &
\footnotesize{7} & $\pattern{scale = 0.5}{3}{1/1,2/2,3/3}{0/0,0/1,0/2,1/0,1/1,1/2,1/3,2/0,2/1,2/2,2/3,3/1,3/2}\pattern{scale = 0.5}{3}{1/1,2/3,3/2}{0/0,0/1,0/2,1/0,1/1,1/2,1/3,2/0,2/1,2/2,2/3,3/1,3/2}$ & 
\footnotesize{8} & $\pattern{scale = 0.5}{3}{1/1,2/2,3/3}{0/1,0/2,1/0,1/1,1/2,1/3,2/0,2/1,2/2,2/3,3/1,3/2,3/3}\pattern{scale = 0.5}{3}{1/1,2/3,3/2}{0/1,0/2,1/0,1/1,1/2,1/3,2/0,2/1,2/2,2/3,3/1,3/2,3/3}$  \\[5pt]
\hline
\footnotesize{9} & $\pattern{scale = 0.5}{3}{1/1,2/2,3/3}{0/2,0/3,1/2,1/3,2/0,2/1,2/2,2/3,3/0,3/1,3/2,3/3}\pattern{scale = 0.5}{3}{1/1,2/3,3/2}{0/2,0/3,1/2,1/3,2/0,2/1,2/2,2/3,3/0,3/1,3/2,3/3}$ &
\footnotesize{10} & $\pattern{scale = 0.5}{3}{1/1,2/2,3/3}{0/0,0/2,0/3,1/2,1/3,2/0,2/1,2/2,2/3,3/0,3/1,3/2,3/3}\pattern{scale = 0.5}{3}{1/1,2/3,3/2}{0/0,0/2,0/3,1/2,1/3,2/0,2/1,2/2,2/3,3/0,3/1,3/2,3/3}$ &
\footnotesize{11} & $\pattern{scale = 0.5}{3}{1/1,2/2,3/3}{0/2,0/3,1/1,1/2,1/3,2/0,2/1,2/2,2/3,3/0,3/1,3/2,3/3}\pattern{scale = 0.5}{3}{1/1,2/3,3/2}{0/2,0/3,1/1,1/2,1/3,2/0,2/1,2/2,2/3,3/0,3/1,3/2,3/3}$	  & 
\footnotesize{12} &  $\pattern{scale = 0.5}{3}{1/1,2/2,3/3}{0/0,0/1,1/0,1/1,1/2,1/3, 2/1,2/2,2/3,3/1,3/2}\pattern{scale = 0.5}{3}{1/1,2/3,3/2}{0/0,0/1,1/0,1/1,1/2,1/3, 2/1,2/2,2/3,3/1,3/2}$ \\[5pt]
			\hline
	\end{tabular}
\end{center} 
}
\vspace{-0.5cm}
 	\caption{Easy explainable equidistributions; pairs 2 and 3, 4 and 5, and 7 and 8 have the same distributions by applying the complement and reverse operations to the monotone pattern.}\label{tab-1}
\end{table}

\noindent 
{\footnotesize $\bullet$} For pair 1, only permutation 123 (resp., 132) contains one occurrence of $p_1$ (resp., $p_2$); other permutations avoid the patterns. Hence, $p_1\sim_d p_2$.  \\[-3mm]

\noindent 
{\footnotesize $\bullet$} For pair 2, only $n$-permutations $123\pi_4\ldots\pi_n$  (resp., $132\pi_4\ldots\pi_n$), for $n\geq 3$, contain one occurrence of $p_1$ (resp., $p_2$); other permutations avoid the patterns. Letting $f(123\pi_4\ldots\pi_n)=132\pi_4\ldots\pi_n$ and $f(132\pi_4\ldots\pi_n)=123\pi_4\ldots\pi_n$ and $f(\pi)=\pi$ for any other permutation $\pi$, we obtain a bijective map proving $p_1\sim_d p_2$.   \\[-3mm]

\noindent 
{\footnotesize $\bullet$} Pair 3 is similar to pair 2 with  $123\pi_4\ldots\pi_n$  (resp., $132\pi_4\ldots\pi_n$) replaced by $\pi_1\ldots\pi_{n-3}(n-2)(n-1)n$ (resp., $\pi_1\ldots\pi_{n-3}(n-2)n(n-1)$). We omit the details.  \\[-3mm]

\noindent 
{\footnotesize $\bullet$}  For pair 4,  only $n$-permutations $12\pi_3\ldots\pi_{n-1}n$  (resp., $1n\pi_3\ldots\pi_{n-1}2$), for $n\geq 3$, contain one occurrence of $p_1$ (resp., $p_2$); other permutations avoid the patterns. Letting $f(12\pi_3\ldots\pi_{n-1}n)=1n\pi_3\ldots\pi_{n-1}2$ and $f(1n\pi_3\ldots\pi_{n-1}2)=12\pi_3\ldots\pi_{n-1}n$ and $f(\pi)=\pi$ for any other permutation $\pi$, we obtain a bijective map proving $p_1\sim_d p_2$.  \\[-3mm]

\noindent 
{\footnotesize $\bullet$} Pair 5 is similar to pair 4 with  $12\pi_3\ldots\pi_{n-1}n$  (resp., $1n\pi_3\ldots\pi_{n-1}2$) replaced by $1\pi_2\ldots\pi_{n-2}(n-1)n$ (resp., $1\pi_2\ldots\pi_{n-2}n(n-1)$). We omit the details.  \\[-3mm]

\noindent 
{\footnotesize $\bullet$} For pairs 6, 7 and 8,   each occurrence of $p_1$ (resp., $p_2$) in a permutation $\pi$ is as consecutive elements $\pi_i\pi_{i+1}\pi_{i+2}$ equal to $a(a+1)(a+2)$ (resp., $a(a+2)(a+1)$) for some $a$, where $\pi_i$ is a left-to-right minimum for pairs 6 and 7 and for pairs 6 and 8 and $p_1$, $\pi_{i+2}$ is a right-to-left maximum. Note that two different occurrences of these patterns cannot overlap (otherwise, some elements in an occurrence would be in a shaded box). But then, a bijective map $f$ is described by traversing from left to right in $\pi$ and replacing occurrences $\pi_i\pi_{i+1}\pi_{i+2}$ of $p_1$ or $p_2$ by $\pi_i\pi_{i+2}\pi_{i+1}$. This proves that  $p_1\sim_d p_2$.  \\[-3mm]

\noindent 
{\footnotesize $\bullet$} For pairs 9,  10 and 11, all occurrences of $p_1$ (resp., $p_2$) are $\pi_i\pi_{n-1}\pi_n=\pi_i(n-1)n$ (resp., $\pi_i\pi_{n-1}\pi_n=\pi_in(n-1)$) for some $i$, where for pair 10 (resp., 11), $\pi_i$ is a left-to-right minimum (resp., a right-to-left maximimum in the permutation obtained by removing $n-1$ and $n$). Letting $f(\pi_1\ldots\pi_{n-2}(n-1)n)=\pi_1\ldots\pi_{n-2}n(n-1)$ and $f(\pi_1\ldots\pi_{n-2}n(n-1))=\pi_1\ldots\pi_{n-2}(n-1)n$ and $f(\pi)=\pi$ for any other permutation $\pi$, we obtain a bijective map proving $p_1\sim_d p_2$. 

We conclude this subsection by proving an equidistribution that is not straightforward but still relatively easy to justify.

\begin{thm}\label{oth-1}
We have $\pattern{scale = 0.5}{3}{1/1,2/2,3/3}{0/0,0/1,1/0,1/1,1/2,1/3, 2/1,2/2,2/3,3/1,3/2}\sim_d\hspace{-0.15cm} \pattern{scale = 0.5}{3}{1/1,2/3,3/2}{0/0,0/1,1/0,1/1,1/2,1/3, 2/1,2/2,2/3,3/1,3/2}$.
\end{thm}

\begin{proof}
Note that any occurrence of the patterns begins with a left-to-right minimum, and any left-to-right minimum can be involved in at most one occurrence of the patterns. Let $\pi=\pi_1\ldots\pi_n \in S_n$. If $\pi_i\pi_j\pi_k$ is an occurrence of $\pattern{scale = 0.5}{3}{1/1,2/2,3/3}{0/0,0/1,1/0,1/1,1/2,1/3, 2/1,2/2,2/3,3/1,3/2}$ (resp., $\pattern{scale = 0.5}{3}{1/1,2/3,3/2}{0/0,0/1,1/0,1/1,1/2,1/3, 2/1,2/2,2/3,3/1,3/2}$), then $j=i+1, \pi_j=\pi_{i}+1$ (resp., $\pi_k=\pi_{i}+1$), and $\pi_j$ and $\pi_k$ are the leftmost and the smallest elements in the North-East area relative to $\pi_i$. Suppose that $x_1>\ldots>x_k$ is the sequence of left-to-right minima in $\pi$.

Replace the occurrence $\pi_i\pi_{i+1}\pi_k$ of $\pattern{scale = 0.5}{3}{1/1,2/2,3/3}{0/0,0/1,1/0,1/1,1/2,1/3, 2/1,2/2,2/3,3/1,3/2}$ (resp., $\pattern{scale = 0.5}{3}{1/1,2/3,3/2}{0/0,0/1,1/0,1/1,1/2,1/3, 2/1,2/2,2/3,3/1,3/2}$), if any, starting at $x_1$, by $\pi_i\pi_k\pi_{i+1}$, which is an occurrence of  $\pattern{scale = 0.5}{3}{1/1,2/3,3/2}{0/0,0/1,1/0,1/1,1/2,1/3, 2/1,2/2,2/3,3/1,3/2}$ (resp., $\pattern{scale = 0.5}{3}{1/1,2/2,3/3}{0/0,0/1,1/0,1/1,1/2,1/3, 2/1,2/2,2/3,3/1,3/2}$). Then do the same for $x_2$, $x_3$, etc. Note that even though an element can be present in more than one occurrence of the patterns, after conducting the swap for $x_i$, no new occurrences of the patterns in $\pi$ can be introduced and no occurrences can be lost (the elements $\pi_{i+1}$ and $\pi_k$ are indistinguishable in $\pi$ for the elements $<x_i$). The procedure described by us is clearly reversible, and it establishes equidistribution of the patterns.

For example, for the permutation $\pi=(10)785942613(11)$ in Figure~\ref{fig-1-1}, the order of occurrences of the two patterns from left to right are $\pi_2\pi_3\pi_5=789, \pi_4\pi_5\pi_8=596, \pi_7\pi_8\pi_{10}=263$, we do the procedure as follows: $\pi=(10)7\uline{8}5\uline{9}42613(11)\xrightarrow{\text{swap 8 \& 9}}(10)795\uline{8}42\uline{6}13(11)\xrightarrow{\text{swap 8 \& 6}}$$(10)795642\uline{8}1\uline{3}(11)$ $\xrightarrow{\text{swap 8 \& 3}}(10)795642318(11)=\pi'$.
Note that the occurrence 789 of $\pattern{scale = 0.5}{3}{1/1,2/2,3/3}{0/0,0/1,1/0,1/1,1/2,1/3, 2/1,2/2,2/3,3/1,3/2}$ in $\pi$ is replaced by the occurrence 798 of $\pattern{scale = 0.5}{3}{1/1,2/3,3/2}{0/0,0/1,1/0,1/1,1/2,1/3, 2/1,2/2,2/3,3/1,3/2}$ in $\pi'$, while the occurrences 596 and 263 of $\pattern{scale = 0.5}{3}{1/1,2/3,3/2}{0/0,0/1,1/0,1/1,1/2,1/3, 2/1,2/2,2/3,3/1,3/2}$ in $\pi$ are replaced by the occurrences 568 and 238 of $\pattern{scale = 0.5}{3}{1/1,2/2,3/3}{0/0,0/1,1/0,1/1,1/2,1/3, 2/1,2/2,2/3,3/1,3/2}$ in $\pi'$. \end{proof}

\begin{figure}
	\begin{center}
		
		\begin{tabular}{ccc}
			
			\begin{tikzpicture}[scale=0.25] 
			\tikzset{    
				grid/.style={      
					draw,      
					step=1cm,      
					gray!100,     
					very thin,      
				}, 
				cell/.style={    
					draw,    
					anchor=center,  
					text centered,    
				},  
				graycell/.style={ 
					fill=gray!40,   
					draw=none,   
					minimum width=1cm, 
					minimum height=1cm,   
					anchor=south west,   
				}
			}  
			
			\draw[grid] (0,0) grid (10,10);  
			%
			\filldraw[black] (0,9) circle (8pt); 
			\draw[thick] (1,6) circle (8pt); \draw[thick] (1,6) circle (2pt); 
			\draw[thick] (2,7) circle (8pt);  \filldraw[white] (2,7) circle (6pt);
			\draw[thick] (3,4) circle (8pt);  \draw[thick] (3,4) circle (2pt);
			\draw[thick] (4,8) circle (8pt);  \filldraw[white] (4,8) circle (6pt);
			\draw[thick] (6,1) circle (8pt);
			\draw[thick] (6,1) circle (2pt);
			\draw[thick] (7,5) circle (8pt);  \filldraw[white] (7,5) circle (6pt);
			\filldraw[black] (8,0) circle (8pt);  
			\draw[thick] (9,2) circle (8pt); \filldraw[white] (9,2) circle (6pt); 
			\filldraw[black] (10,10) circle (8pt);  
			\filldraw[black] (5,3) circle (8pt);  
			
			\end{tikzpicture} 
			
			& \ & 
			\begin{tikzpicture}[scale=0.25] 
			\tikzset{    
				grid/.style={      
					draw,      
					step=1cm,      
					gray!100,     
					very thin,      
				}, 
				cell/.style={    
					draw,    
					anchor=center,  
					text centered,    
				},  
				graycell/.style={ 
					fill=gray!40,   
					draw=none,   
					minimum width=1cm, 
					minimum height=1cm,   
					anchor=south west,   
				}
			}  
			
			\draw[grid] (0,0) grid (10,10);  
			%
			\filldraw[black] (0,9) circle (8pt); 
			\draw[thick] (1,6) circle (8pt); \draw[thick] (1,6) circle (2pt); 
			\draw[thick] (2,8) circle (8pt);  \filldraw[white] (2,8) circle (6pt);
			\draw[thick] (3,4) circle (8pt);  \draw[thick] (3,4) circle (2pt);
			\draw[thick] (4,5) circle (8pt);  \filldraw[white] (4,5) circle (6pt);
			\draw[thick] (6,1) circle (8pt);
			\draw[thick] (6,1) circle (2pt);
			\draw[thick] (7,2) circle (8pt);  \filldraw[white] (7,2) circle (6pt);
			\filldraw[black] (8,0) circle (8pt);  
			\draw[thick] (9,7) circle (8pt); \filldraw[white] (9,7) circle (6pt); 
			\filldraw[black] (10,10) circle (8pt);  
			\filldraw[black] (5,3) circle (8pt);  
			
						\end{tikzpicture}
		\end{tabular}
	\end{center}
	
	\vspace{-0.5cm}
	
	\caption{Permutations $\pi=(10)785942613(11)$  (to the left) and $\pi'=(10)795642318(11)$ (to the right) illustrating the proof of Corollary~\ref{oth-1}. The circled dots represent the first elements of occurrences (certain left-to-right minima in $\pi$), and the white dots represent the second and third elements of occurrences of $p_1$ or $p_2$.}\label{fig-1-1}
\end{figure}
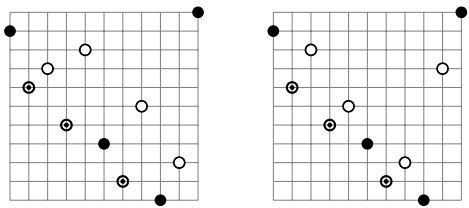

\subsection{Cases justifiable by the complement operation}\label{compl-sec}

Occurrences of any pattern in Table~\ref{tab-2} in a permutation $\pi$ begin with $\pi_1=1$, and removing 1 from the pattern results in a symmetrically shaded pattern with respect to a horizontal line, allowing us to apply the complement operation. Specifically, letting $f(\pi)=\pi$ for any $\pi$ with $\pi_1\neq 1$, and $f(1\pi_2\ldots\pi_n)=1(\pi_2\ldots\pi_n)^c=1\pi^c_2\ldots\pi^c_n$, where the complement is taken on the set $\{2,\ldots,n\}$ (for example, $(2435)^c=5342$) gives a bijective map proving equidistribution of any pair in  Table~\ref{tab-2}, because $1\pi_i\pi_j$ is an occurrence of $p_1$ if and only if $1\pi^c_i\pi^c_j$ is an occurrence of $p_2$.

\begin{table}[t]
 	{
 		\renewcommand{\arraystretch}{1.5}
 \begin{center} 
 		\begin{tabular}{|c|c||c|c||c|c|}
 			\hline
 			
 	\footnotesize{nr.}	& {\footnotesize patterns}  & \footnotesize{nr.}	& {\footnotesize patterns}  & \footnotesize{nr.}	& {\footnotesize patterns}  \\ 
 			\hline
 			\hline	
\footnotesize{13}  & $\pattern{scale = 0.5}{3}{1/1,2/2,3/3}{0/0,0/1,0/2,0/3,1/0,2/0,3/0}\pattern{scale = 0.5}{3}{1/1,2/3,3/2}{0/0,0/1,0/2,0/3,1/0,2/0,3/0}$	  & 
\footnotesize{14}  &  $\pattern{scale = 0.5}{3}{1/1,2/2,3/3}{0/0,0/1,0/2,0/3,1/0,2/0,2/2,3/0}\pattern{scale = 0.5}{3}{1/1,2/3,3/2}{0/0,0/1,0/2,0/3,1/0,2/0,2/2,3/0}$	& 
\footnotesize{15}  &  $\pattern{scale = 0.5}{3}{1/1,2/2,3/3}{0/0,0/1,0/2,0/3,1/0,1/1,1/3,2/0,3/0,3/1,3/3}\pattern{scale = 0.5}{3}{1/1,2/3,3/2}{0/0,0/1,0/2,0/3,1/0,1/1,1/3,2/0,3/0,3/1,3/3}$	 \\[5pt]
 			\hline	
\footnotesize{16}  & $\pattern{scale = 0.5}{3}{1/1,2/2,3/3}{0/0,0/1,0/2,0/3,1/0,1/2,2/0,2/1,2/3,3/0,3/2}\pattern{scale = 0.5}{3}{1/1,2/3,3/2}{0/0,0/1,0/2,0/3,1/0,1/2,2/0,2/1,2/3,3/0,3/2}$	  & 
\footnotesize{17}  &  $\pattern{scale = 0.5}{3}{1/1,2/2,3/3}{0/0,0/1,0/2,0/3,1/0,1/1,1/3,2/0,2/2,3/0,3/1,3/3}\pattern{scale = 0.5}{3}{1/1,2/3,3/2}{0/0,0/1,0/2,0/3,1/0,1/1,1/3,2/0,2/2,3/0,3/1,3/3}$	& 
\footnotesize{18}  &  $\pattern{scale = 0.5}{3}{1/1,2/2,3/3}{0/0,0/1,0/2,0/3,1/0,1/2,2/0,2/1,2/2,2/3,3/0,3/2}\pattern{scale = 0.5}{3}{1/1,2/3,3/2}{0/0,0/1,0/2,0/3,1/0,1/2,2/0,2/1,2/2,2/3,3/0,3/2}$	 \\[5pt]
			\hline
	\end{tabular}
\end{center} 
}
\vspace{-0.5cm}
 	\caption{Equidistributions explainable by the complement operation.}\label{tab-2}
\end{table}

\subsection{Reduction to known equidistribution results}\label{reduction-sec}

\begin{table}[t]
 	{
 		\renewcommand{\arraystretch}{1.5}
 \begin{center} 
 		\begin{tabular}{|c|c|c|c||c|c|c|c|}
 			\hline
 			
 	\footnotesize{nr.}	& {\footnotesize patterns}  & \footnotesize{nr.}	& {\footnotesize patterns}  & \footnotesize{nr.}	& {\footnotesize patterns} & \footnotesize{nr.}	& {\footnotesize patterns}  \\ 
 			\hline
 			\hline	
\footnotesize{19}  & $\pattern{scale = 0.5}{3}{1/1,2/2,3/3}{0/0,0/1,0/2,0/3,1/0,1/1,1/2,2/1,2/2,2/0,3/0}\pattern{scale = 0.5}{3}{1/1,2/3,3/2}{0/0,0/1,0/2,0/3,1/0,1/1,1/2,2/1,2/2,2/0,3/0}$  & 
\footnotesize{20}  &  $\pattern{scale = 0.5}{3}{1/1,2/2,3/3}{0/0,0/1,0/2,0/3,1/0,2/0,2/2,2/3,3/0,3/2,3/3}\pattern{scale = 0.5}{3}{1/1,2/3,3/2}{0/0,0/1,0/2,0/3,1/0,2/0,2/2,2/3,3/0,3/2,3/3}$	& 
\footnotesize{21}  &  $\pattern{scale = 0.5}{3}{1/1,2/2,3/3}{0/0,0/1,0/2,0/3,1/0,1/1,1/2,2/0,2/1,2/2,3/0,3/3}\pattern{scale = 0.5}{3}{1/1,2/3,3/2}{0/0,0/1,0/2,0/3,1/0,1/1,1/2,2/0,2/1,2/2,3/0,3/3}$	&
\footnotesize{22}  &  $\pattern{scale = 0.5}{3}{1/1,2/2,3/3}{0/0,0/1,0/2,0/3,1/0,1/1,2/0,2/2,2/3,3/0,3/2,3/3}\pattern{scale = 0.5}{3}{1/1,2/3,3/2}{0/0,0/1,0/2,0/3,1/0,1/1,2/0,2/2,2/3,3/0,3/2,3/3}$ \\[5pt]
			\hline
	\end{tabular}
\end{center} 
}
\vspace{-0.5cm}
 	\caption{Equidistributions explainable by reduction to known results; pairs 19 and 20 (resp., 21 and 22) have the same distributions.}\label{tab-3}
\end{table}

In this subsection, we explain equidistributions in Table~\ref{tab-3}, which cannot be justified using the complement operation, as done in Section~\ref{compl-sec}. The following theorem  pertains to patterns nr.\ 64 and 65 in \cite{KZ}.

\begin{thm}[\cite{KZ}]\label{KZ-pat-length-2} We have $\pattern{scale = 0.5}{2}{1/1,2/2}{0/0,0/1,1/0,1/1,2/2}\sim_d \hspace{-0.15cm} \pattern{scale = 0.5}{2}{1/1,2/2}{0/1,0/2,1/1,1/2,2/0}$. \end{thm}

The following corollary follows from Theorem~\ref{KZ-pat-length-2} by noting that the complement of $\pattern{scale = 0.5}{2}{1/2,2/1}{0/0,0/1,1/0,1/1,2/2}$ is $\hspace{-0.1cm} \pattern{scale = 0.5}{2}{1/1,2/2}{0/1,0/2,1/1,1/2,2/0}$. However, for our purposes, we need a different proof that establishes the joint equidistribution of the patterns in question. The procedure of moving elements described by us is equivalent, when applying reverse and complement, to the procedure described in the proof of Theorem~\ref{9-box-rs} below (in particular, instead of replacing elements while going from left to right, in the next proof we go from right to left). The justifications that everything works are essentially the same as in the proof of Theorem~\ref{9-box-rs} and are therefore omitted.

\begin{coro}\label{two-pat-length-2} We have $\pattern{scale = 0.5}{2}{1/1,2/2}{0/0,0/1,1/0,1/1,2/2}\sim_d \hspace{-0.15cm} \pattern{scale = 0.5}{2}{1/2,2/1}{0/0,0/1,1/0,1/1,2/2}$. \end{coro}

\begin{proof} Note that any occurrence of the patterns begins with a left-to-right minimum, and occurrences of $\pattern{scale = 0.5}{2}{1/2,2/1}{0/0,0/1,1/0,1/1,2/2}$ can only be formed by (not necessarily all) consecutive left-to-right minima. Also, any element can be the second element in at most one occurrence of a pattern. Finally, the second element in an occurrence of  $\pattern{scale = 0.5}{2}{1/1,2/2}{0/0,0/1,1/0,1/1,2/2}$ must be a right-to-left maximum. 

Let $\pi_{i_1}\pi_{i_2}\ldots\pi_{i_s}$ be the subsequence of a permutation $\pi=\pi_1\ldots\pi_n$ that includes all elements in all occurrences of the patterns. In the procedure we are about to describe, we will permute the elements of $\pi_{i_1}\pi_{i_2}\ldots\pi_{i_s}$ in a certain way while keeping the positions in $\pi$ occupied by these elements unchanged. The permutation $\pi'$ obtained from $\pi$ by permuting $\pi_{i_1}\pi_{i_2}\ldots\pi_{i_s}$ will have the desired properties.  

Consider the unique occurrence $\pi_{i_x}\pi_{i_s}$ of a pattern that ends with $\pi_{i_s}$ and swap $\pi_{i_x}$ and $\pi_{i_s}$ in $\pi$. Such a swap will change exactly one occurrence of a pattern to an occurrence of the other pattern, without introducing or removing any other occurrences of the patterns. Next, consider $\pi^{\star}=\pi^{\star}_1\ldots \pi^{\star}_n$ obtained from $\pi$ by replacing $\pi_{i_x}$ with $\pi_{i_s}$, and suppose $i_j<i_s$ is largest such that $\pi^{\star}_{i_j}$ is the second element in an occurrence $\pi^{\star}_{i_t}\pi^{\star}_{i_j}$ of a pattern. Swap  $\pi^{\star}_{i_t}$ and $\pi^{\star}_{i_j}$. Now, consider $\pi^{\star\star}=\pi^{\star\star}_1\ldots\pi^{\star\star}_n$ obtained from $\pi^{\star}$ by replacing $\pi^{\star}_{i_t}$ and $\pi^{\star}_{i_j}$, and suppose $i_k<i_j$ is largest such that $\pi^{\star\star}_{i_k}$ is the second element in an occurrence $\pi^{\star\star}_{i_{\ell}}\pi^{\star\star}_{i_k}$ of a pattern. Swap  $\pi^{\star\star}_{i_{\ell}}$ and $\pi^{\star\star}_{i_k}$. And so on. Eventually, all occurrences of the patterns in $\pi$ will be swapped, which proves their (joint) equidistribution. 

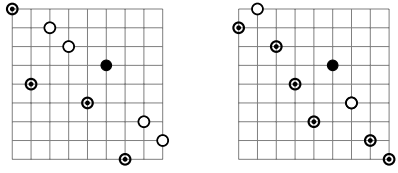
\begin{figure}
	\begin{center}
		
		\begin{tabular}{ccc}
			
			\begin{tikzpicture}[scale=0.25] 
			\tikzset{    
				grid/.style={      
					draw,      
					step=1cm,      
					gray!100,     
					very thin,      
				}, 
				cell/.style={    
					draw,    
					anchor=center,  
					text centered,    
				},  
				graycell/.style={ 
					fill=gray!40,   
					draw=none,   
					minimum width=1cm, 
					minimum height=1cm,   
					anchor=south west,   
				}
			}  
			
			\draw[grid] (0,0) grid (8,8);  
			%
			\draw[thick] (0,8) circle (8pt); \draw[thick] (0,8) circle (2pt); 
			\draw[thick] (1,4) circle (8pt); \draw[thick] (1,4) circle (2pt); 
			\draw[thick] (2,7) circle (8pt);  \filldraw[white] (2,7) circle (6pt);
			\draw[thick] (3,6) circle (8pt);  \filldraw[white] (3,6) circle (6pt);
			\filldraw[black] (5,5) circle (8pt);
			\draw[thick] (4,3) circle (8pt);  \draw[thick] (4,3) circle (2pt);
			\draw[thick] (6,0) circle (8pt);  \draw[thick] (6,0) circle (2pt);
			\draw[thick] (7,2) circle (8pt); \filldraw[white] (7,2) circle (6pt); 
			\draw[thick] (8,1) circle (8pt); \filldraw[white] (8,1) circle (6pt); 
			
			\end{tikzpicture} 
			
			& \ & 
			\begin{tikzpicture}[scale=0.25] 
			\tikzset{    
				grid/.style={      
					draw,      
					step=1cm,      
					gray!100,     
					very thin,      
				}, 
				cell/.style={    
					draw,    
					anchor=center,  
					text centered,    
				},  
				graycell/.style={ 
					fill=gray!40,   
					draw=none,   
					minimum width=1cm, 
					minimum height=1cm,   
					anchor=south west,   
				}
			}  
			
			\draw[grid] (0,0) grid (8,8);  
			%
			\draw[thick] (0,7) circle (8pt); \draw[thick] (0,7) circle (2pt); 
			\draw[thick] (1,8) circle (8pt);  \filldraw[white] (1,8) circle (6pt);
			\draw[thick] (2,6) circle (8pt);  \draw[thick] (2,6) circle (2pt);
			\draw[thick] (3,4) circle (8pt);  \draw[thick] (3,4) circle (2pt);
			\draw[thick] (4,2) circle (8pt); \draw[thick] (4,2) circle (2pt); 
			\draw[thick] (6,3) circle (8pt);  \filldraw[white] (6,3) circle (6pt);
			\filldraw[black] (5,5) circle (8pt);
			\draw[thick] (6,3) circle (8pt);  \filldraw[white] (6,3) circle (6pt);
			\draw[thick] (7,1) circle (8pt); \draw[thick] (7,1) circle (2pt); 
			\draw[thick] (8,0) circle (8pt); \draw[thick] (8,0) circle (2pt); 
			
						\end{tikzpicture}
		\end{tabular}
	\end{center}
	
	\vspace{-0.5cm}
	
	\caption{Permutations $\pi= 958746132$  (to the left) and $\pi'=897536421$ (to the right) illustrating the proof of Corollary~\ref{two-pat-length-2}. The circled dots represent the left-to-right minima in $\pi$, and the white dots represent the second elements of occurrences of $p_1$ or $p_2$.}\label{fig-1-2}
\end{figure}

For example, for the permutation $\pi$ in Figure~\ref{fig-1-2}, we do the procedure as follows: $\pi=958746\uline{1}3\uline{2}\xrightarrow{\text{swap 1 \& 2}}958746\uline{2}\uline{3}1\xrightarrow{\text{swap 2 \& 3}} 9587\uline{4}6\uline{3}21\xrightarrow{\text{swap 3 \& 4}}9\uline{5}8\uline{7}36421\xrightarrow{\text{swap 5 \& 7}} 9\uline{78}536421\xrightarrow{\text{swap 7 \& 8}} \uline{98}7536421\xrightarrow{\text{swap 8 \& 9}} 897536421$, so that $\pi'= 897536421$.

Note that the occurrences 58, 57, 13 and 12 of $\pattern{scale = 0.5}{2}{1/1,2/2}{0/0,0/1,1/0,1/1,2/2}$ in $\pi$ are replaced, respectively, by the occurrence 87, 73, 32 and 21 of $\pattern{scale = 0.5}{2}{1/2,2/1}{0/0,0/1,1/0,1/1,2/2}$ in $\pi'$, while the occurrences 95 and 41 of $\pattern{scale = 0.5}{2}{1/2,2/1}{0/0,0/1,1/0,1/1,2/2}$ in $\pi$ are replaced by the occurrences 89 and 34 of $\pattern{scale = 0.5}{2}{1/1,2/2}{0/0,0/1,1/0,1/1,2/2}$ in $\pi'$. \end{proof}

Let $T_{n,k}$ be the number of $n$-permutations with $k$ occurrences of $\pattern{scale = 0.5}{2}{1/1,2/2}{0/0,0/1,1/0,1/1}$. Theorem 4.1 in \cite{KZ} states that 	
\begin{align}
	T_{n,k} = & \, T_{n-1,k-1}+ (n-1)T_{n-1,k}  \label{rec-rel-8-9}
	\end{align}
	with  the initial conditions $T_{n,0}=(n-1)!$ for $n\geq 1$ and $T_{0,0}=1$,  which shows that $T_{n,k}=C(n,k+1)$, the \textit{unsigned Stirling number of the first kind}. The next theorem is proved by showing that the distribution of the pattern $\pattern{scale = 0.5}{2}{1/2,2/1}{0/0,0/1,1/0,1/1}$ is given by $T_{n,k}$.	 Note that the following theorem can also be proved by copying and pasting the proof of Corollary~\ref{two-pat-length-2}, which demonstrates the joint equidistribution of the patterns. However, we provide an alternative proof in the style of the respective proofs in \cite{KZ}.
	
\begin{thm}\label{equdis-2-patterns-length-2} We have $\pattern{scale = 0.5}{2}{1/1,2/2}{0/0,0/1,1/0,1/1}\sim_d \hspace{-0.15cm} \pattern{scale = 0.5}{2}{1/2,2/1}{0/0,0/1,1/0,1/1}$.\end{thm}

\begin{proof} 
Let $T'_{n,k}$ be the number of $n$-permutations with $k$ occurrences of $\pattern{scale = 0.5}{2}{1/2,2/1}{0/0,0/1,1/0,1/1}$. A permutation counted by $T'_{n,k}$ can either be obtained by inserting the new largest element $n$ 
\begin{itemize}
\item to the left of a permutation counted by $T'_{n-1,k-1}$ (which clearly results in one extra occurrence of $\pattern{scale = 0.5}{2}{1/2,2/1}{0/0,0/1,1/0,1/1}$), or
\item any other but the leftmost position in a permutation $\pi'$ counted by  $T'_{n-1,k}$, which cannot result in a new occurrence of $\pattern{scale = 0.5}{2}{1/2,2/1}{0/0,0/1,1/0,1/1}$ because either $n$ is in the rightmost possible position, or the leftmost element in $\pi'$ is in the shaded area.  
\end{itemize}
Hence, $T'_{n,k}$ satisfies the same recursion as  (\ref{rec-rel-8-9}). A similar argument shows that $T'_{n,0}=(n-1)T'_{n-1,0}$ because inserting $n$ in any except the leftmost position in an $(n-1)$-permutation avoiding $\pattern{scale = 0.5}{2}{1/2,2/1}{0/0,0/1,1/0,1/1}$ results in an $n$-permutation also avoiding $\pattern{scale = 0.5}{2}{1/2,2/1}{0/0,0/1,1/0,1/1}$. Since, $T'_{0,0}=1$, we have that $T'_{n,k}=T_{n,k}$ for all $n,k\geq 0$, proving that $\pattern{scale = 0.5}{2}{1/1,2/2}{0/0,0/1,1/0,1/1}\sim_d \hspace{-0.15cm} \pattern{scale = 0.5}{2}{1/2,2/1}{0/0,0/1,1/0,1/1}$. \end{proof}

Occurrences of any pattern in Table~\ref{tab-3} in a permutation $\pi$ begin with $\pi_1=1$. Removing 1 from such a pattern results in a mesh pattern of length~2. But then, equidistribution of $\pattern{scale = 0.5}{3}{1/1,2/2,3/3}{0/0,0/1,0/2,0/3,1/0,1/1,1/2,2/1,2/2,2/0,3/0}\pattern{scale = 0.5}{3}{1/1,2/3,3/2}{0/0,0/1,0/2,0/3,1/0,1/1,1/2,2/1,2/2,2/0,3/0}$  (resp., $\pattern{scale = 0.5}{3}{1/1,2/2,3/3}{0/0,0/1,0/2,0/3,1/0,1/1,1/2,2/0,2/1,2/2,3/0,3/3}\pattern{scale = 0.5}{3}{1/1,2/3,3/2}{0/0,0/1,0/2,0/3,1/0,1/1,1/2,2/0,2/1,2/2,3/0,3/3}$) follows from equidistribution of $\pattern{scale = 0.5}{2}{1/1,2/2}{0/0,0/1,1/0,1/1} \pattern{scale = 0.5}{2}{1/2,2/1}{0/0,0/1,1/0,1/1}$ (resp., $\pattern{scale = 0.5}{2}{1/1,2/2}{0/0,0/1,1/0,1/1,2/2} \pattern{scale = 0.5}{2}{1/2,2/1}{0/0,0/1,1/0,1/1,2/2}$) by Theorem~\ref{equdis-2-patterns-length-2} (resp., Corollary~\ref{two-pat-length-2}).

The same distribution of pairs 19 and 20 (resp., 21 and 22) follow from the fact that under reverse and complement the pair $\pattern{scale = 0.5}{2}{1/1,2/2}{0/0,0/1,1/0,1/1} \pattern{scale = 0.5}{2}{1/2,2/1}{0/0,0/1,1/0,1/1}$ (resp., $\pattern{scale = 0.5}{2}{1/1,2/2}{0/0,0/1,1/0,1/1,2/2} \pattern{scale = 0.5}{2}{1/2,2/1}{0/0,0/1,1/0,1/1,2/2}$) goes to $\pattern{scale = 0.5}{2}{1/1,2/2}{1/1,1/2,2/1,2/2} \pattern{scale = 0.5}{2}{1/2,2/1}{1/1,1/2,2/1,2/2}$ (resp, $\pattern{scale = 0.5}{2}{1/1,2/2}{0/0,1/1,1/2,2/1,2/2} \pattern{scale = 0.5}{2}{1/2,2/1}{0/0,1/1,1/2,2/1,2/2}$).

\section{Cases based on the proof of Theorem~\ref{thm-L-shape}}\label{L-mesh-sec}

In this section, we explain equidistributions of pairs of patterns in Table~\ref{tab-5}.

\begin{table}[t]
 	{
 		\renewcommand{\arraystretch}{1.5}
 \begin{center} 
 		\begin{tabular}{|c|c||c|c||c|c||c|c|}
 			\hline
 			
 	\footnotesize{nr.}	& {\footnotesize patterns}  & \footnotesize{nr.}	& {\footnotesize patterns}  & \footnotesize{nr.}	& {\footnotesize patterns} & \footnotesize{nr.}	& {\footnotesize patterns}   \\ 
 			\hline
 			\hline	
\footnotesize{23}  & $\pattern{scale = 0.5}{3}{1/1,2/2,3/3}{0/0,0/1,0/2,1/0,2/0}\pattern{scale = 0.5}{3}{1/1,2/3,3/2}{0/0,0/1,0/2,1/0,2/0}$	  & 
\footnotesize{24}  &  $\pattern{scale = 0.5}{3}{1/1,2/2,3/3}{0/0,0/1, 1/0,0/2,2/0,2/2}\pattern{scale = 0.5}{3}{1/1,2/3,3/2}{0/0,0/1, 1/0,0/2,2/0,2/2}$	& 
\footnotesize{25}  &  	$\pattern{scale=0.5}{3}{1/1,2/2,3/3}{0/0,0/1,0/2,1/0,1/2, 2/0,2/1,2/3,3/2}\pattern{scale=0.5}{3}{1/1,2/3,3/2}{0/0,0/1,0/2,1/0,1/2, 2/0,2/1,2/3,3/2}$ &
\footnotesize{26}  & $\pattern{scale=0.5}{3}{1/1,2/2,3/3}{0/0,0/1,0/2,1/0,1/1, 1/3,2/0,3/1,3/3}\pattern{scale=0.5}{3}{1/1,2/3,3/2}{0/0,0/1,0/2,1/0,1/1, 1/3,2/0,3/1,3/3}$ \\[5pt]
 			\hline	
\footnotesize{27}  & 	$\pattern{scale=0.5}{3}{1/1,2/2,3/3}{0/0,0/1,0/2,1/0,1/2, 2/0,2/1,2/2,2/3,3/2}\pattern{scale=0.5}{3}{1/1,2/3,3/2}{0/0,0/1,0/2,1/0,1/2, 2/0,2/1,2/2,2/3,3/2}$  & 
\footnotesize{28}  &  $\pattern{scale=0.5}{3}{1/1,2/2,3/3}{0/0,0/1,0/2,1/0,1/1, 1/3,2/0,2/2,3/1,3/3} \pattern{scale=0.5}{3}{1/1,2/3,3/2}{0/0,0/1,0/2,1/0,1/1, 1/3,2/0,2/2,3/1,3/3}$ 	& 
\footnotesize{29}  &  	$\pattern{scale=0.5}{3}{1/1,2/2,3/3}{0/0,0/1,0/2,1/0,1/1, 1/2,1/3,2/0,2/1,2/3,3/1,3/2,3/3} \pattern{scale=0.5}{3}{1/1,2/3,3/2}{0/0,0/1,0/2,1/0,1/1, 1/2,1/3,2/0,2/1,2/3,3/1,3/2,3/3}$ & 
\footnotesize{30}  &  $\pattern{scale=0.5}{3}{1/1,2/2,3/3}{0/0,0/1,0/2,1/0,1/1, 1/2,2/0,2/1,2/2} \pattern{scale=0.5}{3}{1/1,2/3,3/2}{0/0,0/1,0/2,1/0,1/1, 1/2,2/0,2/1,2/2}$ \\[5pt]
 			\hline	
\footnotesize{31}  & $\pattern{scale=0.5}{3}{1/1,2/2,3/3}{0/0,0/1,0/2,1/0,1/1, 1/2,2/0,2/1,2/2,3/3}\pattern{scale=0.5}{3}{1/1,2/3,3/2}{0/0,0/1,0/2,1/0,1/1, 1/2,2/0,2/1,2/2,3/3}$	  & 
\footnotesize{32}  & $\pattern{scale = 0.5}{3}{1/1,2/2,3/3}{0/0,0/2,2/0}\pattern{scale = 0.5}{3}{1/1,2/3,3/2}{0/0,0/2,2/0}$ 	& 
\footnotesize{33}  &  	 $\pattern{scale = 0.5}{3}{1/1,2/2,3/3}{0/0,0/2,2/0,2/2}\pattern{scale = 0.5}{3}{1/1,2/3,3/2}{0/0,0/2,2/0,2/2}$	& 
\footnotesize{34}  & $\pattern{scale=0.5}{3}{1/1,2/2,3/3}{0/0,0/2,1/2, 2/0,2/1,2/3,3/2}\pattern{scale=0.5}{3}{1/1,2/3,3/2}{0/0,0/2,1/2, 2/0,2/1,2/3,3/2}$ \\[5pt]
 			\hline	
\footnotesize{35}  & 	$\pattern{scale=0.5}{3}{1/1,2/2,3/3}{0/0,0/2,1/1, 1/3,2/0,3/1,3/3}\pattern{scale=0.5}{3}{1/1,2/3,3/2}{0/0,0/2,1/1, 1/3,2/0,3/1,3/3}$   & 
\footnotesize{36}  & $\pattern{scale=0.5}{3}{1/1,2/2,3/3}{0/0,0/2,1/2, 2/0,2/1,2/2,2/3,3/2}\pattern{scale=0.5}{3}{1/1,2/3,3/2}{0/0,0/2,1/2, 2/0,2/1,2/2,2/3,3/2}$ 	& 
\footnotesize{37}  & $\pattern{scale=0.5}{3}{1/1,2/2,3/3}{0/0,0/2,1/1, 1/3,2/0,2/2,3/1,3/3} \pattern{scale=0.5}{3}{1/1,2/3,3/2}{0/0,0/2,1/1, 1/3,2/0,2/2,3/1,3/3}$ 	& 
\footnotesize{38}  &  $\pattern{scale=0.5}{3}{1/1,2/2,3/3}{0/0,0/2,1/1, 1/2,1/3,2/0,2/1,2/3,3/1,3/2,3/3} \pattern{scale=0.5}{3}{1/1,2/3,3/2}{0/0,0/2,1/1, 1/2,1/3,2/0,2/1,2/3,3/1,3/2,3/3}$ \\[5pt]
 			\hline	
\footnotesize{39}  & 	$\pattern{scale=0.5}{3}{1/1,2/2,3/3}{0/0,0/2,1/1,1/2,1/3, 2/0,2/1,2/2,2/3,3/1,3/2}\pattern{scale=0.5}{3}{1/1,2/3,3/2}{0/0,0/2,1/1,1/2,1/3, 2/0,2/1,2/2,2/3,3/1,3/2}$  & 
\footnotesize{40}      &  $\pattern{scale=0.5}{3}{1/1,2/2,3/3}{0/0,0/2,1/1,1/2,1/3, 2/0,2/1,2/2,2/3,3/1,3/2,3/3}\pattern{scale=0.5}{3}{1/1,2/3,3/2}{0/0,0/2,1/1,1/2,1/3, 2/0,2/1,2/2,2/3,3/1,3/2,3/3}$	& 
\footnotesize{41}       &  $\pattern{scale=0.5}{3}{1/1,2/2,3/3}{0/0,0/2,0/3,1/2, 2/0,2/1,2/3,3/0,3/2}\pattern{scale=0.5}{3}{1/1,2/3,3/2}{0/0,0/2,0/3,1/2, 2/0,2/1,2/3,3/0,3/2}$	& 
\footnotesize{42}       & $\pattern{scale=0.5}{3}{1/1,2/2,3/3}{0/0,0/2,0/3,1/2, 2/0,2/1,2/2,2/3,3/0,3/2}\pattern{scale=0.5}{3}{1/1,2/3,3/2}{0/0,0/2,0/3,1/2, 2/0,2/1,2/2,2/3,3/0,3/2}$   \\[5pt]
 			\hline	
\footnotesize{43}  & 	$\pattern{scale=0.5}{3}{1/1,2/2,3/3}{0/0,0/2,0/3,1/1, 1/2,1/3,2/0,2/1,2/3,3/0,3/1,3/2,3/3} \pattern{scale=0.5}{3}{1/1,2/3,3/2}{0/0,0/2,0/3,1/1, 1/2,1/3,2/0,2/1,2/3,3/0,3/1,3/2,3/3}$  & 
\footnotesize{44}      & $\pattern{scale=0.5}{3}{1/1,2/2,3/3}{0/0,2/0,3/0,1/1,2/1,3/1,0/2,1/2,2/2,3/2,0/3,1/3,2/3} \pattern{scale=0.5}{3}{1/1,2/3,3/2}{0/0,2/0,3/0,1/1,2/1,3/1,0/2,1/2,2/2,3/2,0/3,1/3,2/3}$ 	& 
 \footnotesize{45}      & $\pattern{scale=0.5}{3}{1/1,2/2,3/3}{0/0,0/2,0/3,1/1,1/2,1/3, 2/0,2/1,2/2,2/3,3/0,3/1,3/2,3/3}\pattern{scale=0.5}{3}{1/1,2/3,3/2}{0/0,0/2,0/3,1/1,1/2,1/3, 2/0,2/1,2/2,2/3,3/0,3/1,3/2,3/3}$ & 
      &   \\[5pt]
 \hline	

	\end{tabular}
\end{center} 
}
\vspace{-0.5cm}
 	\caption{Equidistributions explainable based on the proof of Theorem~\ref{thm-L-shape}. Pairs 23--29, 32--38, 40--43 and 45 (and pair 6 in Table~\ref{tab-1}) have the property that shading in the North-East area, relative to the minimum element, is symmetric with respect to a horizontal line. Pairs 30 and 46, and 31 and 47 have the same distributions by Theorem~\ref{an-equidistr-thm}.}\label{tab-5}
\end{table}

\begin{thm}\label{thm-L-shape}
We have $\pattern{scale = 0.5}{3}{1/1,2/2,3/3}{0/0,0/1, 1/0,0/2,2/0} \sim_d \hspace{-0.15cm} \pattern{scale = 0.5}{3}{1/1,2/3,3/2}{0/0,0/1, 1/0,0/2,2/0}$.
\end{thm}

\begin{proof}

Let $\pi=\pi_1\ldots\pi_n\in S_n$ and $x_1>x_2>\cdots>x_k$ be the sequence of left-to-right minima in $\pi$. If $\pi$ has no occurrences of the patterns, $\pi$ will be mapped to itself. Otherwise, by definitions, any occurrence of  $\pattern{scale = 0.5}{3}{1/1,2/2,3/3}{0/0,0/1, 1/0,0/2,2/0}$ or $\hspace{-0.1cm} \pattern{scale = 0.5}{3}{1/1,2/3,3/2}{0/0,0/1, 1/0,0/2,2/0}$ in $\pi$ must begin with an $x_i$, $1\leq i\leq k$. Moreover, referring to a schematic representation of $\pi$ in Figure~\ref{fig-1}, occurrences of the patterns must be entirely inside $\{x_i\}\cup A_i$ for some $i$, $1\leq i\leq k$, where any $y\in A_i$ satisfies $x_{i}<y<x_{i-1}$ (some $A_i$'s can be empty, and $A_i$'s with just one element do not give any occurrences of the patterns). Indeed, the element $x_{i-1}$ (resp., $x_{i+1}$) ensures that elements in occurrences of the patterns cannot be above (resp., to the right of) $A_i$.  Also, note that if $A_i$ is of length at least 2, then \textit{each} pair of elements in it is involved in an occurrence of either $\pattern{scale = 0.5}{3}{1/1,2/2,3/3}{0/0,0/1, 1/0,0/2,2/0}$ or $\hspace{-0.1cm}\pattern{scale = 0.5}{3}{1/1,2/3,3/2}{0/0,0/1, 1/0,0/2,2/0}$. This is because the rightmost and leftmost elements in $A_i$, as well as the maximum and minimum elements, are involved in occurrences of these patterns (so there are no elements below $A_i$ and to the left of it).

Let 
$X_i=\{x_i+1,x_i+2,\ldots,x_{i-1}-1\}$, where $1\leq i\leq k$ and $x_0:=n+1$. We define the operation of complementation of $X_i$'s, $1\leq i\leq k$. Suppose $X_i$ is formed by the elements $\pi_{i_1}\pi_{i_2}\ldots\pi_{i_m}$ with $i_1<i_2<\cdots<i_m$. Then, for $j=1,2,\ldots,m$,  if $\pi_{i_j}$ is the $\ell$-th smallest element in $X_i$ then replace it by the $(m+1-\ell)$-th smallest element in $X_i$.  The complementation of $X_i$'s induces complementation on $A_i$ and exchanges occurrences of the patterns 12 and 21 in each $A_i$, hence it replaces each occurrence of $\pattern{scale = 0.5}{3}{1/1,2/2,3/3}{0/0,0/1, 1/0,0/2,2/0}$  by an occurrence of $\pattern{scale = 0.5}{3}{1/1,2/3,3/2}{0/0,0/1, 1/0,0/2,2/0}$, and vice versa, and no new occurrences of these patterns in the obtained permutation $\pi'$ can appear (any new occurrence must involve an element outside of an $A_i$, which is impossible because if an element was in a shaded area with respect to two elements $a$ and $b$ making $ab$ non-occurrence of a pattern, it will continue to be in the shaded area for $ab$). 

To complete the proof of the theorem, apply the operation of complementation to all $X_i$'s in every permutation in $S_n$, which will turn, in a bijective manner, permutations with $x$ occurrences of  $\pattern{scale = 0.5}{3}{1/1,2/2,3/3}{0/0,0/1, 1/0,0/2,2/0}$  and $y$ occurrence of $\pattern{scale = 0.5}{3}{1/1,2/3,3/2}{0/0,0/1, 1/0,0/2,2/0}$ into permutations with $y$ occurrences of  $\pattern{scale = 0.5}{3}{1/1,2/2,3/3}{0/0,0/1, 1/0,0/2,2/0}$  and $x$ occurrence of $\pattern{scale = 0.5}{3}{1/1,2/3,3/2}{0/0,0/1, 1/0,0/2,2/0}$. \end{proof}

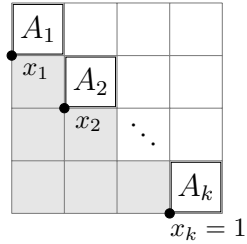
\begin{figure}  
\begin{center}  
\begin{tikzpicture}[scale=0.7]    
\tikzset{      
   grid/.style={        
      draw,        
      step=1cm,        
      gray!100,    
      thin,        
    },   
    cell/.style={      
      draw,      
      anchor=center,    
      text centered,    
     },    
    graycell/.style={   
      fill=gray!20,
      draw=none,     
      minimum width=1cm,   
      minimum height=1cm,     
      anchor=south west,            
    }    
  } 
      
  \fill[graycell] (0,0) rectangle (1,1);
  \fill[graycell] (0,1) rectangle (1,2);
  \fill[graycell] (0,2) rectangle (1,3);
  \fill[graycell] (1,0) rectangle (2,1);
  \fill[graycell] (1,1) rectangle (2,2);
  \fill[graycell] (2,0) rectangle (3,1);

\node[cell,scale=2.4] at (0.5,3.5) {};  
\node[cell,scale=2.4] at (1.5,2.5) {};  
\node[cell,scale=2.4] at (3.5,0.5) {};  
   
 \draw[grid] (0,0) grid (4,4);  

  \node at (0.5,3.5) {$A_1$};  
  \node at (1.5,2.5) {$A_2$};  
  \node at (3.5,0.5) {$A_k$};
  \node[anchor=west] at (0,2.7) {\footnotesize{$x_1$}}; 
  \node[anchor=west] at (1,1.7) {\footnotesize{$x_2$}}; 
  \node[anchor=west] at (2.8,-0.3) {\footnotesize{$x_k=1$}}; 
 
  \filldraw[black] (0,3) circle (2.5pt);  
  \filldraw[black] (1,2) circle (2.5pt); 
  \filldraw[black] (3,0) circle (2.5pt);  
 
 \node[anchor=center, rotate=-45] at (2.5,1.5) {$\ldots$};
\end{tikzpicture} 

\vspace{-0.5cm}

\end{center}
\caption{The structure of permutations in the proof of Theorem~\ref{thm-L-shape}. Each $A_i$ is of length $\geq 0$.}\label{fig-1}
\end{figure}

To illustrate our proof of Theorem~\ref{thm-L-shape}, consider the permutation $\pi=9(11)4(12)8(10)5713(13)62$ with $x_1=9$, $x_2=4$, $x_3=1$, $A_1=(11)$, $A_2=857$,  $A_3=32$, $X_1=(11)(12)(10)(13)$, $X_2=8576$, and $X_3=32$. The complementation process will result in the permutation $\pi'=9(12)4(11)5(13)8612(10)73$.  See Figure~\ref{fig-2} for the permutation diagrams for $\pi$ and $\pi'$. Note that the occurrence 457 of  $\pattern{scale = 0.5}{3}{1/1,2/2,3/3}{0/0,0/1, 1/0,0/2,2/0}$  in $\pi$ is replaced by the occurrence 486 of $\pattern{scale = 0.5}{3}{1/1,2/3,3/2}{0/0,0/1, 1/0,0/2,2/0}$ in $\pi'$, while the occurrences 485, 487, 132 of  $\pattern{scale = 0.5}{3}{1/1,2/3,3/2}{0/0,0/1, 1/0,0/2,2/0}$ in $\pi$ are replaced, respectively, by the occurrences 458, 456, 123 of  $\pattern{scale = 0.5}{3}{1/1,2/2,3/3}{0/0,0/1, 1/0,0/2,2/0}$ in $\pi'$; no new occurrences of the patterns are introduced in~$\pi'$.

\begin{figure}
\begin{center}

\begin{tabular}{ccc}

\begin{tikzpicture}[scale=0.25] 
\tikzset{    
   grid/.style={      
      draw,      
      step=1cm,      
      gray!100,     
      very thin,      
    }, 
    cell/.style={    
      draw,    
      anchor=center,  
      text centered,    
     },  
    graycell/.style={ 
      fill=gray!40,   
      draw=none,   
      minimum width=1cm, 
      minimum height=1cm,   
      anchor=south west,   
    }
  }  
    
  \draw[grid] (0,0) grid (12,12);  
%
\draw[thick] (0,8) circle (8pt); \draw[thick] (0,8) circle (2pt); 
  \draw[thick] (1,10) circle (8pt); \filldraw[white] (1,10) circle (6pt); 
  \draw[thick] (2,3) circle (8pt); \draw[thick] (2,3) circle (2pt); 
  \filldraw[black] (3,11) circle (8pt); 
  \draw[thick] (4,7) circle (8pt);  \filldraw[white] (4,7) circle (6pt); 
  \filldraw[black] (5,9) circle (8pt); 
  \draw[thick] (6,4) circle (8pt); \filldraw[white] (6,4) circle (6pt); 
  \draw[thick] (7,6) circle (8pt);  \filldraw[white] (7,6) circle (6pt); 
  \draw[thick] (8,0) circle (8pt); \draw[thick] (8,0) circle (2pt); 
  \draw[thick] (9,2) circle (8pt); \filldraw[white] (9,2) circle (6pt); 
  \filldraw[black] (10,12) circle (8pt);  
  \filldraw[black] (11,5) circle (8pt); 
  \draw[thick] (12,1) circle (8pt); \filldraw[white] (12,1) circle (6pt); 
  
   \node[cell] at (1,10) {};  
   \node[cell,scale=3.6] at (5.5,5.5) {};  
\node[cell,scale=2,] at (10.5,1.5) { \ \hspace{-0.5mm} \ };

\end{tikzpicture} 

& \ & 
 \begin{tikzpicture}[scale=0.25] 
\tikzset{    
   grid/.style={      
      draw,      
      step=1cm,      
      gray!100,     
      very thin,      
    }, 
    cell/.style={    
      draw,    
      anchor=center,  
      text centered,    
     },  
    graycell/.style={ 
      fill=gray!40,   
      draw=none,   
      minimum width=1cm, 
      minimum height=1cm,   
      anchor=south west,   
    }
  }  
    
  \draw[grid] (0,0) grid (12,12);  
%
\draw[thick] (0,8) circle (8pt); \draw[thick] (0,8) circle (2pt); 
  \draw[thick] (1,11) circle (8pt); \filldraw[white] (1,11) circle (6pt); 
  \draw[thick] (2,3) circle (8pt); \draw[thick] (2,3) circle (2pt); 
  \filldraw[black] (3,10) circle (8pt); 
  \draw[thick] (4,4) circle (8pt);  \filldraw[white] (4,4) circle (6pt); 
  \filldraw[black] (5,12) circle (8pt); 
  \draw[thick] (6,7) circle (8pt); \filldraw[white] (6,7) circle (6pt); 
  \draw[thick] (7,5) circle (8pt);  \filldraw[white] (7,5) circle (6pt); 
  \draw[thick] (8,0) circle (8pt); \draw[thick] (8,0) circle (2pt); 
  \draw[thick] (9,1) circle (8pt); \filldraw[white] (9,1) circle (6pt); 
  \filldraw[black] (10,9) circle (8pt);  
  \filldraw[black] (11,6) circle (8pt); 
  \draw[thick] (12,2) circle (8pt);\filldraw[white] (12,2) circle (6pt); 
  
    \node[cell] at (1,11) {};  
    \node[cell,scale=3.6] at (5.5,5.5) {};  
\node[cell,scale=2,] at (10.5,1.5) { \ \hspace{-0.5mm} \ };  

\end{tikzpicture}
\end{tabular}
\end{center}

\vspace{-0.5cm}

\caption{Permutations $\pi$ (to the left) and $\pi'$ (to the right) illustrating the proof of Theorem~\ref{thm-L-shape}. The circled dots represent the left-to-right minima, and the white dots represent the elements in $A_i$'s, $i=1,2,3$.}\label{fig-2}
\end{figure}
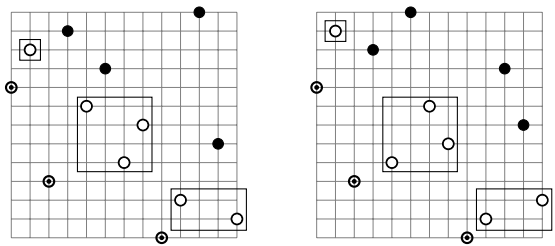

Essentially copying and pasting the proof of Theorem~\ref{thm-L-shape}, we can justify several other equidistributions, which we proceed to do next. Instead of providing full proofs, we will simply explain the differences with, and necessary adjustments to the proof of Theorem~\ref{thm-L-shape}. The common feature of the mesh patterns in the next theorem is that the shading in the North-East area relative to the minimum element is symmetric with respect to a horizontal line (similarly to the patterns considered in Section~\ref{compl-sec}). Note that the equidistribution $\pattern{scale=0.5}{3}{1/1,2/2,3/3}{0/0,0/1,0/2,1/0,1/1, 1/2,1/3,2/0,2/1,2/2,2/3,3/1,3/2,3/3} \sim_d \hspace{-0.15cm} \pattern{scale=0.5}{3}{1/1,2/3,3/2}{0/0,0/1,0/2,1/0,1/1, 1/2,1/3,2/0,2/1,2/2,2/3,3/1,3/2,3/3}$ has already been explained in Section~\ref{direct-arg-sec} (this is pair 6 in Table~\ref{tab-1}), but we still include this pair in Theorem~\ref{thm-L-shape-1} for completeness.

\begin{thm}\label{thm-L-shape-1}
We have $\pattern{scale = 0.5}{3}{1/1,2/2,3/3}{0/0,0/1, 1/0,0/2,2/0,2/2} \sim_d \pattern{scale = 0.5}{3}{1/1,2/3,3/2}{0/0,0/1, 1/0,0/2,2/0,2/2}$, 
$\pattern{scale=0.5}{3}{1/1,2/2,3/3}{0/0,0/1,0/2,1/0,1/2, 2/0,2/1,2/3,3/2} \sim_d \hspace{-0.15cm} \pattern{scale=0.5}{3}{1/1,2/3,3/2}{0/0,0/1,0/2,1/0,1/2, 2/0,2/1,2/3,3/2}$,
$\pattern{scale=0.5}{3}{1/1,2/2,3/3}{0/0,0/1,0/2,1/0,1/1, 1/3,2/0,3/1,3/3} \sim_d \hspace{-0.15cm} \pattern{scale=0.5}{3}{1/1,2/3,3/2}{0/0,0/1,0/2,1/0,1/1, 1/3,2/0,3/1,3/3}$, \\
$\pattern{scale=0.5}{3}{1/1,2/2,3/3}{0/0,0/1,0/2,1/0,1/1, 1/3,2/0,2/2,3/1,3/3} \sim_d \hspace{-0.15cm} \pattern{scale=0.5}{3}{1/1,2/3,3/2}{0/0,0/1,0/2,1/0,1/1, 1/3,2/0,2/2,3/1,3/3}$,
$\pattern{scale=0.5}{3}{1/1,2/2,3/3}{0/0,0/1,0/2,1/0,1/2, 2/0,2/1,2/2,2/3,3/2} \sim_d \hspace{-0.15cm} \pattern{scale=0.5}{3}{1/1,2/3,3/2}{0/0,0/1,0/2,1/0,1/2, 2/0,2/1,2/2,2/3,3/2}$, 
$\pattern{scale=0.5}{3}{1/1,2/2,3/3}{0/0,0/1,0/2,1/0,1/1, 1/2,1/3,2/0,2/1,2/3,3/1,3/2,3/3} \sim_d \hspace{-0.15cm} \pattern{scale=0.5}{3}{1/1,2/3,3/2}{0/0,0/1,0/2,1/0,1/1, 1/2,1/3,2/0,2/1,2/3,3/1,3/2,3/3}$, and $\pattern{scale=0.5}{3}{1/1,2/2,3/3}{0/0,0/1,0/2,1/0,1/1, 1/2,1/3,2/0,2/1,2/2,2/3,3/1,3/2,3/3} \sim_d \hspace{-0.15cm} \pattern{scale=0.5}{3}{1/1,2/3,3/2}{0/0,0/1,0/2,1/0,1/1, 1/2,1/3,2/0,2/1,2/2,2/3,3/1,3/2,3/3}$.
\end{thm}

\begin{proof}
The only difference from the proof of Theorem~\ref{thm-L-shape} is that now not any pair of elements in $A_i$ is involved in an occurrence of a pattern in question, but only pairs forming an occurrence of respective patterns of length 2 in the North-East area relative to $x_i$. This does not affect our arguments in the proof of Theorem~\ref{thm-L-shape} since the only changes in $\pi$ occur inside $X_i$'s and under complementing $X_i$, occurrences of $\pattern{scale = 0.5}{2}{1/1,2/2}{1/1}$ and  $\hspace{-0.15cm} \pattern{scale = 0.5}{2}{1/2,2/1}{1/1}$ (resp., $\pattern{scale = 0.5}{2}{1/1,2/2}{0/1,1/0,1/2,2/1}$ and  $\hspace{-0.15cm} \pattern{scale = 0.5}{2}{1/2,2/1}{0/1,1/0,1/2,2/1}$, $\pattern{scale = 0.5}{2}{1/1,2/2}{0/0,0/2,2/0,2/2}$ and  $\hspace{-0.15cm} \pattern{scale = 0.5}{2}{1/2,2/1}{0/0,0/2,2/0,2/2}$,    $\pattern{scale = 0.5}{2}{1/1,2/2}{0/0,0/2,1/1,2/0,2/2}$ and  $\hspace{-0.15cm} \pattern{scale = 0.5}{2}{1/2,2/1}{0/0,0/2,1/1,2/0,2/2}$, $\pattern{scale = 0.5}{2}{1/1,2/2}{0/1,1/0,1/1,1/2,2/1}$ and  $\hspace{-0.15cm} \pattern{scale = 0.5}{2}{1/2,2/1}{0/1,1/0,1/1,1/2,2/1}$, $\pattern{scale = 0.5}{2}{1/1,2/2}{0/0,0/1,0/2,1/0,1/2,2/0,2/1,2/2}$ and  $\hspace{-0.15cm} \pattern{scale = 0.5}{2}{1/2,2/1}{0/0,0/1,0/2,1/0,1/2,2/0,2/1,2/2}$, $\pattern{scale = 0.5}{2}{1/1,2/2}{0/0,0/1,0/2,1/0,1/1,1/2,2/0,2/1,2/2}$ and  $\hspace{-0.15cm} \pattern{scale = 0.5}{2}{1/2,2/1}{0/0,0/1,0/2,1/0,1/1,1/2,2/0,2/1,2/2}$) go to each other.
\end{proof}

\begin{thm}\label{9-ls}
We have $\pattern{scale=0.5}{3}{1/1,2/2,3/3}{0/0,0/1,0/2,1/0,1/1, 1/2,2/0,2/1,2/2} \sim_d \hspace{-0.15cm} \pattern{scale=0.5}{3}{1/1,2/3,3/2}{0/0,0/1,0/2,1/0,1/1, 1/2,2/0,2/1,2/2}$.
\end{thm}

\begin{proof}
The desired result cannot be obtained through the complementation operation on $X_i$'s. Nevertheless, we proceed following the same steps as in the proof of Theorem~\ref{thm-L-shape-1}. The equidistribution $\pattern{scale=0.5}{3}{1/1,2/2,3/3}{0/0,0/1,0/2,1/0,1/1, 1/2,2/0,2/1,2/2} \sim_d \hspace{-0.15cm} \pattern{scale=0.5}{3}{1/1,2/3,3/2}{0/0,0/1,0/2,1/0,1/1, 1/2,2/0,2/1,2/2}$ follows from the equidistribution of  $\pattern{scale = 0.5}{2}{1/1,2/2}{0/0,0/1,1/0,1/1}\sim_d\pattern{scale = 0.5}{2}{1/2,2/1}{0/0,0/1,1/0,1/1}$, by Theorem~\ref{equdis-2-patterns-length-2}, in each $A_i$ (we can use the respective bijection for the patterns of length 2 to permute the elements of $A_i$ in the same positions accordingly). \end{proof}

 \begin{thm}\label{9-2s}
We have $\pattern{scale=0.5}{3}{1/1,2/2,3/3}{0/0,0/1,0/2,1/0,1/1, 1/2,2/0,2/1,2/2,3/3} \sim_d \hspace{-0.15cm} \pattern{scale=0.5}{3}{1/1,2/3,3/2}{0/0,0/1,0/2,1/0,1/1, 1/2,2/0,2/1,2/2,3/3}$.
\end{thm}

\begin{proof}
Again, the desired result cannot be obtained through the complementation operation on $X_i$'s. However, we can proceed with the same steps as in the proof of Theorem~\ref{9-ls}. The equidistribution $\pattern{scale=0.5}{3}{1/1,2/2,3/3}{0/0,0/1,0/2,1/0,1/1, 1/2,2/0,2/1,2/2,3/3} \sim_d \hspace{-0.15cm} \pattern{scale=0.5}{3}{1/1,2/3,3/2}{0/0,0/1,0/2,1/0,1/1, 1/2,2/0,2/1,2/2,3/3}$ follows from our proof of Corollary~\ref{two-pat-length-2} of the equidistribution of  $\pattern{scale = 0.5}{2}{1/1,2/2}{0/0,0/1,1/0,1/1,2/2}\sim_d\hspace{-0.15cm} \pattern{scale = 0.5}{2}{1/2,2/1}{0/0,0/1,1/0,1/1,2/2}$ applied to  each $A_i$ (the permuted elements in $A_i$ collectively occupy the same positions and have the same values, so conflict like creating/removing an extra occurrence can occur with the shaded box (2,2) giving restrictions outside of $A_i$). 
\end{proof}

To illustrate the proof of Theorem~\ref{9-2s}, see the example in Figure~\ref{fig-3-4} giving the permutation $\pi=(10)47986152$ with $x_1=4, x_2=1$. We can see that $A_1=7986$ and $A_2=23$ in $\pi$ are replaced, respectively, by $A_1'=9867$ and $A_2'=32$ in $\pi'$, by doing the same procedure as Corollary~\ref{two-pat-length-2}. Note that the occurrences 479, 478 and 123 of $\pattern{scale=0.5}{3}{1/1,2/2,3/3}{0/0,0/1,0/2,1/0,1/1, 1/2,2/0,2/1,2/2,3/3}$  in $\pi$ are replaced, respectively, by the occurrences 498, 486 and 132 of $\pattern{scale=0.5}{3}{1/1,2/3,3/2}{0/0,0/1,0/2,1/0,1/1, 1/2,2/0,2/1,2/2,3/3}$ in $\pi'$, while the occurrence 476 of $\pattern{scale=0.5}{3}{1/1,2/3,3/2}{0/0,0/1,0/2,1/0,1/1, 1/2,2/0,2/1,2/2,3/3}$ in $\pi$ is replaced by the occurrence 467 of $\pattern{scale=0.5}{3}{1/1,2/2,3/3}{0/0,0/1,0/2,1/0,1/1, 1/2,2/0,2/1,2/2,3/3}$ in $\pi'$.

\begin{figure}
	\begin{center}
		
		\begin{tabular}{ccc}
			
			\begin{tikzpicture}[scale=0.25] 
			\tikzset{    
				grid/.style={      
					draw,      
					step=1cm,      
					gray!100,     
					very thin,      
				}, 
				cell/.style={    
					draw,    
					anchor=center,  
					text centered,    
				},  
				graycell/.style={ 
					fill=gray!40,   
					draw=none,   
					minimum width=1cm, 
					minimum height=1cm,   
					anchor=south west,   
				}
			}  
			
			\draw[grid] (0,0) grid (9,9);  
			%
			\filldraw[black] (0,9) circle (8pt);
  \draw[thick] (1,3) circle (8pt); \draw[thick] (1,3) circle (2pt); 
  \draw[thick] (2,6) circle (8pt);  \filldraw[white] (2,6) circle (6pt);
  \draw[thick] (3,8) circle (8pt);  \filldraw[white] (3,8) circle (6pt);
  \draw[thick] (4,7) circle (8pt);  \filldraw[white] (4,7) circle (6pt);
  \draw[thick] (5,5) circle (8pt);  \filldraw[white] (5,5) circle (6pt);
  \draw[thick] (6,0) circle (8pt);  \draw[thick] (6,0) circle (2pt);
  \filldraw[black] (7,4) circle (8pt);  
  \draw[thick] (8,1) circle (8pt); \filldraw[white] (8,1) circle (6pt); 
  \draw[thick] (9,2) circle (8pt); \filldraw[white] (9,2) circle (6pt);

    \draw (0.5,2.5)--(0.5,8.5)--(5.4,8.5)--(5.4,2.5)--(0.5,2.5);  
     \draw (5.6,-0.5)--(5.6,3.5)--(9.5,3.5)--(9.5,-0.5)--(5.6,-0.5);

			\end{tikzpicture} 
			
			& \ & 
			\begin{tikzpicture}[scale=0.25] 
			\tikzset{    
				grid/.style={      
					draw,      
					step=1cm,      
					gray!100,     
					very thin,      
				}, 
				cell/.style={    
					draw,    
					anchor=center,  
					text centered,    
				},  
				graycell/.style={ 
					fill=gray!40,   
					draw=none,   
					minimum width=1cm, 
					minimum height=1cm,   
					anchor=south west,   
				}
			}  
			
			\draw[grid] (0,0) grid (9,9);  
			%
			\filldraw[black] (0,9) circle (8pt);
  \draw[thick] (1,3) circle (8pt); \draw[thick] (1,3) circle (2pt); 
  \draw[thick] (2,8) circle (8pt);  \filldraw[white] (2,8) circle (6pt);
  \draw[thick] (3,7) circle (8pt);  \filldraw[white] (3,7) circle (6pt);
  \draw[thick] (4,5) circle (8pt);  \filldraw[white] (4,5) circle (6pt);
  \draw[thick] (5,6) circle (8pt);  \filldraw[white] (5,6) circle (6pt);
  \draw[thick] (6,0) circle (8pt);  \draw[thick] (6,0) circle (2pt);
  \filldraw[black] (7,4) circle (8pt);  
  \draw[thick] (8,2) circle (8pt); \filldraw[white] (8,2) circle (6pt); 
  \draw[thick] (9,1) circle (8pt); \filldraw[white] (9,1) circle (6pt);

    \draw (0.5,2.5)--(0.5,8.5)--(5.4,8.5)--(5.4,2.5)--(0.5,2.5);  
     \draw (5.6,-0.5)--(5.6,3.5)--(9.5,3.5)--(9.5,-0.5)--(5.6,-0.5);

			\end{tikzpicture}
		\end{tabular}
	\end{center}
	
	\vspace{-0.5cm}
	
	\caption{Permutations $\pi=(10)479861523$ (to the left) and $\pi'=(10)498671532$ (to the right) illustrating the proof of Theorem~\ref{9-2s}. The circled dots represent the first elements of occurrences of the patterns (certain left-to-right minima in $\pi$), and the white dots represent the second and third elements of occurrences of one of the two patterns.}\label{fig-3-4}
\end{figure}
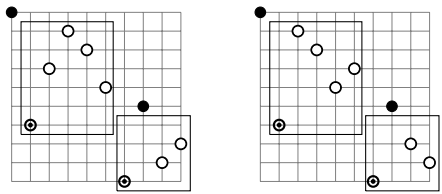

\begin{thm}\label{thm-L-shape-2}
We have $\pattern{scale = 0.5}{3}{1/1,2/2,3/3}{0/0,0/2,2/0} \sim_d \hspace{-0.15cm} \pattern{scale = 0.5}{3}{1/1,2/3,3/2}{0/0,0/2,2/0}$.
\end{thm}

\begin{proof}
The differences from the proof of Theorem~\ref{thm-L-shape} are that now 
\begin{itemize}
\item $x_i$ does not have to be the only element in $\pi$ in the South-West direction relative to $A_i$, allowing multiple left-to-right minima to potentially be the starting points of occurrences of the patterns for elements within the same $A_i$;
\item $A_i=X_i$; specifically, $A_i$'s can overlap vertically, and not all elements within an $A_i$ need to be involved in occurrences of the patterns (some $A_i$'s, even those of length greater than 2, might not contribute any occurrences of the patterns). However, the second and third elements in an occurrence of a pattern cannot be in different $A_i$'s (as is the case in the proof of Theorem~\ref{thm-L-shape}).
\end{itemize} 
Refer to Figure~\ref{fig-5} for a schematic structure of a permutation $\pi$.   The observations above do not affect our arguments in the proof of Theorem~\ref{thm-L-shape},  since complementation of $A_i$'s swaps occurrences of $\pattern{scale = 0.5}{3}{1/1,2/2,3/3}{0/0,0/2,2/0}$ and $\pattern{scale = 0.5}{3}{1/1,2/3,3/2}{0/0,0/2,2/0}$, and does not introduce any new occurrences of these patterns. 
\end{proof}

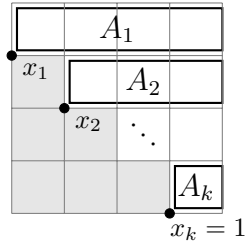
\begin{figure}  
\begin{center}  
\begin{tikzpicture}[scale=0.7]    
\tikzset{      
   grid/.style={        
      draw,        
      step=1cm,        
      gray!100,    
      thin,        
    },   
    cell/.style={      
      draw,      
      anchor=center,    
      text centered,    
     },    
    graycell/.style={   
      fill=gray!20,
      draw=none,     
      minimum width=1cm,   
      minimum height=1cm,     
      anchor=south west,            
    }    
  } 
      
  \fill[graycell] (0,0) rectangle (1,1);
  \fill[graycell] (0,1) rectangle (1,2);
  \fill[graycell] (0,2) rectangle (1,3);
  \fill[graycell] (1,0) rectangle (2,1);
  \fill[graycell] (1,1) rectangle (2,2);
  \fill[graycell] (2,0) rectangle (3,1);

\draw[thick] (0.1,3.1) rectangle (4,3.9);
 \draw[thick] (1.1,2.1) rectangle (4,2.9);
\draw[thick] (3.1,0.1) rectangle (4,0.9);
  
 \draw[grid] (0,0) grid (4,4);  

  \node at (2,3.5) {$A_1$};  
  \node at (2.5,2.5) {$A_2$};  
  \node at (3.5,0.5) {$A_k$};
  \node[anchor=west] at (0,2.7) {\footnotesize{$x_1$}}; 
  \node[anchor=west] at (1,1.7) {\footnotesize{$x_2$}}; 
  \node[anchor=west] at (2.8,-0.3) {\footnotesize{$x_k=1$}}; 
 
  \filldraw[black] (0,3) circle (2.5pt);  
  \filldraw[black] (1,2) circle (2.5pt); 
  \filldraw[black] (3,0) circle (2.5pt);  
 
 \node[anchor=center, rotate=-45] at (2.5,1.5) {$\ldots$};
\end{tikzpicture} 

\vspace{-0.8cm}

\end{center}
\caption{The structure of permutations in the proof of Theorem~\ref{thm-L-shape-2}.}\label{fig-5}
\end{figure}

In the following theorem, we once more utilize the fact that the shading in the North-East area, relative to the minimum element in a pattern, exhibits symmetry with respect to a horizontal line.

\begin{thm}\label{thm-L-shape-3}
We have $\pattern{scale = 0.5}{3}{1/1,2/2,3/3}{0/0,0/2,2/0,2/2} \sim_d \pattern{scale = 0.5}{3}{1/1,2/3,3/2}{0/0,0/2,2/0,2/2}$, 
$\pattern{scale=0.5}{3}{1/1,2/2,3/3}{0/0,0/2,1/2, 2/0,2/1,2/3,3/2} \sim_d \hspace{-0.15cm} \pattern{scale=0.5}{3}{1/1,2/3,3/2}{0/0,0/2,1/2, 2/0,2/1,2/3,3/2}$,
$\pattern{scale=0.5}{3}{1/1,2/2,3/3}{0/0,0/2,1/1, 1/3,2/0,3/1,3/3} \sim_d \hspace{-0.15cm} \pattern{scale=0.5}{3}{1/1,2/3,3/2}{0/0,0/2,1/1, 1/3,2/0,3/1,3/3}$, \\
$\pattern{scale=0.5}{3}{1/1,2/2,3/3}{0/0,0/2,1/1, 1/3,2/0,2/2,3/1,3/3} \sim_d \hspace{-0.15cm} \pattern{scale=0.5}{3}{1/1,2/3,3/2}{0/0,0/2,1/1, 1/3,2/0,2/2,3/1,3/3}$,
$\pattern{scale=0.5}{3}{1/1,2/2,3/3}{0/0,0/2,1/2, 2/0,2/1,2/2,2/3,3/2} \sim_d \hspace{-0.15cm} \pattern{scale=0.5}{3}{1/1,2/3,3/2}{0/0,0/2,1/2, 2/0,2/1,2/2,2/3,3/2}$, 
$\pattern{scale=0.5}{3}{1/1,2/2,3/3}{0/0,0/2,1/1, 1/2,1/3,2/0,2/1,2/3,3/1,3/2,3/3} \sim_d \hspace{-0.15cm} \pattern{scale=0.5}{3}{1/1,2/3,3/2}{0/0,0/2,1/1, 1/2,1/3,2/0,2/1,2/3,3/1,3/2,3/3}$, and $\pattern{scale=0.5}{3}{1/1,2/2,3/3}{0/0,0/2,1/1,1/2,1/3, 2/0,2/1,2/2,2/3,3/1,3/2,3/3}\pattern{scale=0.5}{3}{1/1,2/3,3/2}{0/0,0/2,1/1,1/2,1/3, 2/0,2/1,2/2,2/3,3/1,3/2,3/3}$.
\end{thm}

\begin{proof}
Our arguments here are analogous to those in Theorem~\ref{thm-L-shape-1}, but the structure of permutations corresponds to the one considered in Theorem~\ref{thm-L-shape-2} and depicted in Figure~\ref{fig-5}.
\end{proof}

%

Shading additional boxes $(0,3)$ and $(3,0)$ in the patterns from Theorem~\ref{thm-L-shape-3}, we observe through computer experiments that $\pattern{scale = 0.5}{3}{1/1,2/2,3/3}{0/0,0/2,0/3,2/0,2/2,3/0} \not\sim_d \pattern{scale = 0.5}{3}{1/1,2/3,3/2}{0/0,0/2,0/3,2/0,2/2,3/0}$, $\pattern{scale=0.5}{3}{1/1,2/2,3/3}{0/0,0/2,0/3,1/1, 1/3,2/0,3/0,3/1,3/3} \not\sim_d \hspace{-0.15cm} \pattern{scale=0.5}{3}{1/1,2/3,3/2}{0/0,0/2,0/3,1/1, 1/3,2/0,3/0,3/1,3/3}$, and $\pattern{scale=0.5}{3}{1/1,2/2,3/3}{0/0,0/2,0/3,1/1, 1/3,2/0,2/2,3/0,3/1,3/3} \not\sim_d \hspace{-0.15cm} \pattern{scale=0.5}{3}{1/1,2/3,3/2}{0/0,0/2,0/3,1/1, 1/3,2/0,2/2,3/0,3/1,3/3}$. Additionally, $\pattern{scale=0.5}{3}{1/1,2/2,3/3}{0/0, 0/2,0/3,2/0,3/0}\not\sim_d\pattern{scale=0.5}{3}{1/1,2/3,3/2}{0/0, 0/2,0/3,2/0,3/0}$. As mentioned above, $\pattern{scale=0.5}{3}{1/1,2/2,3/3}{0/0, 0/2,0/3,2/0,3/0}$  and $\pattern{scale=0.5}{3}{1/1,2/3,3/2}{0/0, 0/2,0/3,2/0,3/0}$ (resp., $\pattern{scale=0.5}{3}{1/1,2/2,3/3}{0/0, 0/2,0/3,2/0,2/2,3/0}$  and $\pattern{scale=0.5}{3}{1/1,2/3,3/2}{0/0, 0/2,0/3,2/0,2/2,3/0}$) are equidistributed for $n\leq 7$, but not for $n=8$. However, other patterns in Theorem~\ref{thm-L-shape-3} do have equidistribution, as stated in the next theorem. Note that a similar argument as that for pair 10 can be applied to prove $\pattern{scale=0.5}{3}{1/1,2/2,3/3}{0/0,0/2,0/3,1/1,1/2,1/3, 2/0,2/1,2/2,2/3,3/0,3/1,3/2,3/3}\sim_d \hspace{-0.15cm} \pattern{scale=0.5}{3}{1/1,2/3,3/2}{0/0,0/2,0/3,1/1,1/2,1/3, 2/0,2/1,2/2,2/3,3/0,3/1,3/2,3/3}$. Finally, we refer to Remark~\ref{subtle-rem} explaining why the complement does not work for all patterns in Theorem~\ref{thm-L-shape-3} when boxes (0,3) and (3,0) are shaded. 

\begin{figure}[t]  
\begin{center}  

\begin{tabular}{cc}

\begin{tikzpicture}[scale=0.7]  
\tikzset{      
   grid/.style={        
      draw,        
      step=1cm,        
      gray!100,    
      thin,        
    },   
    cell/.style={      
      draw,      
      anchor=center,    
      text centered,    
     },    
    graycell/.style={   
      fill=gray!20,
      draw=none,     
      minimum width=1cm,   
      minimum height=1cm,     
      anchor=south west,            
    }    
  } 
  
   \fill[graycell] (0,0) rectangle (1,1);  
    \fill[graycell] (0,2) rectangle (1,4);    
   \fill[graycell] (2,0) rectangle (4,1); 
  
   \fill[graycell] (2,1) rectangle (3,2); 
  \fill[graycell] (1,2) rectangle (2,3); 
  \fill[graycell] (3,2) rectangle (4,3); 
  \fill[graycell] (2,3) rectangle (3,4); 
 
 \draw[grid] (0,0) grid (4,4);  

  \node at (2.5,2.5) {$A_1$};  
   \node[anchor=west] at (1,0.8) {\footnotesize{$x_1$}}; 
  
  \node[cell,scale=2.6] at (2.5,2.5) {};  
 
  \filldraw[black] (1,1) circle (2pt);  
 \end{tikzpicture} 
 
 &
 
 \begin{tikzpicture}[scale=0.7]  
\tikzset{      
   grid/.style={        
      draw,        
      step=1cm,        
      gray!100,    
      thin,        
    },   
    cell/.style={      
      draw,      
      anchor=center,    
      text centered,    
     },    
    graycell/.style={   
      fill=gray!20,
      draw=none,     
      minimum width=1cm,   
      minimum height=1cm,     
      anchor=south west,            
    }    
  } 
  
     \fill[graycell] (0,0) rectangle (1,1);  
    \fill[graycell] (0,2) rectangle (1,4);    
   \fill[graycell] (2,0) rectangle (4,1); 
  
     \fill[graycell] (1,1) rectangle (4,2); 
  \fill[graycell] (1,2) rectangle (2,3); 
  \fill[graycell] (3,2) rectangle (4,3); 
  \fill[graycell] (1,3) rectangle (4,4); 

 \draw[grid] (0,0) grid (4,4);  

  \node at (2.5,2.5) {$A_1$};  
   \node[anchor=west] at (1,0.8) {\footnotesize{$x_1$}}; 
  
  \node[cell,scale=2.6] at (2.5,2.5) {};  
 
  \filldraw[black] (1,1) circle (2pt);  
 \end{tikzpicture} 

%
%
%
%
%
%
%

 \end{tabular}

\vspace{-0.5cm}

\end{center}
\caption{Permutations in the proof of Theorem~\ref{thm-L-shape-3-(0,3)-shaded} corresponding to the respective patterns. $A_1$ is of length at least 2.}\label{fig-6}
\end{figure}
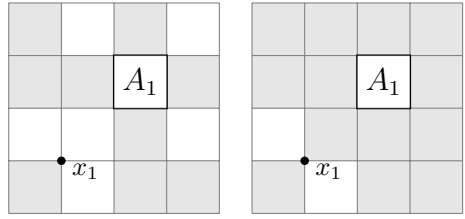

\begin{thm}\label{thm-L-shape-3-(0,3)-shaded}
We have 
$\pattern{scale=0.5}{3}{1/1,2/2,3/3}{0/0,0/2,0/3,1/2, 2/0,2/1,2/3,3/0,3/2} \sim_d \hspace{-0.15cm} \pattern{scale=0.5}{3}{1/1,2/3,3/2}{0/0,0/2,0/3,1/2, 2/0,2/1,2/3,3/0,3/2}$,
$\pattern{scale=0.5}{3}{1/1,2/2,3/3}{0/0,0/2,0/3,1/2, 2/0,2/1,2/2,2/3,3/0,3/2} \sim_d \hspace{-0.15cm} \pattern{scale=0.5}{3}{1/1,2/3,3/2}{0/0,0/2,0/3,1/2, 2/0,2/1,2/2,2/3,3/0,3/2}$, 
$\pattern{scale=0.5}{3}{1/1,2/2,3/3}{0/0,0/2,0/3,1/1, 1/2,1/3,2/0,2/1,2/3,3/0,3/1,3/2,3/3} \sim_d \hspace{-0.15cm} \pattern{scale=0.5}{3}{1/1,2/3,3/2}{0/0,0/2,0/3,1/1, 1/2,1/3,2/0,2/1,2/3,3/0,3/1,3/2,3/3}$, and $\pattern{scale=0.5}{3}{1/1,2/2,3/3}{0/0,0/2,0/3,1/1,1/2,1/3, 2/0,2/1,2/2,2/3,3/0,3/1,3/2,3/3}\sim_d \hspace{-0.15cm} \pattern{scale=0.5}{3}{1/1,2/3,3/2}{0/0,0/2,0/3,1/1,1/2,1/3, 2/0,2/1,2/2,2/3,3/0,3/1,3/2,3/3}$.
\end{thm}

\begin{proof}
Our arguments here are analogous to those in Theorem~\ref{thm-L-shape-3}, but in this case we have at most one $A_i$, as depicted in Figure~\ref{fig-6}: the picture to the left (resp., to the right) corresponds to pairs 41 and 42 (resp., 43 and 45). In both pictures, we assume $x_1$ to be the largest left-to-right minima starting occurrences of the patterns.  Note that we need to unshade the box adjacent to $x_1$ in its bottom-right corner in the leftmost picture because the box (3,0) in the patterns is shaded. Thus, the unshaded area can contain left-to-right minima not involved in occurrences of the patterns. Recall that if $A_1$ has length at most 1, such a permutation will have no occurrences of the patterns and is mapped to itself. Otherwise, by complementing $A_1$ (not $X_1$ as is done previously!), we obtain the desired result.
\end{proof}

\begin{figure}
\begin{center}

\begin{tabular}{ccc}

\begin{tikzpicture}[scale=0.25] 
\tikzset{    
   grid/.style={      
      draw,      
      step=1cm,      
      gray!100,     
      very thin,      
    }, 
    cell/.style={    
      draw,    
      anchor=center,  
      text centered,    
     },  
    graycell/.style={ 
      fill=gray!40,   
      draw=none,   
      minimum width=1cm, 
      minimum height=1cm,   
      anchor=south west,   
    }
  }  
    
  \draw[grid] (0,0) grid (7,7);  

\draw[thick] (0,1) circle (8pt); \draw[thick] (0,1) circle (2pt); 
\filldraw[black] (1,4) circle (8pt); 
\draw[thick] (2,0) circle (8pt); \draw[thick] (2,0) circle (2pt); 
\draw[thick] (3,6) circle (8pt);  \filldraw[white] (3,6) circle (6pt); 
\draw[thick] (4,7) circle (8pt);  \filldraw[white] (4,7) circle (6pt); 
\draw[thick] (5,5) circle (8pt);  \filldraw[white] (5,5) circle (6pt); 
\draw[thick] (6,3) circle (8pt);  \filldraw[white] (6,3) circle (6pt); 
\draw[thick] (7,2) circle (8pt);  \filldraw[white] (7,2) circle (6pt); 

\draw (2.5,1.4) rectangle (7.5,7.5);
  
\end{tikzpicture} 

& \ & 

\begin{tikzpicture}[scale=0.25] 
\tikzset{    
   grid/.style={      
      draw,      
      step=1cm,      
      gray!100,     
      very thin,      
    }, 
    cell/.style={    
      draw,    
      anchor=center,  
      text centered,    
     },  
    graycell/.style={ 
      fill=gray!40,   
      draw=none,   
      minimum width=1cm, 
      minimum height=1cm,   
      anchor=south west,   
    }
  }  
    
  \draw[grid] (0,0) grid (7,7);  

\draw[thick] (0,1) circle (8pt); \draw[thick] (0,1) circle (2pt); 
\filldraw[black] (1,4) circle (8pt); 
\draw[thick] (2,0) circle (8pt); \draw[thick] (2,0) circle (2pt); 
\draw[thick] (3,3) circle (8pt);  \filldraw[white] (3,3) circle (6pt); 
\draw[thick] (4,2) circle (8pt);  \filldraw[white] (4,2) circle (6pt); 
\draw[thick] (5,5) circle (8pt);  \filldraw[white] (5,5) circle (6pt); 
\draw[thick] (6,6) circle (8pt);  \filldraw[white] (6,6) circle (6pt); 
\draw[thick] (7,7) circle (8pt);  \filldraw[white] (7,7) circle (6pt); 

\draw (2.5,1.4) rectangle (7.5,7.5);
  
\end{tikzpicture}

\end{tabular}
\end{center}

\vspace{-0.5cm}

\caption{Permutations $\pi$ (to the left) and $\pi'$ (to the right) supporting Remark~\ref{subtle-rem}. The circled dots represent the left-to-right minima, and the white dots represent the elements in $A_1$.}\label{fig-subtle-remark}
\end{figure}
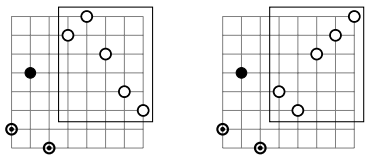

\begin{remark}\label{subtle-rem} It is not entirely clear why the complement does not work for all patterns in Theorem~\ref{thm-L-shape-3} when boxes $(0,3)$ and $(3,0)$ are shaded. The reason for that is rather subtle and lies in the fact that, in the structures in Figure~\ref{fig-6}, $A_1$ must be  above (and to the right) of the only unshaded square in the first column for the complement to work. Otherwise, occurrences of the patterns may become non-occurrences. For example, for the pair $\pattern{scale = 0.5}{3}{1/1,2/2,3/3}{0/0,0/2,0/3,2/0,2/2,3/0}\pattern{scale = 0.5}{3}{1/1,2/3,3/2}{0/0,0/2,0/3,2/0,2/2,3/0}$, a counterexample of length $n=8$ demonstrating this phenomenon is the permutation $\pi=25178643$ mapped to the permutation $\pi'=25143678$; see Figure~\ref{fig-subtle-remark} for the respective permutation diagrams. Here, $A_1=78643$ and the total number of occurrences of the patterns beginning with $x_1=2$ (and defining $A_1$) is preserved. However, $\pi$ has  three occurrences of the patterns beginning with $x_2=1$, namely, $178$, $176$ and $186$, while $\pi'$ has two occurrences of the patterns beginning with $x_2=1$, namely, $167$ and $178$.
\end{remark}

\begin{thm}\label{sim-L-2-patterns}
We have $\pattern{scale=0.5}{3}{1/1,2/2,3/3}{0/0,2/0,1/1,2/1,3/1,0/2,1/2,2/2,3/2,1/3,2/3} \sim_d \hspace{-0.15cm} \pattern{scale=0.5}{3}{1/1,2/3,3/2}{0/0,2/0,1/1,2/1,3/1,0/2,1/2,2/2,3/2,1/3,2/3}$ and $\pattern{scale=0.5}{3}{1/1,2/2,3/3}{0/0,2/0,3/0,1/1,2/1,3/1,0/2,1/2,2/2,3/2,0/3,1/3,2/3} \sim_d \hspace{-0.15cm} \pattern{scale=0.5}{3}{1/1,2/3,3/2}{0/0,2/0,3/0,1/1,2/1,3/1,0/2,1/2,2/2,3/2,0/3,1/3,2/3}$.
\end{thm}

\begin{proof}
Considering the structure of permutations in Theorem~\ref{thm-L-shape-2}, we observe that for the given patterns, if $ab$ is the second and third elements in an occurrence of a pattern in a permutation $\pi=\pi_1\pi_2\ldots\pi_n$, then $ab=\pi_i\pi_{i+1}=j(j+1)$ for some $i$ and $j$, and $j$ and $j+1$ are two smallest elements in their $A_i$ (in particular, there is only $A_1$ for the second pair of patterns), so that each $A_i$ has at most one occurrence of a pattern, and neither the second nor the third elements in an occurrence can be part of another occurrence. The involution that replaces occurrences of the patterns is then given by swapping $ab$ by $ba$ for every occurrence.
\end{proof}

\section{Cases based on the proof of Theorem~\ref{9-box-rs}}\label{cor-box-sec}

In this section, we explain equidistributions of pairs of patterns in Table~\ref{tab-6}.

\begin{table}[t]
 	{
 		\renewcommand{\arraystretch}{1.5}
 \begin{center} 
 		\begin{tabular}{|c|c||c|c||c|c||c|c|}
 			\hline
 			
 	\footnotesize{nr.}	& {\footnotesize patterns}  & \footnotesize{nr.}	& {\footnotesize patterns}  & \footnotesize{nr.}	& {\footnotesize patterns} & \footnotesize{nr.}	& {\footnotesize patterns}   \\ 
 			\hline
 			\hline	
\footnotesize{46}  & $\pattern{scale = 0.5}{3}{1/1,2/2,3/3}{1/1, 1/2,1/3,2/1,2/2,2/3,3/1,3/2,3/3}\pattern{scale = 0.5}{3}{1/1,2/3,3/2}{1/1, 1/2,1/3,2/1,2/2,2/3,3/1,3/2,3/3}$	  & 
\footnotesize{47}  & $\pattern{scale = 0.5}{3}{1/1,2/2,3/3}{0/0,1/1, 1/2,1/3,2/1,2/2,2/3,3/1,3/2,3/3}\pattern{scale = 0.5}{3}{1/1,2/3,3/2}{0/0,1/1, 1/2,1/3,2/1,2/2,2/3,3/1,3/2,3/3}$ 	& 
\footnotesize{48}  & $\pattern{scale = 0.5}{3}{1/1,2/2,3/3}{0/1,1/0,1/1, 1/2,1/3,2/1,2/2,2/3,3/1,3/2,3/3}\pattern{scale = 0.5}{3}{1/1,2/3,3/2}{0/1,1/0,1/1, 1/2,1/3,2/1,2/2,2/3,3/1,3/2,3/3}$ 	&
\footnotesize{49}  & $\pattern{scale = 0.5}{3}{1/1,2/2,3/3}{0/0,0/1,1/0,1/1, 1/2,1/3,2/1,2/2,2/3,3/1,3/2,3/3}\pattern{scale = 0.5}{3}{1/1,2/3,3/2}{0/0,0/1,1/0,1/1, 1/2,1/3,2/1,2/2,2/3,3/1,3/2,3/3}$ \\[5pt]
 			\hline	
\footnotesize{50}  & 	$\pattern{scale = 0.5}{3}{1/1,2/2,3/3}{0/2,1/1, 1/2,1/3,2/0,2/1,2/2,2/3,3/1,3/2,3/3}\pattern{scale = 0.5}{3}{1/1,2/3,3/2}{0/2,1/1, 1/2,1/3,2/0,2/1,2/2,2/3,3/1,3/2,3/3}$  & 
\footnotesize{51}  & 	$\pattern{scale = 0.5}{3}{1/1,2/2,3/3}{0/0,0/2,1/1, 1/2,1/3,2/0,2/1,2/2,2/3,3/1,3/2,3/3}\pattern{scale = 0.5}{3}{1/1,2/3,3/2}{0/0,0/2,1/1, 1/2,1/3,2/0,2/1,2/2,2/3,3/1,3/2,3/3}$  & 
\footnotesize{8}  &  $\pattern{scale = 0.5}{3}{1/1,2/2,3/3}{0/1,0/2,1/0,1/1, 1/2,1/3,2/0,2/1,2/2,2/3,3/1,3/2,3/3}\pattern{scale = 0.5}{3}{1/1,2/3,3/2}{0/1,0/2,1/0,1/1, 1/2,1/3,2/0,2/1,2/2,2/3,3/1,3/2,3/3}$  & 
\footnotesize{6}  & $\pattern{scale = 0.5}{3}{1/1,2/2,3/3}{0/0,0/1,0/2,1/0,1/1, 1/2,1/3,2/0,2/1,2/2,2/3,3/1,3/2,3/3}\pattern{scale = 0.5}{3}{1/1,2/3,3/2}{0/0,0/1,0/2,1/0,1/1, 1/2,1/3,2/0,2/1,2/2,2/3,3/1,3/2,3/3}$   \\[5pt]
 			\hline	
\footnotesize{52}  &  $\pattern{scale = 0.5}{3}{1/1,2/2,3/3}{2/2,2/3, 3/2,3/3}\pattern{scale = 0.5}{3}{1/1,2/3,3/2}{2/2,2/3, 3/2,3/3}$	 & 
\footnotesize{53}  & $\pattern{scale = 0.5}{3}{1/1,2/2,3/3}{0/0,2/2,2/3, 3/2,3/3}\pattern{scale = 0.5}{3}{1/1,2/3,3/2}{0/0,2/2,2/3, 3/2,3/3}$ 	& 
\footnotesize{54}  & $\pattern{scale = 0.5}{3}{1/1,2/2,3/3}{0/1,1/0,2/2,2/3, 3/2,3/3}\pattern{scale = 0.5}{3}{1/1,2/3,3/2}{0/1,1/0,2/2,2/3, 3/2,3/3}$ 	& 
\footnotesize{55}  & $\pattern{scale = 0.5}{3}{1/1,2/2,3/3}{0/0,0/1,1/0,2/2,2/3, 3/2,3/3}\pattern{scale = 0.5}{3}{1/1,2/3,3/2}{0/0,0/1,1/0,2/2,2/3, 3/2,3/3}$  \\[5pt]
 			\hline	
\footnotesize{56}  &  $\pattern{scale = 0.5}{3}{1/1,2/2,3/3}{1/1,2/2,2/3, 3/2,3/3}\pattern{scale = 0.5}{3}{1/1,2/3,3/2}{1/1,2/2,2/3, 3/2,3/3}$	 & 
\footnotesize{57}  & $\pattern{scale = 0.5}{3}{1/1,2/2,3/3}{0/0,1/1,2/2,2/3, 3/2,3/3}\pattern{scale = 0.5}{3}{1/1,2/3,3/2}{0/0,1/1,2/2,2/3, 3/2,3/3}$ 	& 
\footnotesize{58}  & $\pattern{scale = 0.5}{3}{1/1,2/2,3/3}{0/1,1/0,1/1,2/2,2/3, 3/2,3/3}\pattern{scale = 0.5}{3}{1/1,2/3,3/2}{0/1,1/0,1/1,2/2,2/3, 3/2,3/3}$ 	& 
\footnotesize{59}  &  $\pattern{scale = 0.5}{3}{1/1,2/2,3/3}{0/0,0/1,1/0,1/1,2/2,2/3, 3/2,3/3}\pattern{scale = 0.5}{3}{1/1,2/3,3/2}{0/0,0/1,1/0,1/1,2/2,2/3, 3/2,3/3}$  \\[5pt]
 			\hline	
\footnotesize{60}  &  $\pattern{scale = 0.5}{3}{1/1,2/2,3/3}{1/2,1/3,2/1,2/2,2/3,3/1,3/2,3/3}\pattern{scale = 0.5}{3}{1/1,2/3,3/2}{1/2,1/3,2/1,2/2,2/3,3/1,3/2,3/3}$	 & 
\footnotesize{61}  & $\pattern{scale = 0.5}{3}{1/1,2/2,3/3}{0/0,1/2,1/3,2/1,2/2,2/3,3/1,3/2,3/3}\pattern{scale = 0.5}{3}{1/1,2/3,3/2}{0/0,1/2,1/3,2/1,2/2,2/3,3/1,3/2,3/3}$ 	& 
\footnotesize{62}  & $\pattern{scale = 0.5}{3}{1/1,2/2,3/3}{0/1,1/0,1/2,1/3,2/1,2/2,2/3,3/1,3/2,3/3}\pattern{scale = 0.5}{3}{1/1,2/3,3/2}{0/1,1/0,1/2,1/3,2/1,2/2,2/3,3/1, 3/2,3/3}$ 	& 
\footnotesize{63}  &  $\pattern{scale = 0.5}{3}{1/1,2/2,3/3}{0/0,0/1,1/0,1/2,1/3,2/1,2/2,2/3,3/1, 3/2,3/3}\pattern{scale = 0.5}{3}{1/1,2/3,3/2}{0/0,0/1,1/0,1/2,1/3,2/1,2/2,2/3,3/1,3/2,3/3}$  \\[5pt]
 			\hline
\footnotesize{64}  &  $\pattern{scale = 0.5}{3}{1/1,2/2,3/3}{0/2,1/2,1/3,2/0,2/1,2/2,2/3,3/1 ,3/2,3/3}\pattern{scale = 0.5}{3}{1/1,2/3,3/2}{0/2,1/2,1/3,2/0,2/1,2/2,2/3,3/1,3/2,3/3}$	 & 
\footnotesize{65}  & $\pattern{scale = 0.5}{3}{1/1,2/2,3/3}{0/1,0/2,1/2,1/3,1/0,2/0,2/1,2/2,2/3,3/1,3/2,3/3}\pattern{scale = 0.5}{3}{1/1,2/3,3/2}{0/1,0/2,1/2,1/3,1/0,2/0,2/1,2/2,2/3,3/1,3/2,3/3}$ 	
& 
\footnotesize{66} & $\pattern{scale = 0.5}{3}{1/1,2/2,3/3}{0/0,0/2,1/2,1/3,2/0,2/1,2/2,2/3,3/1,3/2,3/3}\pattern{scale = 0.5}{3}{1/1,2/3,3/2}{0/0,0/2,1/2,1/3,2/0,2/1,2/2,2/3,3/1,3/2,3/3}$  & 
\footnotesize{67} & $\pattern{scale = 0.5}{3}{1/1,2/2,3/3}{0/0,0/1,0/2,1/2,1/3,1/0,2/0,2/1,2/2,2/3,3/1,3/2,3/3}\pattern{scale = 0.5}{3}{1/1,2/3,3/2}{0/0,0/1,0/2,1/2,1/3,1/0,2/0,2/1,2/2,2/3,3/1,3/2,3/3}$  \\[5pt]
 			\hline
\footnotesize{68}  &  $\pattern{scale = 0.5}{3}{1/1,2/2,3/3}{0/2,0/3,2/2,2/3, 2/0,3/0,3/2,3/3}\pattern{scale = 0.5}{3}{1/1,2/3,3/2}{0/2,0/3,2/2,2/3, 2/0,3/0,3/2,3/3}$	 & 
\footnotesize{69}  & $\pattern{scale = 0.5}{3}{1/1,2/2,3/3}{0/0,0/2,0/3,2/0,2/2,2/3, 3/0,3/2,3/3}\pattern{scale = 0.5}{3}{1/1,2/3,3/2}{0/0,0/2,0/3,2/0,2/2,2/3, 3/0,3/2,3/3}$	& 
\footnotesize{70} & $\pattern{scale = 0.5}{3}{1/1,2/2,3/3}{0/1,0/2,0/3,1/0,2/0,2/2,2/3, 3/0,3/2,3/3}\pattern{scale = 0.5}{3}{1/1,2/3,3/2}{0/1,0/2,0/3,1/0,2/0,2/2,2/3, 3/0,3/2,3/3}$ & 
\footnotesize{20} & $\pattern{scale = 0.5}{3}{1/1,2/2,3/3}{0/0,0/1,0/2,0/3,1/0,2/0,2/2,2/3, 3/0,3/2,3/3}\pattern{scale = 0.5}{3}{1/1,2/3,3/2}{0/0,0/1,0/2,0/3,1/0,2/0,2/2,2/3, 3/0,3/2,3/3}$ \\[5pt]
 			\hline
\footnotesize{71}  & $\pattern{scale = 0.5}{3}{1/1,2/2,3/3}{0/2,0/3,1/1,2/0,2/2,2/3, 3/0,3/2,3/3}\pattern{scale = 0.5}{3}{1/1,2/3,3/2}{0/2,0/3,1/1,2/0,2/2,2/3, 3/0,3/2,3/3}$	 & 
\footnotesize{72}  & $\pattern{scale = 0.5}{3}{1/1,2/2,3/3}{0/0,0/2,0/3,1/1,2/0,2/2,2/3, 3/0,3/2,3/3}\pattern{scale = 0.5}{3}{1/1,2/3,3/2}{0/0,0/2,0/3,1/1,2/0,2/2,2/3, 3/0,3/2,3/3}$	& 
\footnotesize{73} & $\pattern{scale = 0.5}{3}{1/1,2/2,3/3}{0/1,0/2,0/3,1/0,1/1,2/0,2/2,2/3, 3/0,3/2,3/3}\pattern{scale = 0.5}{3}{1/1,2/3,3/2}{0/1,0/2,0/3,1/0,1/1,2/0,2/2,2/3, 3/0,3/2,3/3}$ & 
\footnotesize{22} & $\pattern{scale = 0.5}{3}{1/1,2/2,3/3}{0/0,0/1,0/2,0/3,1/0,1/1,2/0,2/2,2/3, 3/0,3/2,3/3}\pattern{scale = 0.5}{3}{1/1,2/3,3/2}{0/0,0/1,0/2,0/3,1/0,1/1,2/0,2/2,2/3, 3/0,3/2,3/3}$  \\[5pt]
 			\hline
\footnotesize{74}  & $\pattern{scale = 0.5}{3}{1/1,2/2,3/3}{0/0,0/1,0/2,1/0,2/2,2/3, 2/0,3/2,3/3} \pattern{scale = 0.5}{3}{1/1,2/3,3/2}{0/0,0/1,0/2,1/0,2/2,2/3, 2/0,3/2,3/3}$ & 
\footnotesize{75}  &  $\pattern{scale = 0.5}{3}{1/1,2/2,3/3}{0/0,0/1,0/2,1/0,1/1,2/2,2/3, 2/0,3/2,3/3}\pattern{scale = 0.5}{3}{1/1,2/3,3/2}{0/0,0/1,0/2,1/0,1/1,2/2,2/3, 2/0,3/2,3/3}$ & 
& & 
&  \\[5pt]
 			\hline
	\end{tabular}
\end{center} 
}
\vspace{-0.5cm}
 	\caption{Equidistributions explainable based on the proof of Theorem~\ref{9-box-rs}. Pairs 6, 8, 20 and 22 have been added to the table for completeness. Pairs 30 (resp., 31) and 46 (resp., 47) have the same distributions by Theorem~\ref{an-equidistr-thm}.}\label{tab-6}
\end{table}

\begin{thm}\label{9-box-rs} We have $\pattern{scale=0.5}{3}{1/1,2/2,3/3}{1/1, 1/2,1/3,2/1,2/2,2/3,3/1,3/2,3/3} \sim_d \hspace{-0.15cm} \pattern{scale=0.5}{3}{1/1,2/3,3/2}{1/1, 1/2,1/3,2/1,2/2,2/3,3/1,3/2,3/3}$. \end{thm}

\begin{proof}

Let $\pi=\pi_1\ldots\pi_n \in S_n$, $p_1=\pattern{scale=0.5}{3}{1/1,2/2,3/3}{1/1, 1/2,1/3,2/1,2/2,2/3,3/1,3/2,3/3}$ and $p_2= \pattern{scale=0.5}{3}{1/1,2/3,3/2}{1/1, 1/2,1/3,2/1,2/2,2/3,3/1,3/2,3/3}$. Suppose $xab$ is an occurrence of $p_1$ or $p_2$ in $\pi$. We make the following observations.

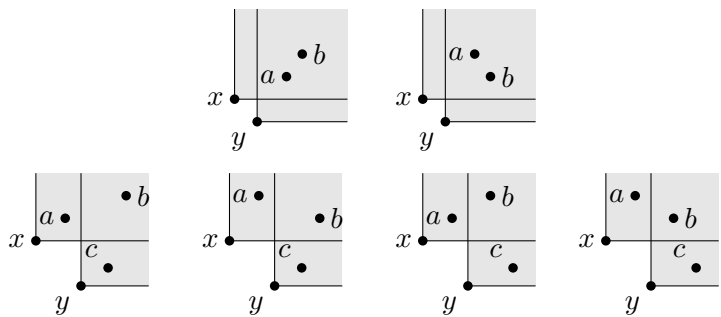
\begin{figure}
\begin{center}

\begin{tabular}{cccc}

&  
\begin{tikzpicture}[scale=0.3] 
\tikzset{      
   grid/.style={        
      draw,        
      step=1cm,        
      gray!100,    
      thin,        
    },   
    cell/.style={      
      draw,      
      anchor=center,    
      text centered,    
     },    
    graycell/.style={   
      fill=gray!20,
      draw=none,     
      minimum width=1cm,   
      minimum height=1cm,     
      anchor=south west,            
    }    
  } 
      
  \fill[graycell] (1,0) rectangle (5,5);
  \fill[graycell] (0,1) rectangle (5,5);

\draw (0,5)--(0,1)--(5,1);
\draw (1,5)--(1,0)--(5,0);
\coordinate[label=left:{\small $x$}] (x) at (0,1);
\coordinate[label=left:{\small $a$}] (a) at (2.3,2);
\coordinate[label=right:{\small $b$}] (b) at (3,3);
\coordinate[label=below left:{\small $y$}] (y) at (1,0);
\fill[] (x)circle(6pt);
\fill[] (a)circle(6pt);
\fill[] (b)circle(6pt);
\fill[] (y)circle(6pt);

\end{tikzpicture} 

&
\vspace{1cm}
\begin{tikzpicture}[scale=0.3] 
\tikzset{      
   grid/.style={        
      draw,        
      step=1cm,        
      gray!100,    
      thin,        
    },   
    cell/.style={      
      draw,      
      anchor=center,    
      text centered,    
     },    
    graycell/.style={   
      fill=gray!20,
      draw=none,     
      minimum width=1cm,   
      minimum height=1cm,     
      anchor=south west,            
    }    
  } 
      
  \fill[graycell] (1,0) rectangle (5,5);
  \fill[graycell] (0,1) rectangle (5,5);
\draw (0,5)--(0,1)--(5,1);
\draw (1,5)--(1,0)--(5,0);
\coordinate[label=left:{\small $x$}] (x) at (0,1);
\coordinate[label=left:{\small $a$}] (a) at (2.3,3);
\coordinate[label=right:{\small $b$}] (b) at (3,2);
\coordinate[label=below left:{\small $y$}] (y) at (1,0);

\fill[] (x)circle(6pt);
\fill[] (a)circle(6pt);
\fill[] (b)circle(6pt);
\fill[] (y)circle(6pt);

\end{tikzpicture}

& \\[-1cm]

\begin{tikzpicture}[scale=0.3] 
\tikzset{      
   grid/.style={        
      draw,        
      step=1cm,        
      gray!100,    
      thin,        
    },   
    cell/.style={      
      draw,      
      anchor=center,    
      text centered,    
     },    
    graycell/.style={   
      fill=gray!20,
      draw=none,     
      minimum width=1cm,   
      minimum height=1cm,     
      anchor=south west,            
    }    
  } 
      
  \fill[graycell] (2,0) rectangle (5,5);
  \fill[graycell] (0,2) rectangle (5,5);
\draw (0,5)--(0,2)--(5,2);
\draw (2,5)--(2,0)--(5,0);
\coordinate[label=left:{\small $x$}] (x) at (0,2);
\coordinate[label=left:{\small $a$}] (a) at (1.3,3);
\coordinate[label=right:{\small $b$}] (b) at (4,4);
\coordinate[label=below left:{\small $y$}] (y) at (2,0);
\coordinate[label=above left:{\small $c$}] (c) at (3.2,0.8);
\fill[] (x)circle(6pt);
\fill[] (a)circle(6pt);
\fill[] (b)circle(6pt);
\fill[] (y)circle(6pt);
\fill[] (c)circle(6pt);
\end{tikzpicture} 
& 
\begin{tikzpicture}[scale=0.3] 
\tikzset{      
   grid/.style={        
      draw,        
      step=1cm,        
      gray!100,    
      thin,        
    },   
    cell/.style={      
      draw,      
      anchor=center,    
      text centered,    
     },    
    graycell/.style={   
      fill=gray!20,
      draw=none,     
      minimum width=1cm,   
      minimum height=1cm,     
      anchor=south west,            
    }    
  } 
      
  \fill[graycell] (2,0) rectangle (5,5);
  \fill[graycell] (0,2) rectangle (5,5);
\draw (0,5)--(0,2)--(5,2);
\draw (2,5)--(2,0)--(5,0);
\coordinate[label=left:{\small $x$}] (x) at (0,2);
\coordinate[label=left:{\small $a$}] (a) at (1.3,4);
\coordinate[label=right:{\small $b$}] (b) at (4,3);
\coordinate[label=below left:{\small $y$}] (y) at (2,0);
\coordinate[label=above left:{\small $c$}] (c) at (3.2,0.8);
\fill[] (x)circle(6pt);
\fill[] (a)circle(6pt);
\fill[] (b)circle(6pt);
\fill[] (y)circle(6pt);
\fill[] (c)circle(6pt);
\end{tikzpicture}

& 
\begin{tikzpicture}[scale=0.3] 
\tikzset{      
   grid/.style={        
      draw,        
      step=1cm,        
      gray!100,    
      thin,        
    },   
    cell/.style={      
      draw,      
      anchor=center,    
      text centered,    
     },    
    graycell/.style={   
      fill=gray!20,
      draw=none,     
      minimum width=1cm,   
      minimum height=1cm,     
      anchor=south west,            
    }    
  } 
      
  \fill[graycell] (2,0) rectangle (5,5);
  \fill[graycell] (0,2) rectangle (5,5);
\draw (0,5)--(0,2)--(5,2);
\draw (2,5)--(2,0)--(5,0);
\coordinate[label=left:{\small $x$}] (x) at (0,2);
\coordinate[label=left:{\small $a$}] (a) at (1.3,3);
\coordinate[label=right:{\small $b$}] (b) at (3,4);
\coordinate[label=below left:{\small $y$}] (y) at (2,0);
\coordinate[label=above left:{\small $c$}] (c) at (4,0.8);
\fill[] (x)circle(6pt);
\fill[] (a)circle(6pt);
\fill[] (b)circle(6pt);
\fill[] (y)circle(6pt);
\fill[] (c)circle(6pt);
\end{tikzpicture}

& 
\begin{tikzpicture}[scale=0.3] 
\tikzset{      
   grid/.style={        
      draw,        
      step=1cm,        
      gray!100,    
      thin,        
    },   
    cell/.style={      
      draw,      
      anchor=center,    
      text centered,    
     },    
    graycell/.style={   
      fill=gray!20,
      draw=none,     
      minimum width=1cm,   
      minimum height=1cm,     
      anchor=south west,            
    }    
  } 
      
  \fill[graycell] (2,0) rectangle (5,5);
  \fill[graycell] (0,2) rectangle (5,5);
\draw (0,5)--(0,2)--(5,2);
\draw (2,5)--(2,0)--(5,0);
\coordinate[label=left:{\small $x$}] (x) at (0,2);
\coordinate[label=left:{\small $a$}] (a) at (1.3,4);
\coordinate[label=right:{\small $b$}] (b) at (3,3);
\coordinate[label=below left:{\small $y$}] (y) at (2,0);
\coordinate[label=above left:{\small $c$}] (c) at (4,0.8);
\fill[] (x)circle(6pt);
\fill[] (a)circle(6pt);
\fill[] (b)circle(6pt);
\fill[] (y)circle(6pt);
\fill[] (c)circle(6pt);
\end{tikzpicture}

\end{tabular}
\end{center}

\vspace{-0.5cm}

\caption{Observations (ii) and (iii) in the proof of Theorem~\ref{9-box-rs}}\label{fig-4}
\end{figure}

\begin{itemize}
\item[(i)] No other occurrence of $p_1$ or $p_2$ can begin with $x$ (otherwise, $xab$ cannot be an occurrence of $p_1$ or $p_2$ because of some elements in $\pi$ in a shaded area). 

\item[(ii)] If $x\neq y$ and $yac$ is another occurrence of $p_1$ or $p_2$ then $b=c$ (which is illustrated schematically on the first line in Figure~\ref{fig-4}). Indeed, suppose $b\neq c$ and $y$ is to the right of $x$ (the case when $y$ is to the left of $x$ can be considered similarly). If $y>x$ then $xab$ cannot be an occurrence of $p_1$ or $p_2$ since in this case, the element $y$ is in a shaded area. Hence $y<x$, but then $yac$ is not an occurrence of $p_1$ or $p_2$ because the element $b$ is in a shaded area. This contradiction shows that $b=c$.

\item[(iii)] If $x> y$, $a\neq c$, and $ycb$ (resp., $ybc$) is an occurrence of $p_1$ or $p_2$ then the only two possibilities are the leftmost (resp., rightmost) two drawings on the second line in Figure~\ref{fig-4}. Indeed, element $a$ must be between elements   $x$ and $y$ or else $ycb$ is not an occurrence because of $a$. Also, we must have $x>c$ or else $xab$ is not an occurrence because of $c$. 
\end{itemize}

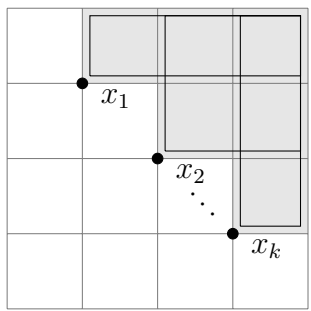
\begin{figure}  
\begin{center}  

\begin{tikzpicture} 
\tikzset{      
   grid/.style={        
      draw,        
      step=1cm,        
      gray!100,    
      thin,        
    },   
    cell/.style={      
      draw,      
      anchor=center,    
      text centered,    
     },    
    graycell/.style={   
      fill=gray!20,
      draw=none,     
      minimum width=1cm,   
      minimum height=1cm,     
      anchor=south west,            
    }    
  } 
      
  \fill[graycell] (1,3) rectangle (4,4);
  \fill[graycell] (2,2) rectangle (4,3);
  \fill[graycell] (3,1) rectangle (4,2);

 \draw[grid] (0,0) grid (4,4);

  \node[anchor=west] at (1.1,2.8) {$x_1$}; 
  \node[anchor=west] at (2.1,1.8) {$x_2$}; 
  \node[anchor=west] at (3.1,0.8) {$x_k$}; 
 
  \filldraw[black] (1,3) circle (2pt);  
  \filldraw[black] (2,2) circle (2pt); 
  \filldraw[black] (3,1) circle (2pt);  
 \draw (1.1,3.1)--(1.1,3.9)--(3.9,3.9)--(3.9,3.1)--(1.1,3.1);
 \draw (2.1,2.1)--(2.1,3.9)--(3.9,3.9)--(3.9,2.1)--(2.1,2.1);
\draw (3.1,1.1)--(3.1,3.9)--(3.9,3.9)--(3.9,1.1)--(3.1,1.1);

 \node[anchor=center, rotate=-45] at (2.6,1.4) {$\ldots$};
 
\end{tikzpicture} 

\vspace{-0.5cm}

\end{center}
\caption{The structure of permutations in the proof of Theorem~\ref{9-box-rs}}\label{fig-4-1}
\end{figure}

Assume that $\pi$ has $x$ occurrences of $p_1$ and $y$ occurrences of $p_2$.  Consider all occurrences of $p_1$ and $p_2$ in $\pi$ and let $x_1>x_2>\cdots>x_k$ be the first elements in these occurrences (note that $x+y=k$).  By observations (i)--(iii), to the right of each $x_i$ there are exactly two elements in $\pi$ larger than $x_i$. We subdivide $x_i$'s into blocks as follows: $x_i$ and $x_{i+1}$ belong to the same block if occurrences of $p_1$ or $p_2$ starting at $x_i$ and $x_{i+1}$ share at least one element. Let $X_i=\{x_{m_{i-1}+1},x_{m_{i-1}+2},\ldots,x_{m_{i}}\}$ be the $i$-th block, where $1\leq i\leq t$, $m_0=0$, $t\geq 1$, and $m_1<m_2<\cdots<m_t$. For example, for the permutation $\pi$ in Figure~\ref{fig-5-ex}, $t=2$, $m_1=6$, $m_2=2$, $X_1= \{6,9,11,13\}$, and $X_2=\{2\}$. 

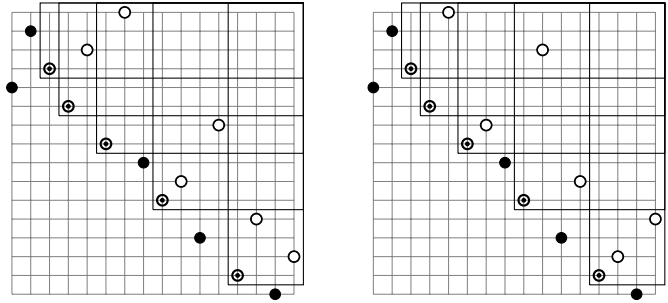
\begin{figure}
\begin{center}

\begin{tabular}{ccc}

\begin{tikzpicture}[scale=0.25] 
\tikzset{    
   grid/.style={      
      draw,      
      step=1cm,      
      gray!100,     
      very thin,      
    }, 
    cell/.style={    
      draw,    
      anchor=center,  
      text centered,    
     },  
    graycell/.style={ 
      fill=gray!40,   
      draw=none,   
      minimum width=1cm, 
      minimum height=1cm,   
      anchor=south west,   
    }
  }  
    
  \draw[grid] (0,0) grid (15,15);  
%
\filldraw[black] (0,11) circle (8pt);
  \filldraw[black] (1,14) circle (8pt); 
  \draw[thick] (2,12) circle (8pt); \draw[thick] (2,12) circle (2pt); 
  \draw[thick] (3,10) circle (8pt);  \draw[thick] (3,10) circle (2pt); 
  \draw[thick] (4,13) circle (8pt);  \filldraw[white] (4,13) circle (6pt);  
  \draw[thick] (5,8) circle (8pt); \draw[thick] (5,8) circle (2pt); 
  
  \filldraw[black] (7,7) circle (8pt);  
  \draw[thick] (8,5) circle (8pt); \draw[thick] (8,5) circle (2pt); 
  \draw[thick] (9,6) circle (8pt); \filldraw[white] (9,6) circle (6pt); 
  \filldraw[black] (10,3) circle (8pt);  
  \draw[thick] (11,9) circle (8pt);  \filldraw[white] (11,9) circle (6pt); 
  \draw[thick] (12,1) circle (8pt); \draw[thick] (12,1) circle (2pt); 
  \draw[thick] (13,4) circle (8pt); \filldraw[white] (13,4) circle (6pt); 
  \filldraw[black] (14,0) circle (8pt);

   \draw (1.5,11.5)--(1.5,15.5)--(15.5,15.5)--(15.5,11.5)--(1.5,11.5);  
   \draw (2.5,9.5)--(2.5,15.5)--(15.5,15.5)--(15.5,9.5)--(2.5,9.5); 
   \draw (4.5,7.5)--(4.5,15.5)--(15.5,15.5)--(15.5,7.5)--(4.5,7.5);
   \draw (7.5,4.5)--(7.5,15.5)--(15.5,15.5)--(15.5,4.5)--(7.5,4.5);
   \draw (11.5,0.5)--(11.5,15.5)--(15.5,15.5)--(15.5,0.5)--(11.5,0.5);
   \draw[thick] (6,15) circle (8pt); \filldraw[white] (6,15) circle (6pt); 
  \draw[thick] (15,2) circle (8pt); \filldraw[white] (15,2) circle (6pt); 
\end{tikzpicture} 

& \ & 
 \begin{tikzpicture}[scale=0.25] 
\tikzset{    
   grid/.style={      
      draw,      
      step=1cm,      
      gray!100,     
      very thin,      
    }, 
    cell/.style={    
      draw,    
      anchor=center,  
      text centered,    
     },  
    graycell/.style={ 
      fill=gray!40,   
      draw=none,   
      minimum width=1cm, 
      minimum height=1cm,   
      anchor=south west,   
    }
  }  
    
  \draw[grid] (0,0) grid (15,15);  
\filldraw[black] (0,11) circle (8pt);
  \filldraw[black] (1,14) circle (8pt); 
  \draw[thick] (2,12) circle (8pt); \draw[thick] (2,12) circle (2pt); 
  \draw[thick] (3,10) circle (8pt); \draw[thick] (3,10) circle (2pt); 
  
  \draw[thick] (5,8) circle (8pt); \draw[thick] (5,8) circle (2pt); 
  \draw[thick] (6,9) circle (8pt); \filldraw[white] (6,9) circle (6pt); 
  \filldraw[black] (7,7) circle (8pt);  
  \draw[thick] (8,5) circle (8pt); \draw[thick] (8,5) circle (2pt); 
  \draw[thick] (9,13) circle (8pt); \filldraw[white] (9,13) circle (6pt); 
  \filldraw[black] (10,3) circle (8pt);  
  \draw[thick] (11,6) circle (8pt);  \filldraw[white] (11,6) circle (6pt); 
  \draw[thick] (12,1) circle (8pt); \draw[thick] (12,1) circle (2pt); 
  \draw[thick] (13,2) circle (8pt); \filldraw[white] (13,2) circle (6pt); 
  \filldraw[black] (14,0) circle (8pt);

    \draw (1.5,11.5)--(1.5,15.5)--(15.5,15.5)--(15.5,11.5)--(1.5,11.5);  
   \draw (2.5,9.5)--(2.5,15.5)--(15.5,15.5)--(15.5,9.5)--(2.5,9.5); 
   \draw (4.5,7.5)--(4.5,15.5)--(15.5,15.5)--(15.5,7.5)--(4.5,7.5);
   \draw (7.5,4.5)--(7.5,15.5)--(15.5,15.5)--(15.5,4.5)--(7.5,4.5);
   \draw (11.5,0.5)--(11.5,15.5)--(15.5,15.5)--(15.5,0.5)--(11.5,0.5);
   \draw[thick] (4,15) circle (8pt);  \filldraw[white] (4,15) circle (6pt);  
  \draw[thick] (15,4) circle (8pt); \filldraw[white] (15,4) circle (6pt);

\end{tikzpicture}
\end{tabular}
\end{center}

\vspace{-0.5cm}

\caption{Permutations $\pi= (12)(15)(13)(11)(14)9(16)8674(10)2513$  (to the left) and $\pi'=(12)(15)(13)(11)(16)9(10)86(14)4(7)2315$ (to the right) illustrating the proof of Theorem~\ref{9-box-rs}. The circled dots represent the first elements of occurrences of $p_1$ or $p_2$ (the elements in $X_1$ and $X_2$), and the white dots represent $\pi_{i_1}\ldots\pi_{i_s}$ for $X_1$ and $X_2$.}\label{fig-5-ex}
\end{figure}

We will next describe a procedure of permuting the second and third elements, forming a subsequence $\pi_{i_1}\ldots\pi_{i_s}$ of $\pi$, in occurrences of $p_1$ and $p_2$ that begin with the elements in block $X_i$, $1\leq i\leq t$ (the subsequence $\pi_{i_1}\ldots\pi_{i_s}$ of $\pi$ is between $x_{m_{i-1}+1}$ and $x_{m_{i}+1}$ in $\pi$; $x_t+1$ is defined as the right end of $\pi$). All of the $\pi_i$'s are in the shaded area in the schematic representation of $\pi$ in Figure~\ref{fig-4-1}. The resulting permutation will be placed in $\pi$ in positions $i_1,\ldots,i_s$. We say that the procedure has been applied to block $X_i$. For example, for the permutation $\pi$ in Figure~\ref{fig-5-ex}, $\pi_{i_1}\ldots\pi_{i_s}$  for $X_1$ is $\pi_5\pi_7\pi_{10}\pi_{12}=(14)(16)7(10)$, and  for $X_2$, it is $\pi_{14}\pi_{16}=53$.  Applying the procedure to each block will result, in a bijective manner,  in a permutation $\pi'$ that has $y$ occurrences of $p_1$ and $x$ occurrences of $p_2$. 

Let $\pi^{(1)}_{i_1}\ldots\pi^{(1)}_{i_s}=\pi_{i_1}\ldots\pi_{i_s}$. First, suppose $\pi^{(1)}_{i_a}$ is the largest element among $\pi^{(1)}_{i_2}\ldots\pi^{(1)}_{i_s}$. Note that $x_i\pi^{(1)}_{i_1}\pi^{(1)}_{i_a}$ is an occurrence of $p_1$ or $p_2$ for $m_{i-1}+1\leq i\leq j_1\leq m_i$ for some $j_1$. Let $\pi^{(2)}_{i_1}\ldots\pi^{(2)}_{i_s}$ be the permutation obtained from $\pi^{(1)}_{i_1}\ldots\pi^{(1)}_{i_s}$ by swapping $\pi^{(1)}_{i_1}$ and $\pi^{(1)}_{i_a}$. Note that this move swaps occurrences of $p_1$ and $p_2$ in $\pi$  starting at $x_i$ for ${m}_{i-1}+1\leq i\leq j_1\leq m_i$ and does not affect any other occurrences of $p_1$ and $p_2$ (i.e., no new occurrences are introduced, and no other occurrences are changed). We let $\pi^{(2)}$ be the permutation obtained from $\pi$ by the swap. 

Next, suppose $\pi^{(2)}_{i_b}$ is the largest element among $\pi^{(2)}_{i_3}\ldots\pi^{(2)}_{i_s}$. Note that $x_i\pi^{(2)}_{i_2}\pi^{(2)}_{i_b}$ is an occurrence of $p_1$ or $p_2$ for $j_1+1\leq i\leq j_2\leq m_i$ for some $j_2$. Let $\pi^{(3)}_{i_1}\ldots\pi^{(3)}_{i_s}$ be the permutation obtained from $\pi^{(2)}_{i_1}\ldots\pi^{(2)}_{i_s}$ by swapping $\pi^{(2)}_{i_2}$ and $\pi^{(2)}_{i_b}$. Note that this move swaps occurrences of $p_1$ and $p_2$ in $\pi^{(2)}$  starting at $x_i$ for $j_1+1\leq i\leq j_2$ and does not affect any other occurrences of $p_1$ and $p_2$. We let $\pi^{(3)}$ be the permutation obtained from $\pi^{(2)}$ by the swap. 

And so on. Continuing in this way,  eventually we arrive to swapping the elements $\pi^{(s-1)}_{i_{s-1}}$ and $\pi^{(s-1)}_{i_{s}}$ and obtaining $\pi^{(s)}_{i_1}\ldots\pi^{(s)}_{i_s}$, which completes swapping  all occurrences of $p_1$ and $p_2$ in $\pi$ starting at $X_i$. Recall that $\pi'$ is obtained by applying the procedure to each block $X_i$, $1\leq i\leq t$.    

For example, for the permutation $\pi$ in Figure~\ref{fig-5-ex}, for $X_2$ we have a single swap of  the elements 3 and 5, while for $X_1$,

$\pi^{(1)}_{i_1}\pi^{(1)}_{i_2}\pi^{(1)}_{i_3}\pi^{(1)}_{i_4}=\pi^{(1)}_5\pi^{(1)}_7\pi^{(1)}_{10}\pi^{(1)}_{12}=(14)(16)7(10)$ $\xrightarrow{\text{swap 14 \& 16}}$

$\pi^{(2)}_{i_1}\pi^{(2)}_{i_2}\pi^{(2)}_{i_3}\pi^{(2)}_{i_4}=\pi^{(2)}_5\pi^{(2)}_7\pi^{(2)}_{10}\pi^{(2)}_{12}=(16)(14)7(10)$ $\xrightarrow{\text{swap 10 \& 14}}$

$\pi^{(3)}_{i_1}\pi^{(3)}_{i_2}\pi^{(3)}_{i_3}\pi^{(3)}_{i_4}=\pi^{(3)}_5\pi^{(3)}_7\pi^{(3)}_{10}\pi^{(3)}_{12}=(16)(10)7(14)$ $\xrightarrow{\text{swap 7 \& 14}}$

$\pi^{(4)}_{i_1}\pi^{(4)}_{i_2}\pi^{(4)}_{i_3}\pi^{(4)}_{i_4}=\pi^{(4)}_5\pi^{(4)}_7\pi^{(4)}_{10}\pi^{(4)}_{12}=(16)(10)(14)7$.

Note that the occurrence (13)(14)(16), (11)(14)(16) and 67(10) of  $p_1$ in $\pi$ are replaced, respectively, by the occurrence (13)(16)(14), (11)(16)(14), 6(14)(7) of $p_2$ in $\pi'$, while the occurrences 9(16)(10) and  253 of  $p_2$ in $\pi$ are replaced, respectively, by the occurrences 9(10)(14), 235 of $p_1$ in $\pi'$; no new occurrences of the patterns are introduced in~$\pi'$. 

\begin{remark}\label{moving-red-remark} Note that no element in $\pi_{i_1}\ldots\pi_{i_s}$ in $\pi$ will be in the same position in $\pi'$. Indeed, at each step in our procedure, an element $\pi_{i_j}$ either moves to the left and then moves nowhere else, or it moves to the right, and if it moves again, it moves to a position $>i_j$.  \end{remark}

It remains to explain why the procedure described above is a bijection. Clearly, the steps are reversible (instead of conducting swaps in $\pi_{i_1}\ldots\pi_{i_s}$ from left to right, we do it from right to left). Now, for injectivity, suppose $\pi=\pi_1\ldots\pi_n$ and $\sigma=\sigma_1\ldots\sigma_n$ are two different permutations, so that $\pi_i\neq\sigma_i$ for some $i$, and we assume $i$ is the smallest such index,  but $\pi'=\sigma'$, so that $\pi'_i=\sigma'_i$. The procedure only changes the places of elements in $\pi_{i_1}\ldots\pi_{i_s}$ and $\sigma_{k_1}\ldots\sigma_{k_{\ell}}$.  Hence, if $\pi_i$ is not in $\pi_{i_1}\ldots\pi_{i_s}$ or $\sigma_i$ is not in $\sigma_{k_1}\ldots\sigma_{k_{\ell}}$ then $\pi'_i=\sigma'_i\in\{\pi_i,\sigma_i\}$. It is impossible because either $\pi_i$ and $\sigma_i$ are not changed in their original position, or one of them would be involved as a second or third element in an occurrence of $p_1$ or $p_2$ in $\pi'=\sigma'$, while the other one would not. 

Therefore, the remaining case to consider is when $\pi_i$ is in $\pi_{i_1}\ldots\pi_{i_s}$ and $\sigma_i$ is in $\sigma_{k_1}\ldots\sigma_{k_{\ell}}$. Because $\pi'=\sigma'$, $\{i_1,\ldots,i_s\}=\{k_1,\ldots,k_{\ell}\}$, in particular, $s=\ell$, and  $\pi_{i_1}\ldots\pi_{i_s}$ and  $\sigma_{i_1}\ldots\sigma_{i_s}$ correspond to the same block $X_i$. Let $x=\pi_i$, $y=\sigma_i$, $z=\pi'_i=\sigma'_i$, and assume, w.l.o.g., that $y>x$. By Remark~\ref{moving-red-remark}, $z\not\in\{x,y\}$. Because we assumed that $\pi_j=\sigma_j$ for $j<i$, we have that $y$  is to the right of $x$ in $\pi$, and $x$  is to the right of $y$ in $\sigma$. One can see that $z$ must be in position $<i$ because otherwise it cannot be in position $i$ in $\pi'=\sigma'$. Clearly, $z>x$, and we consider two subcases. 
\begin{itemize}
\item Suppose $x<z<y$. When implementing the process in $\sigma$, $z$ and $y$ will be swapped, but $z$ will then move to the right again because $x$ is still to the right of $z$ (and maybe some other elements in $\sigma_{i_1}\ldots\sigma_{i_s}$). This contradicts to $z=\sigma'_i$.
\item Suppose $z>y$. Then $z=\pi_{i_1}=\sigma_{i_1}$ since otherwise $z$ must be swapped with an element to its left. And because $y>x$ in $\pi$, $xy$ must be the second and third elements in an occurrence of $p_1$ in $\pi$, which implies that $yx$ must be the second and third elements in an occurrence of  $p_2$ in $\sigma$  because all elements to the left of $y$ are equal. Therefore, $y$ must be the largest element in $\sigma_i\ldots \sigma_{i_s}$ because of the shaded area in the definition of an occurrence of the mesh pattern $p_2$. Consequently, because $\sigma_i=z$, $z$ will then move to the right again because $x$ is still to the right of $z$ (and maybe some other elements in $\sigma_{i_1}\ldots\sigma_{i_s}$). This contradicts to $z=\sigma'_i$.
\end{itemize}
We showed that the procedure is bijective, which completes our proof of the theorem.
\end{proof}

The following theorem is also true for any shading of the boxes (0,0), (0,1), (0,2), (1,0), and (2,0), whether symmetric or not, provided that the boxes (3,0) and (0,3) are unshaded and the other boxes are shaded. Note that this theorem offers alternative proofs for pairs 6 and 8. If the boxes (3,0) and (0,3) are shaded, along with the boxed shaded in Theorem~\ref{9-box-rs}, then only half of the patterns with additional symmetric shading of boxes are equidistributed, and they are pairs 1, 3, 11, and 45.

\begin{thm}\label{thm-square-shape-1}
We have $\pattern{scale = 0.5}{3}{1/1,2/2,3/3}{0/0,1/1, 1/2,1/3,2/1,2/2,2/3,3/1,3/2,3/3}\sim_d\hspace{-0.15cm} \pattern{scale = 0.5}{3}{1/1,2/3,3/2}{0/0,1/1, 1/2,1/3,2/1,2/2,2/3,3/1,3/2,3/3}$, 
$\pattern{scale = 0.5}{3}{1/1,2/2,3/3}{0/1,1/0,1/1, 1/2,1/3,2/1,2/2,2/3,3/1,3/2,3/3}\sim_d\hspace{-0.15cm} \pattern{scale = 0.5}{3}{1/1,2/3,3/2}{0/1,1/0,1/1, 1/2,1/3,2/1,2/2,2/3,3/1,3/2,3/3}$, 
$\pattern{scale = 0.5}{3}{1/1,2/2,3/3}{0/0,0/1,1/0,1/1, 1/2,1/3,2/1,2/2,2/3,3/1,3/2,3/3}\sim_d\hspace{-0.15cm} \pattern{scale = 0.5}{3}{1/1,2/3,3/2}{0/0,0/1,1/0,1/1, 1/2,1/3,2/1,2/2,2/3,3/1,3/2,3/3}$, \\
$\pattern{scale = 0.5}{3}{1/1,2/2,3/3}{0/2,1/1, 1/2,1/3,2/0,2/1,2/2,2/3,3/1,3/2,3/3}\sim_d\hspace{-0.15cm} \pattern{scale = 0.5}{3}{1/1,2/3,3/2}{0/2,1/1, 1/2,1/3,2/0,2/1,2/2,2/3,3/1,3/2,3/3}$, 
$\pattern{scale = 0.5}{3}{1/1,2/2,3/3}{0/0,0/2,1/1, 1/2,1/3,2/0,2/1,2/2,2/3,3/1,3/2,3/3}\sim_d\hspace{-0.15cm} \pattern{scale = 0.5}{3}{1/1,2/3,3/2}{0/0,0/2,1/1, 1/2,1/3,2/0,2/1,2/2,2/3,3/1,3/2,3/3}$,
$\pattern{scale = 0.5}{3}{1/1,2/2,3/3}{0/1,0/2,1/0,1/1, 1/2,1/3,2/0,2/1,2/2,2/3,3/1,3/2,3/3}\sim_d\hspace{-0.15cm} \pattern{scale = 0.5}{3}{1/1,2/3,3/2}{0/1,0/2,1/0,1/1, 1/2,1/3,2/0,2/1,2/2,2/3,3/1,3/2,3/3}$, and $\pattern{scale = 0.5}{3}{1/1,2/2,3/3}{0/0,0/1,0/2,1/0,1/1, 1/2,1/3,2/0,2/1,2/2,2/3,3/1,3/2,3/3}\sim_d\hspace{-0.15cm} \pattern{scale = 0.5}{3}{1/1,2/3,3/2}{0/0,0/1,0/2,1/0,1/1, 1/2,1/3,2/0,2/1,2/2,2/3,3/1,3/2,3/3}$.
\end{thm}

\begin{proof} Our arguments in the proof of Theorem~\ref{9-box-rs} are independent from any (symmetric or not) shading of  (0,0), (0,1), (0,2), (1,0), and (2,0).
\end{proof}

\begin{remark} Note that the arguments in the proof of Theorem~\ref{thm-square-shape-1}  do not hold if we shade the boxes $(0,3)$ and $(3,0)$. For example, consider the pair $\pattern{scale = 0.5}{3}{1/1,2/2,3/3}{0/3,1/1, 1/2,1/3,2/1,2/2,2/3,3/0,3/1,3/2,3/3}\pattern{scale = 0.5}{3}{1/1,2/3,3/2}{0/3,1/1, 1/2,1/3,2/1,2/2,2/3,3/0,3/1,3/2,3/3}$, for which our experiments show that $\pattern{scale = 0.5}{3}{1/1,2/2,3/3}{0/3,1/1, 1/2,1/3,2/1,2/2,2/3,3/0,3/1,3/2,3/3}\not\sim_d\hspace{-0.15cm} \pattern{scale = 0.5}{3}{1/1,2/3,3/2}{0/3,1/1, 1/2,1/3,2/1,2/2,2/3,3/0,3/1,3/2,3/3}$. Then, for instance, consider the permutation $34125$ with two occurrence of $\pattern{scale = 0.5}{3}{1/1,2/2,3/3}{0/3,1/1, 1/2,1/3,2/1,2/2,2/3,3/0,3/1,3/2,3/3}$ and no occurrences of $\pattern{scale = 0.5}{3}{1/1,2/3,3/2}{0/3,1/1, 1/2,1/3,2/1,2/2,2/3,3/0,3/1,3/2,3/3}$. Moving the elements results in the permutation $35142$ with no occurrences of either of the patterns. In fact, even the first step of moving elements (resulting in the permutation $35124$) already fails, as two occurrences of $\pattern{scale = 0.5}{3}{1/1,2/2,3/3}{0/3,1/1, 1/2,1/3,2/1,2/2,2/3,3/0,3/1,3/2,3/3}$ are replaced by a single occurrence of $\pattern{scale = 0.5}{3}{1/1,2/3,3/2}{0/3,1/1, 1/2,1/3,2/1,2/2,2/3,3/0,3/1,3/2,3/3}$. \end{remark}

\begin{thm}\label{an-equidistr-thm}
The patterns in the sets $\left\{\hspace{-1mm}\pattern{scale=0.5}{3}{1/1,2/2,3/3}{0/0,0/1,0/2,1/1, 1/2,2/1,2/2,1/0,2/0},\hspace{-1mm}\pattern{scale=0.5}{3}{1/1,2/3,3/2}{0/0,0/1,0/2,1/1, 1/2,2/1,2/2,1/0,2/0},\hspace{-1mm}\pattern{scale=0.5}{3}{1/1,2/2,3/3}{1/1, 1/2,1/3,2/1,2/2,2/3,3/1,3/2,3/3},\hspace{-1mm}\pattern{scale=0.5}{3}{1/1,2/3,3/2}{1/1, 1/2,1/3,2/1,2/2,2/3,3/1,3/2,3/3}  \right\}$ and

$\left\{\hspace{-1mm}\pattern{scale=0.5}{3}{1/1,2/2,3/3}{0/0,0/1,0/2,1/1, 1/2,2/1,2/2,1/0,2/0,3/3},\hspace{-1mm}\pattern{scale=0.5}{3}{1/1,2/3,3/2}{0/0,0/1,0/2,1/1, 1/2,2/1,2/2,1/0,2/0,3/3},\hspace{-1mm}\pattern{scale=0.5}{3}{1/1,2/2,3/3}{0/0,1/1, 1/2,1/3,2/1,2/2,2/3,3/1,3/2,3/3},\hspace{-1mm}\pattern{scale=0.5}{3}{1/1,2/3,3/2}{0/0,1/1, 1/2,1/3,2/1,2/2,2/3,3/1,3/2,3/3}  \right\}$ have the same distribution.
\end{thm}
\begin{proof}

Applying the complement and reverse operations, we have: \\[-3mm]
\[
\pattern{scale=0.5}{3}{1/1,2/2,3/3}{0/0,0/1,0/2,1/1, 1/2,2/1,2/2,1/0,2/0} \xrightarrow{r} \pattern{scale=0.5}{3}{1/3,2/2,3/1}{1/1, 1/2,2/1,2/2,1/0,2/0,3/0,3/1,3/2} \xrightarrow{c} \pattern{scale=0.5}{3}{1/1,2/2,3/3}{1/1, 1/2,1/3,2/1,2/2,2/3,3/1,3/2,3/3}
\]
\[
\pattern{scale=0.5}{3}{1/1,2/2,3/3}{0/0,0/1,0/2,1/1, 1/2,2/1,2/2,1/0,2/0,3/3} \xrightarrow{r} \pattern{scale=0.5}{3}{1/3,2/2,3/1}{0/3,1/1, 1/2,2/1,2/2,1/0,2/0,3/0,3/1,3/2} \xrightarrow{c} \pattern{scale=0.5}{3}{1/1,2/2,3/3}{0/0,1/1, 1/2,1/3,2/1,2/2,2/3,3/1,3/2,3/3}
\]
Hence, $\pattern{scale=0.5}{3}{1/1,2/2,3/3}{0/0,0/1,0/2,1/1, 1/2,2/1,2/2,1/0,2/0} \sim_d \hspace{-0.15cm} \pattern{scale=0.5}{3}{1/1,2/2,3/3}{1/1, 1/2,1/3,2/1,2/2,2/3,3/1,3/2,3/3}$ and $\pattern{scale=0.5}{3}{1/1,2/2,3/3}{0/0,0/1,0/2,1/1, 1/2,2/1,2/2,1/0,2/0,3/3} \sim_d \hspace{-0.15cm} \pattern{scale=0.5}{3}{1/1,2/2,3/3}{0/0,1/1, 1/2,1/3,2/1,2/2,2/3,3/1,3/2,3/3}$. The result now follows from Theorems~\ref{9-ls}, \ref{9-2s}, \ref{9-box-rs} and~\ref{thm-square-shape-1}. 
\end{proof}

\begin{thm}\label{2x2-thm} We have $\pattern{scale = 0.5}{3}{1/1,2/2,3/3}{2/2,2/3, 3/2,3/3}\sim_d \hspace{-0.15cm}\pattern{scale = 0.5}{3}{1/1,2/3,3/2}{2/2,2/3, 3/2,3/3}$.
\end{thm}

\begin{proof}
The arguments in the proof of Theorem~\ref{9-box-rs} can be copied and pasted. Indeed, the first element $x_i$  in an occurrence of a pattern cannot be the second or third element in another occurrence of a pattern. The only difference now is that each $x_i$ can have more than two elements to its right that are larger than $x_i$, and also we can have occurrences of patterns $x_iab$ and $x_jab$ such that $x_i<x_j$ and $i<j$. This does not change the procedure for permuting the second and third elements in occurrences of the patterns, although at each step we might be swapping multiple occurrences of the patterns, not necessarily just one.
\end{proof}

The following theorem is also true for any shading of the boxes (0,0), (0,1), (1,0) and (1,1), whether symmetric or not, provided the boxes (2,2), (2,3), (3,2), and (3,3) are shaded.

\begin{thm}\label{2x2-thm-cor}  We have
$\pattern{scale = 0.5}{3}{1/1,2/2,3/3}{0/0,2/2,2/3, 3/2,3/3}\sim_d\hspace{-0.15cm} \pattern{scale = 0.5}{3}{1/1,2/3,3/2}{0/0,2/2,2/3, 3/2,3/3}$,  	
$\pattern{scale = 0.5}{3}{1/1,2/2,3/3}{0/1,1/0,2/2,2/3, 3/2,3/3}\sim_d\hspace{-0.15cm} \pattern{scale = 0.5}{3}{1/1,2/3,3/2}{0/1,1/0,2/2,2/3, 3/2,3/3}$, 
$\pattern{scale = 0.5}{3}{1/1,2/2,3/3}{0/0,0/1,1/0,2/2,2/3, 3/2,3/3}\sim_d\hspace{-0.15cm} \pattern{scale = 0.5}{3}{1/1,2/3,3/2}{0/0,0/1,1/0,2/2,2/3, 3/2,3/3}$, \\
$\pattern{scale = 0.5}{3}{1/1,2/2,3/3}{1/1,2/2,2/3, 3/2,3/3}\sim_d\hspace{-0.15cm} \pattern{scale = 0.5}{3}{1/1,2/3,3/2}{1/1,2/2,2/3, 3/2,3/3}$, 
$\pattern{scale = 0.5}{3}{1/1,2/2,3/3}{0/0,1/1,2/2,2/3, 3/2,3/3}\sim_d\hspace{-0.15cm} \pattern{scale = 0.5}{3}{1/1,2/3,3/2}{0/0,1/1,2/2,2/3, 3/2,3/3}$, 
$\pattern{scale = 0.5}{3}{1/1,2/2,3/3}{0/1,1/0,1/1,2/2,2/3, 3/2,3/3}\sim_d\hspace{-0.15cm} \pattern{scale = 0.5}{3}{1/1,2/3,3/2}{0/1,1/0,1/1,2/2,2/3, 3/2,3/3}$, and 
$\pattern{scale = 0.5}{3}{1/1,2/2,3/3}{0/0,0/1,1/0,1/1,2/2,2/3, 3/2,3/3}\sim_d\hspace{-0.15cm} \pattern{scale = 0.5}{3}{1/1,2/3,3/2}{0/0,0/1,1/0,1/1,2/2,2/3, 3/2,3/3}$.
\end{thm}

\begin{proof} Our arguments in the proof of Theorem~\ref{2x2-thm} are independent of (symmetric) shadings of the boxes (0,0), (0,1), (1,0), and (1,1). This is because the first element $x_i$ in an occurrence of a pattern cannot be the second or third element in another occurrence, and the shaded areas do not affect the procedures described in Theorems~\ref{9-box-rs} and~\ref{2x2-thm}. Hence, the result follows. \end{proof}

\begin{thm}\label{2x2+2-thm} We have $\pattern{scale = 0.5}{3}{1/1,2/2,3/3}{1/2,1/3,2/2,2/3, 2/1,3/1,3/2,3/3}\sim_d \hspace{-0.15cm}\pattern{scale = 0.5}{3}{1/1,2/3,3/2}{1/2,1/3,2/2,2/3, 2/1,3/1,3/2,3/3}$.
\end{thm}

\begin{proof} The arguments in the proof of Theorem~\ref{2x2-thm} can be copied and pasted. The only difference is that if $x_iab$ is an occurrence of a pattern in
 question then there are no elements greater than $\min\{a,b\}$ in positions between $x_i$ and $a$, and there are no elements to the right of $a$, different from $b$, that are larger than $x_i$. This does not affect the procedure of moving the elements.   \end{proof}

\begin{remark} Note that the arguments in the proof of Theorem~\ref{2x2+2-thm}  do not work for the pair $\pattern{scale = 0.5}{3}{1/1,2/2,3/3}{1/3,2/2,2/3, 3/1,3/2,3/3}\pattern{scale = 0.5}{3}{1/1,2/3,3/2}{1/3,2/2,2/3,3/1,3/2,3/3}$, for which our experiments show that $\pattern{scale = 0.5}{3}{1/1,2/2,3/3}{1/3,2/2,2/3, 3/1,3/2,3/3}\not\sim_d \hspace{-0.15cm}\pattern{scale = 0.5}{3}{1/1,2/3,3/2}{1/3,2/2,2/3,3/1,3/2,3/3}$. For example, consider the permutation $1324$ with two occurrence of $\pattern{scale = 0.5}{3}{1/1,2/2,3/3}{1/3,2/2,2/3, 3/1,3/2,3/3}$ and no occurrences of $\pattern{scale = 0.5}{3}{1/1,2/3,3/2}{1/3,2/2,2/3,3/1,3/2,3/3}$. Moving the elements results in the permutation $1432$ with no occurrences of either of the patterns. In fact, even the first step of moving elements (resulting in the permutation $1423$) already fails, as two occurrences of $\pattern{scale = 0.5}{3}{1/1,2/2,3/3}{1/3,2/2,2/3, 3/1,3/2,3/3}$ are replaced by a single occurrence of $\pattern{scale = 0.5}{3}{1/1,2/3,3/2}{1/3,2/2,2/3, 3/1,3/2,3/3}$. \end{remark}

The following theorem is also true for any shading of the boxes (0,0), (0,1), (0,2), (1,0) and (2,0), whether symmetric or not, provided the shading in Theorem~\ref{2x2+2-thm}.

\begin{thm}\label{2x2+2-thm-cor}  We have
$\pattern{scale = 0.5}{3}{1/1,2/2,3/3}{0/0,1/2,1/3,2/2,2/3, 2/1,3/1,3/2,3/3}\sim_d\hspace{-0.15cm} \pattern{scale = 0.5}{3}{1/1,2/3,3/2}{0/0,1/2,1/3,2/2,2/3, 2/1,3/1,3/2,3/3}$,  	
$\pattern{scale = 0.5}{3}{1/1,2/2,3/3}{0/1,1/0,1/2,1/3,2/2,2/3, 2/1,3/1,3/2,3/3}\sim_d\hspace{-0.15cm} \pattern{scale = 0.5}{3}{1/1,2/3,3/2}{0/1,1/0,1/2,1/3,2/2,2/3, 2/1,3/1,3/2,3/3}$, 
$\pattern{scale = 0.5}{3}{1/1,2/2,3/3}{0/0,0/1,1/0,1/2,1/3,2/2,2/3, 2/1,3/1,3/2,3/3}\sim_d\hspace{-0.15cm} \pattern{scale = 0.5}{3}{1/1,2/3,3/2}{0/0,0/1,1/0,1/2,1/3,2/2,2/3, 2/1,3/1,3/2,3/3}$, \\
$\pattern{scale = 0.5}{3}{1/1,2/2,3/3}{0/2,1/2,1/3,2/0,2/2,2/3, 2/1,3/1,3/2,3/3}\sim_d\hspace{-0.15cm} \pattern{scale = 0.5}{3}{1/1,2/3,3/2}{0/2,1/2,1/3,2/0,2/2,2/3, 2/1,3/1,3/2,3/3}$, 
$\pattern{scale = 0.5}{3}{1/1,2/2,3/3}{0/0,0/2,1/2,1/3,2/0,2/2,2/3, 2/1,3/1,3/2,3/3}\sim_d\hspace{-0.15cm} \pattern{scale = 0.5}{3}{1/1,2/3,3/2}{0/0,0/2,1/2,1/3,2/0,2/2,2/3, 2/1,3/1,3/2,3/3}$, 
$\pattern{scale = 0.5}{3}{1/1,2/2,3/3}{0/1,0/2,1/0,1/2,1/3,2/2,2/3, 2/0,2/1,3/1,3/2,3/3}\sim_d\hspace{-0.15cm} \pattern{scale = 0.5}{3}{1/1,2/3,3/2}{0/1,0/2,1/0,1/2,1/3,2/0,2/2,2/3, 2/1,3/1,3/2,3/3}$, and 
$\pattern{scale = 0.5}{3}{1/1,2/2,3/3}{0/0,0/1,0/2,1/0,1/2,1/3,2/0,2/2,2/3, 2/1,3/1,3/2,3/3}\sim_d\hspace{-0.15cm} \pattern{scale = 0.5}{3}{1/1,2/3,3/2}{0/0,0/1,0/2,1/0,1/2,1/3,2/0,2/2,2/3, 2/1,3/1,3/2,3/3}$.
\end{thm}

\begin{proof} Our arguments in the proof of Theorem~\ref{2x2+2-thm} are independent of (symmetric) shadings of the boxes (0,0), (0,1), (0,2), (1,0) and (2,0). Hence, the result follows. \end{proof}

\begin{thm}\label{2x2+2+2-thm} We have $\pattern{scale = 0.5}{3}{1/1,2/2,3/3}{0/2,0/3,2/2,2/3, 2/0,3/0,3/2,3/3}\sim_d \hspace{-0.15cm}\pattern{scale = 0.5}{3}{1/1,2/3,3/2}{0/2,0/3,2/2,2/3, 2/0,3/0,3/2,3/3}$.
\end{thm}

\begin{proof} Our procedure for moving elements in the proof of Theorem~\ref{2x2-thm} is not affected by shading the boxes (0,2), (0,3), (2,0), and (3,0).  \end{proof}

The following theorem is also true for any shading of the boxes (0,0), (0,1), (1,0) and (1,1), whether symmetric or not, provided the shading in Theorem~\ref{2x2+2+2-thm}.

\begin{thm}\label{2x2+2+2-thm-cor}  We have
$\pattern{scale = 0.5}{3}{1/1,2/2,3/3}{0/0,0/2,0/3,2/0,2/2,2/3, 3/0,3/2,3/3}\sim_d\hspace{-0.15cm} \pattern{scale = 0.5}{3}{1/1,2/3,3/2}{0/0,0/2,0/3,2/0,2/2,2/3, 3/0,3/2,3/3}$,  	
$\pattern{scale = 0.5}{3}{1/1,2/2,3/3}{0/1,0/2,0/3,1/0,2/0,2/2,2/3, 3/0,3/2,3/3}\sim_d\hspace{-0.15cm} \pattern{scale = 0.5}{3}{1/1,2/3,3/2}{0/1,0/2,0/3,1/0,2/0,2/2,2/3, 3/0,3/2,3/3}$, 
$\pattern{scale = 0.5}{3}{1/1,2/2,3/3}{0/0,0/1,0/2,0/3,1/0,2/0,2/2,2/3, 3/0,3/2,3/3}\sim_d\hspace{-0.15cm} \pattern{scale = 0.5}{3}{1/1,2/3,3/2}{0/0,0/1,0/2,0/3,1/0,2/0,2/2,2/3, 3/0,3/2,3/3}$, \\
$\pattern{scale = 0.5}{3}{1/1,2/2,3/3}{0/2,0/3,1/1,2/0,2/2,2/3, 3/0,3/2,3/3}\sim_d\hspace{-0.15cm} \pattern{scale = 0.5}{3}{1/1,2/3,3/2}{0/2,0/3,1/1,2/0,2/2,2/3, 3/0,3/2,3/3}$, 
$\pattern{scale = 0.5}{3}{1/1,2/2,3/3}{0/0,0/2,0/3,1/1,2/0,2/2,2/3, 3/0,3/2,3/3}\sim_d\hspace{-0.15cm} \pattern{scale = 0.5}{3}{1/1,2/3,3/2}{0/0,0/2,0/3,1/1,2/0,2/2,2/3, 3/0,3/2,3/3}$, 
$\pattern{scale = 0.5}{3}{1/1,2/2,3/3}{0/1,0/2,0/3,1/0,1/1,2/0,2/2,2/3, 3/0,3/2,3/3}\sim_d\hspace{-0.15cm} \pattern{scale = 0.5}{3}{1/1,2/3,3/2}{0/1,0/2,0/3,1/0,1/1,2/0,2/2,2/3, 3/0,3/2,3/3}$, and 
$\pattern{scale = 0.5}{3}{1/1,2/2,3/3}{0/0,0/1,0/2,0/3,1/0,1/1,2/0,2/2,2/3, 3/0,3/2,3/3}\sim_d\hspace{-0.15cm} \pattern{scale = 0.5}{3}{1/1,2/3,3/2}{0/0,0/1,0/2,0/3,1/0,1/1,2/0,2/2,2/3, 3/0,3/2,3/3}$.
\end{thm}

\begin{proof} Our arguments in the proof of Theorem~\ref{2x2+2+2-thm} are independent of (symmetric) shadings of the boxes (0,0), (0,1), (1,0), and (1,1). We are done. \end{proof}

We conclude this section by proving equidistribution for two additional pairs of patterns, this time using both Theorems~\ref{thm-L-shape} and~\ref{9-box-rs}.

\begin{thm}\label{L+square} We have $\pattern{scale = 0.5}{3}{1/1,2/2,3/3}{0/0,0/1,0/2,1/0,2/2,2/3, 2/0,3/2,3/3}\sim_d \hspace{-0.15cm}\pattern{scale = 0.5}{3}{1/1,2/3,3/2}{0/0,0/1,0/2,1/0,2/2,2/3, 2/0,3/2,3/3}$ and $\pattern{scale = 0.5}{3}{1/1,2/2,3/3}{0/0,0/1,0/2,1/0,1/1,2/2,2/3, 2/0,3/2,3/3}\sim_d \hspace{-0.15cm}\pattern{scale = 0.5}{3}{1/1,2/3,3/2}{0/0,0/1,0/2,1/0,1/1,2/2,2/3, 2/0,3/2,3/3}$.
\end{thm}

\begin{proof} Referring to the notation in the proof of Theorem~\ref{thm-L-shape}, each occurrence of any of the four patterns must contain the second and third elements within an $A_i$. Therefore, the procedure of permuting some elements, as explained in the proof of Theorem~\ref{9-box-rs}, will be applied independetly within each $A_i$.  All arguments will hold regardless of whether  the box (1,1) is shaded or not.  \end{proof}

\section{Concluding remarks}\label{concluding-sec}
 
In this paper, we proved 75 equidistributions for mesh patterns 123 and 132 out of a maximum of 93 potential equidistributions, where shading is symmetric relative to the anti-diagonal. The cases we did not solve are presented in Table~\ref{unsolved-table}. A few comments on these cases are in order.
 
\begin{table}[t]
  {
   \renewcommand{\arraystretch}{1.5}
 \begin{center} 
   \begin{tabular}{|c|c||c|c||c|c||c|c|}
    \hline
    
  \footnotesize{nr.} & {\footnotesize patterns}  & \footnotesize{nr.} & {\footnotesize patterns}  & \footnotesize{nr.} & {\footnotesize patterns} & \footnotesize{nr.} & {\footnotesize patterns}   \\ 
    \hline
    \hline 
\footnotesize{76} & $\pattern{scale = 0.5}{3}{1/1,2/2,3/3}{0/2,1/2,2/0,2/1,2/2,3/3}\pattern{scale = 0.5}{3}{1/1,2/3,3/2}{0/2,1/2,2/0,2/1,2/2,3/3}$ &
\footnotesize{77} & $\pattern{scale = 0.5}{3}{1/1,2/2,3/3}{0/0,0/2,1/2,2/0,2/1,2/2,3/3}\pattern{scale = 0.5}{3}{1/1,2/3,3/2}{0/0,0/2,1/2,2/0,2/1,2/2,3/3}$  &

\footnotesize{78} &  $\pattern{scale = 0.5}{3}{1/1,2/2,3/3}{0/1,1/0,0/2,1/2,2/0,2/1,2/2,3/3}\pattern{scale = 0.5}{3}{1/1,2/3,3/2}{0/1,1/0,0/2,1/2,2/0,2/1,2/2,3/3}$&  
\footnotesize{79} &$\pattern{scale = 0.5}{3}{1/1,2/2,3/3}{0/0,0/1,1/0,0/2,1/2,2/0,2/1,2/2,3/3}\pattern{scale = 0.5}{3}{1/1,2/3,3/2}{0/0,0/1,1/0,0/2,1/2,2/0,2/1,2/2,3/3}$   \\[5pt]
\hline
\footnotesize{80} & $\pattern{scale = 0.5}{3}{1/1,2/2,3/3}{0/0,0/2,0/3,3/0,1/2,2/0,2/1,2/2,3/3}\pattern{scale = 0.5}{3}{1/1,2/3,3/2}{0/0,0/2,0/3,3/0,1/2,2/0,2/1,2/2,3/3}$ &
\footnotesize{81} & $\pattern{scale = 0.5}{3}{1/1,2/2,3/3}{0/0,0/1,1/0,0/2,0/3,3/0,1/2,2/0,2/1,2/2,3/3}\pattern{scale = 0.5}{3}{1/1,2/3,3/2}{0/0,0/1,1/0,0/2,0/3,3/0,1/2,2/0,2/1,2/2,3/3}$  & 
\footnotesize{82} &  $\pattern{scale = 0.5}{3}{1/1,2/2,3/3}{1/1,0/2,1/2,2/0,2/1,2/2,3/3}\pattern{scale = 0.5}{3}{1/1,2/3,3/2}{1/1,0/2,1/2,2/0,2/1,2/2,3/3}$  & 
\footnotesize{83}&$\pattern{scale = 0.5}{3}{1/1,2/2,3/3}{0/0,0/2,1/1,1/2,2/0,2/1,2/2}\pattern{scale = 0.5}{3}{1/1,2/3,3/2}{0/0,0/2,1/1,1/2,2/0,2/1,2/2}$  \\[5pt]
\hline
\footnotesize{84}&$\pattern{scale = 0.5}{3}{1/1,2/2,3/3}{0/0,0/2,1/1,1/2,2/0,2/1,2/2,3/3}\pattern{scale = 0.5}{3}{1/1,2/3,3/2}{0/0,0/2,1/1,1/2,2/0,2/1,2/2,3/3}$ & 
\footnotesize{85}& $\pattern{scale = 0.5}{3}{1/1,2/2,3/3}{0/1,1/0,0/2,1/1,1/2,2/0,2/1,2/2,3/3}\pattern{scale = 0.5}{3}{1/1,2/3,3/2}{0/1,1/0,0/2,1/1,1/2,2/0,2/1,2/2,3/3}$& 
\footnotesize{86}&$\pattern{scale = 0.5}{3}{1/1,2/2,3/3}{0/0,0/2,0/3,3/0,1/1,1/2,2/0,2/1,2/2}\pattern{scale = 0.5}{3}{1/1,2/3,3/2}{0/0,0/2,0/3,3/0,1/1,1/2,2/0,2/1,2/2}$  & 
\footnotesize{87}& $\pattern{scale = 0.5}{3}{1/1,2/2,3/3}{0/0,0/2,0/3,3/0,1/1,1/2,2/0,2/1,2/2,3/3}\pattern{scale = 0.5}{3}{1/1,2/3,3/2}{0/0,0/2,0/3,3/0,1/1,1/2,2/0,2/1,2/2,3/3}$ \\[5pt]
\hline
\footnotesize{88}&$\pattern{scale = 0.5}{3}{1/1,2/2,3/3}{0/0,0/1,1/0,1/1,2/2,2/3,3/2}\pattern{scale = 0.5}{3}{1/1,2/3,3/2}{0/0,0/1,1/0,1/1,2/2,2/3,3/2}$ & 
\footnotesize{89}&$\pattern{scale = 0.5}{3}{1/1,2/2,3/3}{0/0,0/1,1/0,1/1,0/2,2/0,2/2,2/3,3/2}\pattern{scale = 0.5}{3}{1/1,2/3,3/2}{0/0,0/1,1/0,1/1,0/2,2/0,2/2,2/3,3/2}$ &
\footnotesize{90}&$\pattern{scale = 0.5}{3}{1/1,2/2,3/3}{0/0,1/1,0/2,2/0,0/3,3/0,2/2,2/3,3/2}\pattern{scale = 0.5}{3}{1/1,2/3,3/2}{0/0,1/1,0/2,2/0,0/3,3/0,2/2,2/3,3/2}$ &
\footnotesize{91}&$\pattern{scale = 0.5}{3}{1/1,2/2,3/3}{0/0,0/1,1/0,1/1,0/2,2/0,0/3,3/0,2/2,2/3,3/2}\pattern{scale = 0.5}{3}{1/1,2/3,3/2}{0/0,0/1,1/0,1/1,0/2,2/0,0/3,3/0,2/2,2/3,3/2}$  \\[5pt]
\hline 
\footnotesize{92}&$\pattern{scale = 0.5}{3}{1/1,2/2,3/3}{0/0,0/1,1/0,1/2,2/1,2/2,2/3,3/2}\pattern{scale = 0.5}{3}{1/1,2/3,3/2}{0/0,0/1,1/0,1/2,2/1,2/2,2/3,3/2}$ &
\footnotesize{93}&$\pattern{scale = 0.5}{3}{1/1,2/2,3/3}{0/1,0/2,0/3,1/0,2/0,3/0,1/1,2/2,1/3,3/1,3/3}\pattern{scale = 0.5}{3}{1/1,2/3,3/2}{0/1,0/2,0/3,1/0,2/0,3/0,1/1,2/2,1/3,3/1,3/3}$  &&&&\\[5pt]
\hline
\end{tabular}
\end{center} 
}
\vspace{-0.5cm}
\caption{Conjectured equidistributions.}\label{unsolved-table}
\end{table}
 
The pairs of patterns $\pattern{scale = 0.5}{3}{1/1,2/2,3/3}{0/0,0/1,0/2,0/3,1/0,1/1,2/0,2/2,2/3,3/0,3/2}\pattern{scale = 0.5}{3}{1/1,2/3,3/2}{0/0,0/1,0/2,0/3,1/0,1/1,2/0,2/2,2/3,3/0,3/2}$	and $\pattern{scale = 0.5}{3}{1/1,2/2,3/3}{0/0,0/1,0/2,0/3,1/0,1/2,2/0,2/1,2/2,3/0,3/3}\pattern{scale = 0.5}{3}{1/1,2/3,3/2}{0/0,0/1,0/2,0/3,1/0,1/2,2/0,2/1,2/2,3/0,3/3}$	appear to have the same distributions. Proving  the equidistribution for the second pair is  equivalent to proving Kitaev and Zhang's conjecture \cite{KZ} on equidistribution of patterns nr.\ 57 and 58, respectively,  $\pattern{scale = 0.5}{2}{1/1,2/2}{0/1,1/1,1/2,2/0}$ and $\pattern{scale = 0.5}{2}{1/1,2/2}{0/1,1/0,1/1,2/2}$, which was reposted in \cite{HanZeng}. Indeed, the equivalence of the conjectures follows from reducing to the patterns of length 2  and applying the complement operation.

Also, we encountered issues while trying to explain the equidistributions for the pairs $\pattern{scale=0.5}{3}{1/1,2/2,3/3}{0/0,0/2,1/1, 1/2,2/0,2/1,2/2} \pattern{scale=0.5}{3}{1/1,2/3,3/2}{0/0,0/2,1/1, 1/2,2/0,2/1,2/2}$ and $\pattern{scale=0.5}{3}{1/1,2/2,3/3}{0/0,0/2,0/3,1/1, 1/2,2/0,2/1,2/2,3/0} \pattern{scale=0.5}{3}{1/1,2/3,3/2}{0/0,0/2,0/3,1/1,1/2,2/0,2/1,2/2,3/0}$ using our approaches. For the former pair, consider the permutation $\pi=265143$ for which $A_1$ is given by 6543, in the reduced form 4321. Under the bijection in Theorem~\ref{equdis-2-patterns-length-2}, 4321 (the only permutation of length 4 with the maximum number of occurrences of  $\pattern{scale = 0.5}{2}{1/2,2/1}{0/0,0/1,1/0,1/1}$) must go to 1432 (the only permutation of length 4 with the maximum number of occurrences of  $\pattern{scale = 0.5}{2}{1/1,2/2}{0/0,0/1,1/0,1/1}$), and hence, $\pi$ must go to $\pi'=236154$ if we want to apply the reduction to the smaller case for $A_1$. But $\pi$ has three occurrences of $\pattern{scale=0.5}{3}{1/1,2/3,3/2}{0/0,0/2,1/1, 1/2,2/0,2/1,2/2}$ (265, 243, 143), while $\pi'$ has just one occurrence of $\pattern{scale=0.5}{3}{1/1,2/2,3/3}{0/0,0/2,1/1, 1/2,2/0,2/1,2/2}$ (236), so our approach does not work in its present form.

Finally, note that the arguments used in the proof of Theorem~\ref{thm-L-shape-2} are not applicable for proving equidistribution  for the pairs $\pattern{scale=0.5}{3}{1/1,2/2,3/3}{0/0,0/2,0/3,1/1, 1/2,2/0,2/1,2/2,3/0} \pattern{scale=0.5}{3}{1/1,2/3,3/2}{0/0,0/2,0/3,1/1,1/2,2/0,2/1,2/2,3/0}$ and $\pattern{scale=0.5}{3}{1/1,2/2,3/3}{0/0,0/2,0/3,1/1, 1/2,2/0,2/1,2/2,3/0,3/3} \pattern{scale=0.5}{3}{1/1,2/3,3/2}{0/0,0/2,0/3,1/1,1/2,2/0,2/1,2/2,3/0,3/3}$ as Corollary~\ref{two-pat-length-2}  and Theorem~\ref{equdis-2-patterns-length-2} are not particularly useful in this context. Indeed, for the permutation $\pi=36178254$, we have $A=7854\rightarrow 7845\rightarrow 4875\rightarrow8475$, resulting in $\pi'=36184275$. While $\pi$ contains occurrences 178 and 354 of the patterns with either of the two shadings, $\pi'$ does not contain any occurrences of these patterns.

\begin{table}[t]
 	{
 		\renewcommand{\arraystretch}{1.5}
 \begin{center} 
 		\begin{tabular}{|c|c|c|c|c|c|}
 			\hline
$\pattern{scale = 0.5}{3}{1/1,2/2,3/3}{0/0,0/1,0/2,1/0,2/0,3/0,0/3,1/2,1/3,2/2,2/3}\pattern{scale = 0.5}{3}{1/1,2/3,3/2}{0/0,0/1,0/2,0/3,1/0,2/0,3/0,1/2,1/3,2/2,2/3}$ &
$\pattern{scale = 0.5}{3}{1/1,2/2,3/3}{0/0,0/1,0/2,1/0,2/0,3/0,0/3,3/1,1/2,1/3,2/2,2/3}\pattern{scale = 0.5}{3}{1/1,2/3,3/2}{0/0,0/1,0/2,0/3,1/0,2/0,3/0,3/1,1/2,1/3,2/2,2/3}$ &
$\pattern{scale = 0.5}{3}{1/1,2/2,3/3}{0/0,0/1,0/2,1/0,2/0,1/2,1/3,2/2,2/3}\pattern{scale = 0.5}{3}{1/1,2/3,3/2}{0/0,0/1,0/2,1/0,2/0,1/2,1/3,2/2,2/3}$ &
$\pattern{scale = 0.5}{3}{1/1,2/2,3/3}{0/0,0/1,0/2,1/0,2/0,3/1,1/2,1/3,2/2,2/3}\pattern{scale = 0.5}{3}{1/1,2/3,3/2}{0/0,0/1,0/2,1/0,2/0,3/1,1/2,1/3,2/2,2/3}$  &
$\pattern{scale = 0.5}{3}{1/1,2/2,3/3}{0/1,1/1,1/2,1/3,2/1,2/2,2/3,3/1,3/2,3/3}\pattern{scale = 0.5}{3}{1/1,2/3,3/2}{0/1,1/1,1/2,1/3,2/1,2/2,2/3,3/1,3/2,3/3}$ & 
$\pattern{scale = 0.5}{3}{1/1,2/2,3/3}{0/2,1/1,1/2,1/3,2/1,2/2,2/3,3/1,3/2,3/3}\pattern{scale = 0.5}{3}{1/1,2/3,3/2}{0/2,1/1,1/2,1/3,2/1,2/2,2/3,3/1,3/2,3/3}$ \\[5pt]
\hline
$\pattern{scale = 0.5}{3}{1/1,2/2,3/3}{0/0,0/1,1/1,1/2,2/1,2/2,2/3,1/3,3/1,3/2,3/3}\pattern{scale = 0.5}{3}{1/1,2/3,3/2}{0/0,0/1,1/1,1/2,2/1,2/2,2/3,1/3,3/1,3/2,3/3}$ & 
$\pattern{scale = 0.5}{3}{1/1,2/2,3/3}{0/0,0/2,1/1,1/2,1/3,2/1,2/2,2/3,3/1,3/2,3/3}\pattern{scale = 0.5}{3}{1/1,2/3,3/2}{0/0,0/2,1/1,1/2,1/3,2/1,2/2,2/3,3/1,3/2,3/3}$ &
$\pattern{scale = 0.5}{3}{1/1,2/2,3/3}{0/1,0/2,1/1,1/2,1/3,2/1,2/2,2/3,3/1,3/2,3/3}\pattern{scale = 0.5}{3}{1/1,2/3,3/2}{0/1,0/2,1/1,1/2,1/3,2/1,2/2,2/3,3/1,3/2,3/3}$ & 
$\pattern{scale = 0.5}{3}{1/1,2/2,3/3}{0/1,1/1,1/2,1/3,2/0,2/1,2/2,2/3,3/1,3/2,3/3}\pattern{scale = 0.5}{3}{1/1,2/3,3/2}{0/1,1/1,1/2,1/3,2/0,2/1,2/2,2/3,3/1,3/2,3/3}$ &
$\pattern{scale = 0.5}{3}{1/1,2/2,3/3}{0/0,0/1,1/1,1/2,1/3,2/0,2/1,2/2,2/3,3/1,3/2,3/3}\pattern{scale = 0.5}{3}{1/1,2/3,3/2}{0/0,0/1,1/1,1/2,1/3,2/0,2/1,2/2,2/3,3/1,3/2,3/3}$ & 
$\pattern{scale = 0.5}{3}{1/1,2/2,3/3}{0/1,1/0,0/2,1/1,1/2,1/3,2/1,2/2,2/3,3/1,3/2,3/3}\pattern{scale = 0.5}{3}{1/1,2/3,3/2}{0/1,0/2,1/0,1/1,1/2,1/3,2/1,2/2,2/3,3/1,3/2,3/3}$ \\[5pt]
\hline
$\pattern{scale = 0.5}{3}{1/1,2/2,3/3}{0/1,0/2,1/1,1/2,1/3,2/0,2/1,2/2,2/3,3/1,3/2,3/3}\pattern{scale = 0.5}{3}{1/1,2/3,3/2}{0/1,0/2,1/1,1/2,1/3,2/0,2/1,2/2,2/3,3/1,3/2,3/3}$ & 
$\pattern{scale = 0.5}{3}{1/1,2/2,3/3}{0/0,0/1,0/2,1/1,1/2,1/3,2/1,2/2,2/3,3/1,3/2,3/3}\pattern{scale = 0.5}{3}{1/1,2/3,3/2}{0/0,0/1,0/2,1/1,1/2,1/3,2/1,2/2,2/3,3/1,3/2,3/3}$ &
$\pattern{scale = 0.5}{3}{1/1,2/2,3/3}{0/0,0/1,1/0,1/1,1/2,1/3,0/2,,2/1,2/2,2/3,3/1,3/2,3/3}\pattern{scale = 0.5}{3}{1/1,2/3,3/2}{0/0,0/1,1/0,1/1,1/2,1/3,0/2,2/1,2/2,2/3,3/1,3/2,3/3}$ &
$\pattern{scale = 0.5}{3}{1/1,2/2,3/3}{0/0,0/2,0/1,1/1,1/2,1/3,2/0,2/1,2/2,2/3,3/1,3/2,3/3}\pattern{scale = 0.5}{3}{1/1,2/3,3/2}{0/0,0/2,0/1,1/1,1/2,1/3,2/0,2/1,2/2,2/3,3/1,3/2,3/3}$  &
$\pattern{scale = 0.5}{3}{1/1,2/2,3/3}{0/1,2/2,2/3,3/2,3/3}\pattern{scale = 0.5}{3}{1/1,2/3,3/2}{0/1,2/2,2/3,3/2,3/3}$ & 
$\pattern{scale = 0.5}{3}{1/1,2/2,3/3}{0/0,0/1,2/2,2/3,3/2,3/3}\pattern{scale = 0.5}{3}{1/1,2/3,3/2}{0/0,0/1,2/2,2/3,3/2,3/3}$ \\[5pt]
\hline  
$\pattern{scale = 0.5}{3}{1/1,2/2,3/3}{0/1,1/1,2/2,2/3,3/2,3/3}\pattern{scale = 0.5}{3}{1/1,2/3,3/2}{0/1,1/1,2/2,2/3,3/2,3/3}$ & 
$\pattern{scale = 0.5}{3}{1/1,2/2,3/3}{0/0,0/1,1/1,2/2,2/3,3/2,3/3}\pattern{scale = 0.5}{3}{1/1,2/3,3/2}{0/0,0/1,1/1,2/2,2/3,3/2,3/3}$ &
$\pattern{scale = 0.5}{3}{1/1,2/2,3/3}{0/1,1/2,1/3,2/1,2/2,2/3,3/1,3/2,3/3}\pattern{scale = 0.5}{3}{1/1,2/3,3/2}{0/1,1/2,1/3,2/1,2/2,2/3,3/1,3/2,3/3}$ & 
$\pattern{scale = 0.5}{3}{1/1,2/2,3/3}{0/2,1/2,1/3,2/1,2/2,2/3,3/1,3/2,3/3}\pattern{scale = 0.5}{3}{1/1,2/3,3/2}{0/2,1/2,1/3,2/1,2/2,2/3,3/1,3/2,3/3}$ & 
$\pattern{scale = 0.5}{3}{1/1,2/2,3/3}{0/0,0/1,1/2,1/3,2/1,2/2,2/3,3/1,3/2,3/3}\pattern{scale = 0.5}{3}{1/1,2/3,3/2}{0/0,0/1,1/2,1/3,2/1,2/2,2/3,3/1,3/2,3/3}$ & 
$\pattern{scale = 0.5}{3}{1/1,2/2,3/3}{0/0,0/2,1/2,1/3,2/1,2/2,2/3,3/1,3/2,3/3}\pattern{scale = 0.5}{3}{1/1,2/3,3/2}{0/0,0/2,1/2,1/3,2/1,2/2,2/3,3/1,3/2,3/3}$ \\[5pt]
\hline
$\pattern{scale = 0.5}{3}{1/1,2/2,3/3}{0/1,0/2,1/2,1/3,2/1,2/2,2/3,3/1,3/2,3/3}\pattern{scale = 0.5}{3}{1/1,2/3,3/2}{0/1,0/2,1/2,1/3,2/1,2/2,2/3,3/1,3/2,3/3}$ & 
$\pattern{scale = 0.5}{3}{1/1,2/2,3/3}{0/1,1/2,1/3,2/0,2/1,2/2,2/3,3/1,3/2,3/3}\pattern{scale = 0.5}{3}{1/1,2/3,3/2}{0/1,1/2,1/3,2/0,2/1,2/2,2/3,3/1,3/2,3/3}$ &
$\pattern{scale = 0.5}{3}{1/1,2/2,3/3}{0/0,0/1,1/2,1/3,2/0,2/1,2/2,2/3,3/1,3/2,3/3}\pattern{scale = 0.5}{3}{1/1,2/3,3/2}{0/0,0/1,1/2,1/3,2/0,2/1,2/2,2/3,3/1,3/2,3/3}$ & 
$\pattern{scale = 0.5}{3}{1/1,2/2,3/3}{0/1,1/0,1/2,1/3,2/0,2/1,2/2,2/3,3/1,3/2,3/3}\pattern{scale = 0.5}{3}{1/1,2/3,3/2}{0/1,1/0,1/2,1/3,2/0,2/1,2/2,2/3,3/1,3/2,3/3}$ &
$\pattern{scale = 0.5}{3}{1/1,2/2,3/3}{0/1,0/2,1/2,1/3,2/0,2/1,2/2,2/3,3/1,3/2,3/3}\pattern{scale = 0.5}{3}{1/1,2/3,3/2}{0/1,0/2,1/2,1/3,2/0,2/1,2/2,2/3,3/1,3/2,3/3}$ & 
$\pattern{scale = 0.5}{3}{1/1,2/2,3/3}{0/0,0/1,0/2,1/2,1/3,2/1,2/2,2/3,3/1,3/2,3/3}\pattern{scale = 0.5}{3}{1/1,2/3,3/2}{0/0,0/1,0/2,1/2,1/3,2/1,2/2,2/3,3/1,3/2,3/3}$ \\[5pt]
\hline
$\pattern{scale = 0.5}{3}{1/1,2/2,3/3}{0/0,0/1,1/0,1/2,1/3,2/0,2/1,2/2,2/3,3/1,3/2,3/3}\pattern{scale = 0.5}{3}{1/1,2/3,3/2}{0/0,0/1,1/0,1/2,1/3,2/0,2/1,2/2,2/3,3/1,3/2,3/3}$ &
$\pattern{scale = 0.5}{3}{1/1,2/2,3/3}{0/0,0/2,1/0,1/2,1/3,2/0,2/1,2/2,2/3,3/1,3/2,3/3}\pattern{scale = 0.5}{3}{1/1,2/3,3/2}{0/0,0/2,1/0,1/2,1/3,2/0,2/1,2/2,2/3,3/1,3/2,3/3}$ &
$\pattern{scale = 0.5}{3}{1/1,2/2,3/3}{0/1,0/2,0/3,2/0,2/2,2/3,3/0,3/2,3/3}\pattern{scale = 0.5}{3}{1/1,2/3,3/2}{0/1,0/2,0/3,2/0,2/2,2/3,3/0,3/2,3/3}$ & 
$\pattern{scale = 0.5}{3}{1/1,2/2,3/3}{0/0,0/1,0/2,0/3,2/0,2/2,2/3,3/0,3/2,3/3}\pattern{scale = 0.5}{3}{1/1,2/3,3/2}{0/0,0/1,0/2,0/3,2/0,2/2,2/3,3/0,3/2,3/3}$ & 
$\pattern{scale = 0.5}{3}{1/1,2/2,3/3}{0/1,0/2,0/3,1/1,2/0,2/2,2/3,3/0,3/2,3/3}\pattern{scale = 0.5}{3}{1/1,2/3,3/2}{0/1,0/2,0/3,1/1,2/0,2/2,2/3,3/0,3/2,3/3}$ & 
$\pattern{scale = 0.5}{3}{1/1,2/2,3/3}{0/0,0/1,0/2,0/3,1/1,2/0,2/2,2/3,3/0,3/2,3/3}\pattern{scale = 0.5}{3}{1/1,2/3,3/2}{0/0,0/1,0/2,0/3,1/1,2/0,2/2,2/3,3/0,3/2,3/3}$ \\[5pt]
 			\hline
	\end{tabular}
\end{center} 
}
\vspace{-0.5cm}
\caption{Non-symmetric equidistributions, up to reflection with respect to the anti-diagonal, obtained as a by-product in this paper. The first four equidistributions in the top row are obtained by applying the complement to the patterns in Corollary~\ref{two-pat-length-2} and Theorem~\ref{equdis-2-patterns-length-2}. The other equidistributions are obtained from the proofs of Theorems~\ref{9-box-rs},~\ref{2x2-thm},~\ref{2x2+2-thm} and~\ref{2x2+2+2-thm}.}\label{non-sym-table}
\end{table}

Even though our focus in this paper has been on symmetric shadings, it follows from the proofs of Theorems~\ref{9-box-rs} and~\ref{2x2-thm-cor}, where the symmetry of additional shadings is irrelevant, that the pairs of mesh patterns listed in Table~\ref{non-sym-table} are equidistributed. Interestingly,  the patterns in the set  $\left\{\hspace{-1mm}\pattern{scale = 0.5}{3}{1/1,2/2,3/3}{0/0,0/1,2/2,2/3,3/2,3/3},\hspace{-1mm}\pattern{scale = 0.5}{3}{1/1,2/3,3/2}{0/0,0/1,2/2,2/3,3/2,3/3},\hspace{-1mm}\pattern{scale = 0.5}{3}{1/1,2/2,3/3}{0/1,1/1,2/2,2/3,3/2,3/3},\hspace{-1mm}\pattern{scale = 0.5}{3}{1/1,2/3,3/2}{0/1,1/1,2/2,2/3,3/2,3/3}\right\}$ appear to have the same distribution. Why is this the case? \\

\noindent
{\bf Acknowledgement.}  The authors are grateful to Philip B. Zhang for sharing with us the results of computer experiments related to the equidistributions of mesh patterns in question.  \\

\end{document}